\documentclass[12pt,reqno]{amsart}
\usepackage{amsmath}
\usepackage{amssymb}
\usepackage{verbatim}
\usepackage{mathrsfs}
\usepackage{graphicx}
\usepackage[font=footnotesize]{caption}
\usepackage{subcaption}
\usepackage{tabularx}
\usepackage{longtable}
\usepackage{tikz}
\usepackage{pgfplots}
\usepackage{xcolor}
\usepackage{pgfplots}
\usepackage{pgfplotstable}
\usepackage{enumitem}
\allowdisplaybreaks

\usepackage[toc]{appendix}
\usepackage[capitalise]{cleveref}
\usepackage[noadjust]{cite}
\usepackage{color}

\usepackage[top=0.6in,bottom=0.5in,left=0.65in,right=0.65in]{geometry}

\newtheorem{lemma}{Lemma}
\newtheorem{prop}[lemma]{Proposition}

\newtheorem{theorem}[lemma]{Theorem}

\newtheorem{example}{Example}

\newtheorem{algorithm}{Algorithm}

\numberwithin{equation}{section}
\numberwithin{lemma}{section}

\newcommand{\bp}{ \begin{proof} }
\newcommand{\ep}{\hfill \end{proof} }
\newcommand{\be}{ \begin{equation} }
\newcommand{\ee}{ \end{equation} }
\newcommand{\C}{\mathbb{C}}    
\newcommand{\N}{\mathbb{N}}    
\newcommand{\NN}{\mathbb{N}_0} 
\newcommand{\dirac}{\delta}   
\newcommand{\supp}{\operatorname{supp}}
\newcommand{\bo}{\mathscr{O}} 
\newcommand{\gl}{\lambda}
\newcommand{\xb}{x_b} 
\newcommand{\eu}{u_e} 

\newcommand{\EE}{\mathcal{E}}
\newcommand{\EF}{\mathcal{F}}
\newcommand{\tp}{\mathsf{T}}  

\begin{document}

\title[Dirac Assisted Tree Method for Helmholtz Equations]{Dirac Assisted Tree Method for 1D Heterogeneous Helmholtz Equations with Arbitrary Variable  Wave Numbers}

\author{Bin Han, Michelle Michelle, and Yau Shu Wong}

\thanks{
Corresponding author: Bin Han (bhan@ualberta.ca).
Research supported in part by
Natural Sciences and Engineering Research Council (NSERC) of Canada under grant RGPIN-2019-04276, Alberta Innovates and Alberta Advanced Education, Westgrid (www.westgrid.ca), and Compute Canada Calcul Canada (www.computecanada.ca)}

\address{Department of Mathematical and Statistical Sciences,
University of Alberta, Edmonton,\quad Alberta, Canada T6G 2G1.
\quad {\tt bhan@ualberta.ca}\quad {\tt mmichell@ualberta.ca}
\quad{\tt yauwong@ualberta.ca}
}

\makeatletter \@addtoreset{equation}{section} \makeatother

\begin{abstract}
In this paper we introduce a new method called the Dirac Assisted Tree (DAT) method, which can handle 1D heterogeneous Helmholtz equations with arbitrarily large variable wave numbers. DAT breaks an original global problem into many parallel tree-structured small local problems, which are
linked together to form a global solution by solving small linking problems.
To solve the local problems in DAT, we propose a compact finite difference method (FDM) with arbitrarily high accuracy order and low numerical dispersion for piecewise smooth coefficients and variable wave numbers.
This compact FDM is particularly appealing for DAT, because the local problems and their fluxes in DAT can be computed with high accuracy.
DAT with such compact FDMs
can solve heterogeneous Helmholtz equations with arbitrarily large variable wave numbers accurately by solving small linear systems --- $4 \times 4$ matrices in the extreme case --- with tridiagonal coefficient matrices in a parallel fashion.
Several numerical examples are provided to illustrate the effectiveness
of DAT using the $M$th order compact FDMs with $M=6,8$ for numerically solving heterogeneous Helmholtz equations with variable wave numbers. We shall also discuss how to solve some special 2D Helmholtz equations using DAT.
\end{abstract}


\maketitle

\pagenumbering{arabic}

\section{Introduction and Motivations}
The heterogenous Helmholtz equation is given by
\be \label{heqn}
\nabla \cdot (a(x)\nabla u(x))+\kappa^{2}(x)u(x)=f(x), \qquad x\in \Omega
\ee
coupled with suitable boundary conditions on the boundary $\partial \Omega$ of the domain $\Omega$, where the source term $f\in L_2(\Omega)$  and the coefficients $a,\kappa\in L_\infty(\Omega)$
satisfying $\mbox{ess-inf}_{x\in \Omega} a(x)>0$.
In this paper, we mainly focus on one-dimensional (1D) heterogeneous Helmholtz equations with $\Omega=(0,1)$, and particular 2D Helmholtz equations which can be converted and solved through 1D Helmholtz equations.
In practical applications (e.g.,
geophysics \cite{C16, DL19, EOV06, FG17} and electromagnetics \cite{GS20,N01}), the coefficients $a$, the wave number $\kappa$ and the source term $f$ in \eqref{heqn} are often piecewise smooth functions.

The Helmholtz equation is known to be a challenging problem in computational mathematics due to the presence of the pollution effect caused by large wave numbers and the highly ill-conditioned coefficient matrix due to its discretization (e.g., see \cite{EG11}).
The former means that one cannot simply decrease the grid size in proportion to the growth in the wave number, but rather its reduction has to significantly outweigh the growth of the wave number to reduce the pollution effect (i.e., numerical dispersion). In the presence of a very large constant wave number, this means that standard Finite Element Method (FEM) and Finite Difference Method (FDM) inevitably yield a very large ill-conditioned coefficient matrix. More specifically, the authors in \cite{MS11} found that to properly handle a large constant wave number $\kappa$ in the $hp$-FEM setting, the polynomial degree $p$ and the mesh size $h$ need to be chosen such that $p \gtrsim \log(\kappa)$ and the quantity $\kappa h/p$ is small enough. A pre-asymptotic error analysis in the $hp$-FEM setting with the assumption that $\kappa h/p \lesssim (p/\kappa)^{1/(p+1)}$ was also done in \cite{ZW13}. On the other hand, it has been documented that for second-order and fourth-order FDM, the appropriate resolution conditions are respectively $\kappa^{3}h^{2} \lesssim 1$ and $\kappa^{5}h^{4} \lesssim 1$ \cite{CCFW13,DL19,WW14}. Generally speaking, for polynomial-based schemes like standard FEM and FDM, the higher the order a scheme has, the better it is in dealing with the pollution effect.

There is a vast array of methods developed over the years to deal with the demanding resolution condition on the mesh size $h$ discussed above. In the Galerkin setting, one can relax the inter-element continuity condition \cite{FW09}, impose a penalty on the normal jump of the derivatives of elements in trial and test spaces \cite{BWZ16, WW14}, or enrich the local approximation space with solutions of the homogeneous Helmholtz equation \cite{HMP16}. In the FDM setting, numerous schemes with various accuracy, dispersion orders and stencil sizes can be found in the literature \cite{BTT11,cgx18,CGX20,CCFW13, DL19, FLQ11, F12, F08, NSD07, SHC08, ST98, SFL16, TGGT13, WX18, ZWG19}. Furthermore, a pollution-free finite difference scheme for 1D Helmholtz equation with a constant wave number in homogeneous media is available \cite{WW14, WWH17, WL11}. A wave splitting method for 1D Helmholtz equation was analyzed in \cite{PR11}. Many preconditioners and domain decomposition methods have also been proposed for numerical solutions of Helmholtz equations \cite{GZ19}.

Most of the studies and papers cited above mainly deal with constant or smoothly varying wave numbers in homogeneous media. There are significantly fewer studies on the numerical and mathematical analysis of the heterogeneous Helmholtz equation (e.g., \cite{C16,cgx18,EOV06, FG17,GPS19,GS20} and references therein). In the context of FDM, the authors in \cite{BTT11} provided a fourth-order compact finite difference scheme under the assumptions that the coefficient $a$ and the wave number $\kappa$ are smooth functions. The heterogeneous Helmholtz equation is indeed even harder to solve numerically, because we need to capture extra fine scale features on top of the oscillations introduced by the large wave number if the coefficient $a(x)$ in \eqref{heqn} is rough.

Our goal of this paper is to develop
numerical schemes with high accuracy order for solving 1D heterogeneous Helmholtz equation with arbitrarily large variable wave numbers, but without solving enormous ill-conditioned linear systems of equations.
To achieve our goal, we shall present a new method called the Dirac Assisted Tree (DAT) method and develop compact finite difference schemes with arbitrarily high accuracy order.
The key idea of DAT is to
break the global problem on $[0,1]$ into many small local problems which can be effectively solved by any known methods in a parallel fashion. Such local solutions are then stitched together into a desired global solution through solving many small linking problems in a parallel fashion as well.
As a consequence, this leads to parallel systems of linear equations which have much smaller condition numbers and are significantly smaller in size.

More specifically, we multiply the source term $f$ with a partition of unity on $[0,1]$ to break the global source term $f$ into many highly localized source terms $f_j$ for $1\le j\le N_0$ such that $\sum_{j=1}^{N_0} f_j=f$.
For simplicity, we shall use hat linear functions (i.e., the B-spline of order $2$) to create a partition of unity.
Other choices like higher order B-splines or (piecewise) smooth functions are also permissible.
Then for each highly localized source term $f_j$, we solve a local 1D Helmholtz equation on the subinterval of the support of $f_j$.
With the exception of local problems that touch the boundaries $0$ or $1$, we impose homogeneous Dirichlet boundary conditions for all local problems.
Such local problems can be solved by any known discretization method, provided that the fluxes at the endpoints of local problems can be accurately computed (see \cref{sec:dat} for details).
In order to assemble all the local solutions into a global solution, we shall see in \cref{sec:dat} that the Dirac distribution will naturally appear in the fluxes of local solutions. Indeed using the Dirac distribution, we can naturally link all such local solutions together via DAT.
Due to the refinability of the hat function (\cite{hanbook}), the local problems can be further decomposed into sub-local problems. Applying this recursively, we obtain the tree structure of DAT. As we shall see in \cref{sec:dat},
DAT enjoys the following inherent advantages:
\begin{enumerate}[noitemsep,topsep=5pt]
\item[(1)] DAT breaks the global problem into tree-structured small local problems, which can be solved by any methods with reduced pollution effect. All such local solutions are then linked together to form a global solution by solving small linking problems.

\item[(2)] All involved local and linking problems in DAT can be solved in a parallel manner by solving small linear systems often with significantly smaller condition numbers. Extremely speaking, all local and linking problems in DAT can be solved by solving only at most $4$ linear equations.

\item[(3)] The accuracy of DAT only depends on the accuracy of the local problem solvers. If the maximum size of all local problems are bounded and independent of mesh sizes, then all the coefficient matrices for local problems have uniformly bounded condition numbers.

\item[(4)] DAT naturally brings about domain decomposition, which reduces the problem size, and adaptivity, which further improves the approximation. It is efficient due to its tree structure and tridiagonal coefficient matrices.
\item[(5)] Without imposing any magnitude constraints on variable wave numbers, DAT can solve heterogeneous Helmholtz equations with arbitrarily large variable wave numbers and oscillatory jumping coefficients, which other known methods have difficulties in handling.
\end{enumerate}
In order to solve the local problems in DAT and to reduce the numerical dispersion, we shall propose a compact finite difference method (FDM) with arbitrarily high accuracy order. Our second contribution is to rigorously prove that we can always find a 1D compact FDM (that handles both interior and any boundary conditions) with arbitrarily high accuracy orders by assuming that $a, \kappa^{2}, f$ are piecewise smooth.
In a much simpler setting with $a=1$ and $\kappa^{2}=0$, \cite[Theorem 3.1]{SDKS13} already showed that we can always find an arbitrarily accurate compact FDM for the 1D Poisson equation with Dirichlet boundary conditions. Our result generalizes this statement. It also unifies many 1D FDMs for the equations \eqref{heqn} including $\kappa^{2}(x)$ in \eqref{heqn} being replaced by $-\kappa^2(x)$ (i.e., an elliptic problem). Compact stencils are very much desirable, since they give rise to a tridiagonal matrix in the 1D setting. Without losing any accuracy order, our approach can indeed handle all Dirichlet, Neumann, and Robin boundary conditions. Such high order compact FDMs are particularly appealing for DAT, since the fluxes and derivatives of the local solutions can be computed accurately. This preserves the high accuracy of all our linking problems.
As a consequence, in the presence of a very large wave number, we avoid dealing with a very large ill-conditioned coefficient matrix altogether. Instead, we harness parallel computing resources to solve small and much better conditioned coefficient matrices. This makes DAT a very attractive alternative when dealing with heterogeneous Helmholtz equations.

Next, we discuss how our method differs from the partition of unity FEM (PUFEM) presented in \cite{MB96,bm97}. Refer to \cref{ex:constant} and \cref{table:vsPUFEM} in \cref{sec:examples} for a juxtaposition of the numerical performance between DAT, $6$th and $8$th order FDM, as well as PUFEM (see \cite[Section 3]{BS00} for special shape functions used in 1D PUFEM and the corresponding coefficient matrix). One stark difference is in how we apply the partition of unity. Recall that DAT applies the partition of unity to the source term $f$. By the definition of PUFEM, we multiply the partition of unity with local approximation spaces and use this as our test and trial spaces in FEM. The effectiveness of PUFEM hinges on how to find suitable local approximation spaces. For the case of $a(x)=1$ and a constant wave number, the local approximation spaces take the form of plane waves or generalized harmonic polynomials.
In the presence of a large wave number, the trial functions in PUFEM are highly oscillatory. Hence, selecting/developing an appropriate quadrature becomes a major concern and challenge for PUFEM. Moreover, numerical experiments in \cite{WTTF12} indicate that the coefficient matrix of PUFEM has an extremely large condition number, which may produce extra stability issues. For the heterogeneous Helmholtz equation with piecewise smooth coefficients and wave number, finding suitable local approximation spaces can in fact be challenging and computationally expensive. See \cite{FG17} for some work along this direction. Due to our way of applying the partition of unity, the local problems allow us to focus on different sub-domains, but with the same original grid size. This is why as we further decompose the local problems, our coefficient matrices continue to reduce in size and often have uniformly bounded condition numbers.

Since DAT breaks a global problem into several local problems,
DAT may be regarded as a domain decomposition method with some differences.
In many domain decomposition methods, the source term is typically restricted to each subdomain. In DAT, we locally modify our source term by multiplying it with hat functions, which constitute our partition of unity. In addition, to obtain the global solution from the local solutions, we require the aid of the Dirac linking problems. Unlike some domain decomposition methods,
DAT does not require any initial guesses in its algorithm.

The organization of this paper is as follows. In \cref{sec:dat}, we present the DAT method. In \cref{sec:fdm}, to solve the local problems in DAT, we prove that we can always find a 1D compact FDM with arbitrarily high accuracy order for 1D heterogeneous Helmholtz equation in \eqref{heqn} with piecewise smooth coefficients, variable wave numbers, and source terms.
A few examples of such compact FDMs are also presented in this section.
In \cref{sec:convergence}, we discuss the convergence of DAT.
We apply DAT by employing the compact FDMs in \cref{sec:fdm} as local problem solvers to some 1D and 2D problems, and then evaluate its numerical performance in \cref{sec:examples}. Finally, we state our conclusions in \cref{sec:conclusions}.

\section{Dirac Assisted Tree (DAT) Method}
\label{sec:dat}

In this section we shall discuss the key ingredients and algorithm for the DAT method.
Let $\mathcal{L}$ be the associated linear differential operator of the 1D heterogeneous Helmholtz equations in \eqref{heqn}:
\be \label{heqnL}
\mathcal{L} u:=[a(x)u'(x)]'+\kappa^2(x) u(x)=f(x), \qquad x\in \Omega:=(0,1)
\ee
with any given linear boundary conditions
\be \label{heqnbdry}
\mathcal{B}_0 u(0):=\lambda_{0}^{L} u(0) + \lambda_{1}^{L} u'(0)=g_0, \qquad \mathcal{B}_1 u(0):= \lambda_{0}^{R} u(1) + \lambda_{1}^{R} u'(1) =g_1,
\ee
where $\lambda_{0}^{L}, \lambda_{1}^{L}, \lambda_{0}^{R}, \lambda_{1}^{R} \in \C$ satisfy $|\lambda_0^L|+|\lambda_1^L|\ne 0$ and  $|\lambda_0^R|+|\lambda_1^R|\ne 0$. I.e., $\mathcal{B}_0 u(0)$ and $\mathcal{B}_1u(1)$ can be Dirichlet, Neumann or Robin (e.g. Sommerfeld) boundary conditions.

Let $(\alpha,\beta)\subseteq (0,1)$
with $0\le \alpha<\beta\le 1$. Let $f\in (H^{1}(\alpha,\beta))'$ be a source term. Let $u_{loc}\in H^1(\alpha,\beta)$ be the weak solution in the Sobolev space $H^1(\alpha,\beta)$ to the following local problem:
\be \label{heqnL:local}
\mathcal{L} u_{loc}(x)=f(x), \qquad x\in (\alpha,\beta),
\ee
where
if $\alpha,\beta\in \{0,1\}$, then we preserve the boundary conditions as in
\eqref{heqnbdry}; otherwise, we use homogeneous Dirichlet boundary conditions. Putting these boundary conditions into a compact form, the boundary conditions to \eqref{heqnL:local} are given by
\be \label{helmholtz:bc:local}
(\mathcal{B}_0 u_{loc}(0)-g_0)\delta_{0,\alpha}+
(1-\delta_{0,\alpha})u_{loc}(\alpha)=0,\quad
(\mathcal{B}_1 u_{loc}(1)-g_1)\delta_{1,\beta}+
(1-\delta_{1,\beta})u_{loc}(\beta)=0,
\ee
where $\delta_{c,c}=1$ and $\delta_{c,d}=0$ for $c\ne d$.
Recall that $\psi\in H^1(\alpha,\beta)$ if $\psi\in L_2(\alpha,\beta)$ and its weak/distributional derivative $\psi'\in L_2(\alpha,\beta)$. Moreover, $\|\psi\|_{H^1(\alpha,\beta)}^2:=\|\psi\|_{L_2(\alpha,\beta)}^2+\|\psi'\|_{L_2(\alpha,\beta)}^2$, where $\psi'$ stands for the weak derivative of $\psi$.
Due to \eqref{helmholtz:bc:local}, we can extend $u_{loc}\in H^1(\alpha,\beta)$ as an element in $H^1(0,1)$ by zero extension, which is denoted by $\tilde{u}_{loc}$.
Therefore, using the definition of $\mathcal{L}$ in \eqref{heqnL}, we observe that
\be \label{heqnL:2}
\tilde{f}:=\mathcal{L} \tilde{u}_{loc}
=\begin{cases}
0, &\text{$x\in (0,\alpha)\cup (\beta,1)$},\\
d_\alpha(\tilde{u}_{loc})\delta_\alpha, &\text{$x=\alpha$ and $\alpha\ne 0$},\\
f(x), &\text{$x \in (\alpha,\beta)$},\\
d_\beta(\tilde{u}_{loc}) \delta_\beta, &\text{$x=\beta$ and $\beta\ne 1$},
\end{cases}
\ee
where $\delta_\alpha$ is the Dirac distribution at the point $\alpha$ and the above numbers
$d_\alpha(\tilde{u}_{loc}), d_\beta(\tilde{u}_{loc})\in \C$ for $\alpha\ne 0$ and $\beta\ne 1$ are given by
\be \label{dab}
d_\alpha(\tilde{u}_{loc}):=\lim_{x\to \alpha^+} a(x) \tilde{u}_{loc}'(x),\qquad
d_\beta(\tilde{u}_{loc}):=-\lim_{x\to \beta^-} a(x) \tilde{u}_{loc}'(x),
\ee
which, up to a sign change, are simply the fluxes of $\tilde{u}_{loc}$ at $\alpha$ and $\beta$. Then  $\tilde{u}_{loc}$ is a global solution of
\[
\mathcal{L} \tilde{u}_{loc}(x)=\tilde{f}(x),\qquad x\in \Omega=(0,1).
\]

We now introduce the DAT method. Let ${N_0}\in \N$ be a positive integer greater than one. We take a partition of unity $\{\varphi_j\}_{j=0}^{N_0}$ of  piecewise smooth functions such that each function $\varphi_j$ is supported on $[0,1]$ and $\sum_{j=0}^{N_0}\varphi_j(x)=1$ for all $x\in (0,1)$. For simplicity, we use piecewise linear hat functions $\varphi_j$. Let $0=x_0<\cdots<x_{N_0}=1$ be a partition of $[0,1]$.
For simplicity, we define $x_{-1}=0$ and $x_{{N_0}+1}:=1$.
We let $\varphi_j$ be the linear hat function supported on $[x_{j-1},x_{j+1}]$ with
$\varphi_j(x_j)=1$ and $\varphi_j(x_{j-1})=\varphi_j(x_{j+1})=0$.
Obviously, we define $\varphi_0(x_0)=1$ and $\varphi_0(x_1)=0$, while $\varphi_{N_0}(x_{N_0})=1$ and $\varphi_{N_0}(x_{{N_0}-1})=0$.
We now partition the original source function $f$ into small pieces as follows:
\[
f_j(x):=f(x) \varphi_j(x),\qquad j=0,\ldots,{N_0}.
\]
Since $\sum_{j=0}^{N_0} \varphi_j(x)=1$ for all $x\in (0,1)$, we have
$f=\sum_{j=0}^{N_0} f_j$. Let $u_j\in H^1(x_{j-1},x_{j+1})$ be the weak solution to the regular local problem:
\be \label{heqnL:local:fj}
\mathcal{L} u_{j}(x)=f_j(x), \qquad x\in (x_{j-1},x_{j+1})
\ee
with the following boundary conditions:
\be \label{heqnl:local:fj:bdry}
\begin{split}
&(\mathcal{B}_0 u_j(0)-g_0\delta_{0,j})\delta_{0,x_{j-1}}+
(1-\delta_{0,x_{j-1}})u_j(x_{j-1})=0,\quad\\
&(\mathcal{B}_1 u_j(1)-g_1\delta_{N_0,j})\delta_{1,x_{j+1}}+
(1-\delta_{1,x_{j+1}})u_j(x_{j+1})=0.
\end{split}
\ee
That is, we use the homogeneous Dirichlet boundary conditions $u_j(x_{j-1})=u_j(x_{j+1})=0$, except
$\mathcal{B}_0 u_0(0)=g_0$,
$\mathcal{B}_0 u_1(0)=0$,
$\mathcal{B}_1 u_{N_0}(1)=g_1$, and
$\mathcal{B}_1 u_{N_0-1}(1)=0$.
Due to \eqref{heqnl:local:fj:bdry}, we can extend $u_j\in H^1(x_{j-1},x_{j+1})$ as an element in $H^1(0,1)$ by zero extension, which is denoted by $\tilde{u}_j$. Hence,
\be \label{heqnL:3}
\tilde{f}_j(x):=\mathcal{L} \tilde{u}_j(x)=
\begin{cases}
0, &\text{$x\in (0,x_{j-1})\cup (x_{j+1},1)$},\\
d_{x_{j-1}}(\tilde{u}_j)\delta_{x_{j-1}}, &\text{$x=x_{j-1}$ and $x_{j-1}\ne 0$},\\
f_j(x), &\text{$ x \in (x_{j-1},x_{j+1})$},\\
d_{x_{j+1}}(\tilde{u}_j) \delta_{x_{j+1}}, &\text{$x=x_{j+1}$ and $x_{j+1}\ne 1$}.
\end{cases}
\ee
Now we discuss how to link/stitch all these local solutions $\{\tilde{u}_j\}_{j=0}^{N_0}$ together. To do so, for $j=1,\ldots,{N_0}-1$,
we solve the following Dirac assisted local problem:
\be \label{heqnL:local:delta}
\mathcal{L} v_j(x)=\delta_{x_j}, \qquad x\in (x_{j-1},x_{j+1})
\ee
with the following boundary conditions:
\be \label{heqnl:local:dirac:bdry}
\begin{split}
&\mathcal{B}_0 v_j(0)\delta_{0,x_{j-1}}+(1-\delta_{0,x_{j-1}})
v_j(x_{j-1})=0,\quad\\
&\mathcal{B}_1 v_j(1)\delta_{1,x_{j+1}}+
(1-\delta_{1,x_{j+1}})v_j(x_{j+1})=0.
\end{split}
\ee
That is, we use
the homogeneous Dirichlet boundary condition $v_j(x_{j-1})=v_j(x_{j+1})=0$, except $\mathcal{B}_0 v_1(0)=0$ and $\mathcal{B}_1 v_{N_0-1}(1)=0$.
As explained before, due to \eqref{heqnl:local:dirac:bdry}, we can extend $v_j\in H^1(x_{j-1},x_{j+1})$ as an element in $H^1(0,1)$ by zero extension, which is denoted by $\tilde{v}_j$. So, we must have
\be \label{heqnL:4}
\tilde{\delta}_{x_j}(x):=\mathcal{L} \tilde{v}_j(x)=
\begin{cases}
0, &\text{$x\in (0,x_{j-1}) \cup (x_{j+1},1)$},\\
d_{x_{j-1}}(\tilde{v}_j)\delta_{x_{j-1}}, &\text{$x=x_{j-1}$ and $x_{j-1}\ne 0$},\\
\delta_{x_j}, &\text{$x \in (x_{j-1},x_{j+1})$},\\
d_{x_{j+1}}(\tilde{v}_j) \delta_{x_{j+1}}, &\text{$x=x_{j+1}$ and $x_{j+1}\ne 1$}.
\end{cases}
\ee

To link all the local solutions $\{\tilde{u}_j\}_{j=0}^{N_0}$ together, we need the following result.

\begin{theorem}\label{thm:link}
The elements in $\{\tilde{v}_1,\ldots,\tilde{v}_{{N_0}-1}\}$ are linearly independent and for any complex numbers $\mu_j, j=1,\ldots,{N_0}-1$, the following linear system induced by
\be \label{link:uniqueness}
\sum_{j=1}^{{N_0}-1} \tilde{\mu}_j \tilde{\delta}_{x_j}=\sum_{j=1}^{{N_0}-1} \mu_j \delta_{x_j}
\ee
has a unique solution $\{\tilde{\mu}_j\}_{j=1}^{{N_0}-1}$. Moreover, $V=W$, where $V$ is the linear span of $\{\tilde{v}_j\}_{j=1}^{{N_0}-1}$ and $W$ is the linear span of $\{w_j\}_{j=1}^{{N_0}-1}$, where $w_j$ is the weak solution to the following global problem:
\be \label{w}
\mathcal{L} w_j(x)=\delta_{x_j}, \qquad x\in (0,1)\quad \mbox{with}\quad \mathcal{B}_0 w_j(0)=0,\quad \mathcal{B}_1 w_j(1)=0.
\ee
\end{theorem}

\begin{proof} Since $\mathcal{L} \tilde{v}_j=\tilde{\delta}_{x_{j}}$ on $(0,1)$ and $\tilde{v}_j(0)=\tilde{v}_j(1)=0$ except $\mathcal{B}_0 \tilde{v}_1(0)=0$ and $\mathcal{B}_1 \tilde{v}_{N_0-1}(1)=0$, we obviously have $\tilde{v}_j\in W$ and hence $V\subseteq W$. We now prove that $\tilde{v}_1,\ldots,\tilde{v}_{N_0-1}$ are linearly independent. To do so, we claim that it is impossible that either $\tilde{v}_j|_{(x_{j-1},x_j)}$ or $\tilde{v}_j|_{(x_j,x_{j+1})}$ can be identically zero. Without loss of generality, we assume that $\tilde{v}_j|_{(x_{j-1},x_j)}$ is identically zero. Since $\tilde{v}_j|_{(x_{j-1},x_{j+1})}=v_j\in H^1(x_{j-1},x_{j+1})$, the function $v_j$ is continuous on $(x_{j-1},x_{j+1})$ and hence $v_j(x_j)=0$. However, since $v_j$ is the weak solution to the local linking problem in \eqref{heqnL:local:delta}, we see that $\mathring{v}_j:=v_j|_{(x_{j},x_{j+1})}$ must be the weak solution to $\mathcal{L} \mathring{v}_j(x)=0$ on $(x_j,x_{j+1})$ with the boundary conditions $\mathring{v}_j(x_j)=0$ and $\mathring{v}_j(x_{j+1})=0$ (if $j=N_0-1$, then $x_{j+1}=1$ and replace $\mathring{v}_j(x_{j+1})=0$ by $\mathcal{B}_1 \mathring{v}_{N_0-1}(1)=0$).
By the uniqueness of the solution, the weak solution $\mathring{v}_j$ must be identically zero. Hence, $v_j$ must be identically zero, which contradicts \eqref{heqnL:local:delta}. Hence, both
$v_j|_{(x_{j-1},x_j)}$ and $v_j|_{(x_j,x_{j+1})}$ cannot be identically zero; i.e., $\tilde{v}_j|_{(x_{j-1},x_j)}$ and $\tilde{v}_j|_{(x_j,x_{j+1})}$ cannot be identically zero.

Consider the linear combination $v:=\sum_{j=1}^{N_0-1} \mu_j \tilde{v}_j$ such that $v$ is identically zero. Because $\tilde{v}_j$ vanishes outside $(x_{j-1},x_{j+1})$, we have $0=v|_{(x_0,x_1)}=\mu_1 \tilde{v}_1|_{(x_0,x_1)}$. Since $\tilde{v}_1|_{(x_0,x_1)}$ cannot be identically zero, we must have $\mu_1=0$. By induction on $j$, we must have $\mu_1=\mu_2=\cdots=\mu_{N_0-1}=0$. This proves that the elements in $\{\tilde{v}_j\}_{j=1}^{N_0-1}$ must be linearly independent. Now by $V\subseteq W$, we conclude that $V=W$.
The uniqueness of the solution to the linear system in \eqref{link:uniqueness} follows straightforwardly, since $\mathcal{L} \tilde{v}_j=\tilde{\delta}_{x_j}$, $\mathcal{L} w_j=\delta_{x_j}$ and $V=W$.
\end{proof}

The following result is the main ingredient of our DAT method.

\begin{theorem}\label{thm:dat}
Define
\be \label{dat:u}
u:=u_f+u_\delta \quad \mbox{with}\quad u_f:=\sum_{j=0}^{N_0} \tilde{u}_j,\quad u_\delta:=\sum_{j=1}^{{N_0}-1} \mu_j \tilde{v}_j,
\ee
where $\tilde{u}_j$ is the weak solution to \eqref{heqnL:local:fj} with prescribed boundary conditions in \eqref{heqnl:local:fj:bdry} extended by zero, $\tilde{v}_j$ is the weak solution to \eqref{heqnL:local:delta} with the prescribed boundary conditions in \eqref{heqnl:local:dirac:bdry} extended by zero, and $\{\mu_j\}_{j=1}^{N_0-1}$ is the unique solution to the following linear system for the linking problem:
\be \label{link:coef}
\sum_{j=1}^{{N_0}-1} \mu_j \tilde{\delta}_{x_j}=-\sum_{j=1}^{N_0-1}(d_{x_j}(\tilde{u}_{j-1})+d_{x_j}(\tilde{u}_{j+1})) \dirac_{x_j},
\ee
Then
$u$ must be the weak solution to the heterogeneous Helmholtz equation in
\eqref{heqnL} with the boundary conditions in \eqref{heqnbdry}.
\end{theorem}

\begin{proof}
Since $\sum_{j=0}^{N_0} f_j=f$, we can write
\be \label{f:tf}
f(x)=\sum_{j=0}^{N_0} f_j(x)=
\sum_{j=0}^{N_0} \tilde{f}_j(x)-
\sum_{j=1}^{{N_0}-1} (d_{x_j}(\tilde{u}_{j-1})+d_{x_j}(\tilde{u}_{j+1})) \delta_{x_j}.
\ee
By \cref{thm:link}, there is a unique solution $\{\mu_j\}_{j=1}^{N_0-1}$ to \eqref{link:coef}. That is, the linking problem in \eqref{link:coef} can be uniquely solved.
Hence, using \eqref{f:tf}, we can further write
\be \label{f:tf:2}
f(x)=
\sum_{j=0}^{N_0} \tilde{f}_j(x)+
\sum_{j=1}^{{N_0}-1} \mu_j \tilde{\delta}_{x_j}.
\ee
By the definition of $\tilde{u}_j$, we observe that
$\mathcal{L}u_f (x)=\sum_{j=0}^{N_0}\tilde{f}_j(x)$ for $x\in (0,1)$ with the boundary conditions $\mathcal{B}_0 u_f(0)=g_0$ and $\mathcal{B}_1 u_f(1)=g_1$.
On the other hand, by the definition of $\tilde{\delta}_{x_j}$, we have
\[
\mathcal{L}u_\delta (x)=\sum_{j=1}^{N_0-1}\mu_j \tilde{\delta}_{x_j},\qquad x\in (0,1)
\]
and $u_\delta$ satisfies the boundary conditions $\mathcal{B}_0 u_\delta(0)=0$ and $\mathcal{B}_1 u_\delta(1)=0$.
Thus, by \eqref{f:tf:2} we have
\[
\mathcal{L} u=\mathcal{L} u_f+\mathcal{L}u_\delta=\sum_{j=0}^{N_0} \tilde{f}_j(x)+
\sum_{j=1}^{N_0-1} \mu_j \tilde{\delta}_{x_j}=f(x),\qquad x\in (0,1)
\]
and $u$ satisfies the prescribed boundary conditions in \eqref{heqnbdry}.
\end{proof}

Obviously, we can recursively apply the above procedure to solve each of the local problems in \eqref{heqnL:local:fj} with prescribed boundary conditions in \eqref{heqnl:local:fj:bdry} to further reduce the size of the problem. To elucidate this point, we now present the DAT algorithm below.

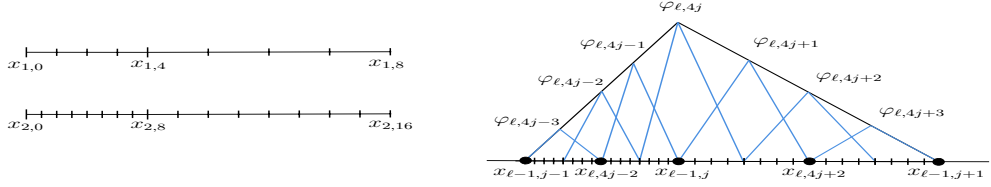
\begin{figure}[htbp]
	 \resizebox{0.7\textwidth}{1.25cm}{%
	 \begin{minipage}{0.45\textwidth}
	\tikzset{every picture/.style={line width=0.75pt}} 
	 \begin{tikzpicture}[x=0.75pt,y=0.75pt,yscale=-1,xscale=1]
	\draw    (59.9,130.6) -- (139.83,130.67) (79.9,126.62) -- (79.9,134.62)(99.9,126.63) -- (99.9,134.63)(119.9,126.65) -- (119.9,134.65) ;
	\draw [shift={(139.83,130.67)}, rotate = 180.05] [color={rgb, 255:red, 0; green, 0; blue, 0 }  ][line width=0.75]    (0,5.59) -- (0,-5.59)   ;
	\draw [shift={(59.9,130.6)}, rotate = 180.05] [color={rgb, 255:red, 0; green, 0; blue, 0 }  ][line width=0.75]    (0,5.59) -- (0,-5.59)   ;
	\draw    (139.83,130.67) -- (300.5,130) (179.82,126.5) -- (179.85,134.5)(219.82,126.33) -- (219.85,134.33)(259.82,126.17) -- (259.85,134.17)(299.82,126) -- (299.85,134) ;
	\draw    (59.9,199.6) -- (139.83,199.67) (69.9,195.61) -- (69.9,203.61)(79.9,195.62) -- (79.9,203.62)(89.9,195.63) -- (89.9,203.63)(99.9,195.63) -- (99.9,203.63)(109.9,195.64) -- (109.9,203.64)(119.9,195.65) -- (119.9,203.65)(129.9,195.66) -- (129.9,203.66) ;
	\draw [shift={(139.83,199.67)}, rotate = 180.05] [color={rgb, 255:red, 0; green, 0; blue, 0 }  ][line width=0.75]    (0,5.59) -- (0,-5.59)   ;
	\draw [shift={(59.9,199.6)}, rotate = 180.05] [color={rgb, 255:red, 0; green, 0; blue, 0 }  ][line width=0.75]    (0,5.59) -- (0,-5.59)   ;
	\draw    (139.83,199.67) -- (300.5,199) (159.82,195.58) -- (159.85,203.58)(179.82,195.5) -- (179.85,203.5)(199.82,195.42) -- (199.85,203.42)(219.82,195.33) -- (219.85,203.33)(239.82,195.25) -- (239.85,203.25)(259.82,195.17) -- (259.85,203.17)(279.82,195.09) -- (279.85,203.09)(299.82,195) -- (299.85,203) ;
	
	\draw (46,139.4) node [anchor=north west][inner sep=0.75pt]    {$x_{1,0}$};
	\draw (284,137.4) node [anchor=north west][inner sep=0.75pt]    {$x_{1,8}$};
	\draw (127,138.4) node [anchor=north west][inner sep=0.75pt]    {$x_{1,4}$};
	\draw (46,206.4) node [anchor=north west][inner sep=0.75pt]    {$x_{2,0}$};
	\draw (127,206.4) node [anchor=north west][inner sep=0.75pt]    {$x_{2,8}$};
	\draw (284,205.4) node [anchor=north west][inner sep=0.75pt]    {$x_{2,16}$};
	\end{tikzpicture}
	\end{minipage}
	 \begin{minipage}{0.45\textwidth}
	\tikzset{every picture/.style={line width=0.75pt}} 
	 \begin{tikzpicture}[x=0.75pt,y=0.75pt,yscale=-1,xscale=1,scale=0.9]
	\draw    (250,212.2) -- (477.5,212.68) (298.01,208.3) -- (297.99,216.3)(346.01,208.4) -- (345.99,216.4)(394.01,208.5) -- (393.99,216.5)(442.01,208.6) -- (441.99,216.6) ;
	\draw    (109.5,211.26) -- (250,212.2) (137.53,207.44) -- (137.47,215.44)(165.53,207.63) -- (165.47,215.63)(193.53,207.82) -- (193.47,215.82)(221.52,208.01) -- (221.47,216.01)(249.52,208.2) -- (249.47,216.2) ;
	\draw    (138,210.78) -- (249.5,40.6) ;
	\draw    (249.5,40.6) -- (440.5,212.68) ;
	\draw [color={rgb, 255:red, 74; green, 144; blue, 226 }  ,draw opacity=1 ]   (249.5,40.6) -- (297.5,212.68) ;
	\draw [color={rgb, 255:red, 74; green, 144; blue, 226 }  ,draw opacity=1 ]   (249.5,40.6) -- (221.5,211.26) ;
	\draw [color={rgb, 255:red, 74; green, 144; blue, 226 }  ,draw opacity=1 ]   (217,91.33) -- (250,212.2) ;
	\draw [color={rgb, 255:red, 74; green, 144; blue, 226 }  ,draw opacity=1 ]   (217,91.33) -- (193.5,214.1) ;
	\draw [color={rgb, 255:red, 74; green, 144; blue, 226 }  ,draw opacity=1 ]   (193.75,125.69) -- (221.5,211.26) ;
	\draw [color={rgb, 255:red, 74; green, 144; blue, 226 }  ,draw opacity=1 ]   (193.75,125.69) -- (165.5,212.68) ;
	\draw [color={rgb, 255:red, 74; green, 144; blue, 226 }  ,draw opacity=1 ]   (163,172.38) -- (193.5,212.1) ;
	\draw [color={rgb, 255:red, 74; green, 144; blue, 226 }  ,draw opacity=1 ]   (163,172.38) -- (138,210.78) ;
	\draw [color={rgb, 255:red, 74; green, 144; blue, 226 }  ,draw opacity=1 ]   (301.5,87.53) -- (250,212.2) ;
	\draw [color={rgb, 255:red, 74; green, 144; blue, 226 }  ,draw opacity=1 ]   (301.5,87.53) -- (345.5,212.68) ;
	\draw [color={rgb, 255:red, 74; green, 144; blue, 226 }  ,draw opacity=1 ]   (345,126.64) -- (297.5,212.68) ;
	\draw [color={rgb, 255:red, 74; green, 144; blue, 226 }  ,draw opacity=1 ]   (345,126.64) -- (393.5,211.26) ;
	\draw [color={rgb, 255:red, 74; green, 144; blue, 226 }  ,draw opacity=1 ]   (391,168.12) -- (345.5,212.68) ;
	\draw [color={rgb, 255:red, 74; green, 144; blue, 226 }  ,draw opacity=1 ]   (391,168.12) -- (440.5,212.68) ;
	\draw    (250,212.2) -- (440.5,212.68) (262.01,208.23) -- (261.99,216.23)(274.01,208.26) -- (273.99,216.26)(286.01,208.29) -- (285.99,216.29)(298.01,208.32) -- (297.99,216.32)(310.01,208.35) -- (309.99,216.35)(322.01,208.38) -- (321.99,216.38)(334.01,208.41) -- (333.99,216.41)(346.01,208.44) -- (345.99,216.44)(358.01,208.47) -- (357.99,216.47)(370.01,208.5) -- (369.99,216.5)(382.01,208.53) -- (381.99,216.53)(394.01,208.56) -- (393.99,216.56)(406.01,208.59) -- (405.99,216.59)(418.01,208.62) -- (417.99,216.62)(430.01,208.65) -- (429.99,216.65) ;
	\draw    (130.5,211.73) -- (250,212.2) (137.52,207.76) -- (137.48,215.76)(144.52,207.78) -- (144.48,215.78)(151.52,207.81) -- (151.48,215.81)(158.52,207.84) -- (158.48,215.84)(165.52,207.87) -- (165.48,215.87)(172.52,207.9) -- (172.48,215.9)(179.52,207.92) -- (179.48,215.92)(186.52,207.95) -- (186.48,215.95)(193.52,207.98) -- (193.48,215.98)(200.52,208.01) -- (200.48,216.01)(207.52,208.03) -- (207.48,216.03)(214.52,208.06) -- (214.48,216.06)(221.52,208.09) -- (221.48,216.09)(228.52,208.12) -- (228.48,216.12)(235.52,208.15) -- (235.48,216.15)(242.51,208.17) -- (242.48,216.17)(249.51,208.2) -- (249.48,216.2) ;
	\draw    (136,95)  ;
	\draw  [fill={rgb, 255:red, 0; green, 0; blue, 0 }  ,fill opacity=1 ] (246,212.19) .. controls (246,209.89) and (247.8,208.03) .. (250.01,208.04) .. controls (252.22,208.05) and (254.01,209.91) .. (254,212.21) .. controls (254,214.51) and (252.2,216.37) .. (249.99,216.37) .. controls (247.78,216.36) and (245.99,214.49) .. (246,212.19) -- cycle ;
	\draw  [fill={rgb, 255:red, 0; green, 0; blue, 0 }  ,fill opacity=1 ] (342,212.19) .. controls (342,209.89) and (343.8,208.03) .. (346.01,208.04) .. controls (348.22,208.05) and (350.01,209.91) .. (350,212.21) .. controls (350,214.51) and (348.2,216.37) .. (345.99,216.37) .. controls (343.78,216.36) and (341.99,214.49) .. (342,212.19) -- cycle ;
	\draw  [fill={rgb, 255:red, 0; green, 0; blue, 0 }  ,fill opacity=1 ] (436.48,212.83) .. controls (436.49,210.53) and (438.29,208.67) .. (440.5,208.68) .. controls (442.71,208.68) and (444.5,210.55) .. (444.49,212.85) .. controls (444.49,215.15) and (442.69,217.01) .. (440.48,217) .. controls (438.27,217) and (436.48,215.13) .. (436.48,212.83) -- cycle ;
	\draw  [fill={rgb, 255:red, 0; green, 0; blue, 0 }  ,fill opacity=1 ] (189,212.19) .. controls (189,209.89) and (190.8,208.03) .. (193.01,208.04) .. controls (195.22,208.05) and (197.01,209.91) .. (197,212.21) .. controls (197,214.51) and (195.2,216.37) .. (192.99,216.37) .. controls (190.78,216.36) and (188.99,214.49) .. (189,212.19) -- cycle ;
	\draw  [fill={rgb, 255:red, 0; green, 0; blue, 0 }  ,fill opacity=1 ] (134,211.19) .. controls (134,208.89) and (135.8,207.03) .. (138.01,207.04) .. controls (140.22,207.05) and (142.01,208.91) .. (142,211.21) .. controls (142,213.51) and (140.2,215.37) .. (137.99,215.37) .. controls (135.78,215.36) and (133.99,213.49) .. (134,211.19) -- cycle ;
	
	\draw (231,219.67) node [anchor=north west][inner sep=0.75pt]    {$x_{\ell -1,j}$};
	\draw (416,220.09) node [anchor=north west][inner sep=0.75pt]    {$x_{\ell -1,j+1}$};
	\draw (112.5,220.14) node [anchor=north west][inner sep=0.75pt]    {$x_{\ell -1,j-1}$};
	\draw (234,13.04) node [anchor=north west][inner sep=0.75pt]    {$\varphi _{\ell ,4j}$};
	\draw (176,59.97) node [anchor=north west][inner sep=0.75pt]    {$\varphi _{\ell ,4j-1}$};
	\draw (145,105.48) node [anchor=north west][inner sep=0.75pt]    {$\varphi _{\ell ,4j-2}$};
	\draw (113,155.25) node [anchor=north west][inner sep=0.75pt]    {$\varphi _{\ell ,4j-3}$};
	\draw (303,58.55) node [anchor=north west][inner sep=0.75pt]    {$\varphi _{\ell ,4j+1}$};
	\draw (349,101.21) node [anchor=north west][inner sep=0.75pt]    {$\varphi _{\ell ,4j+2}$};
	\draw (395,143.88) node [anchor=north west][inner sep=0.75pt]    {$\varphi _{\ell ,4j+3}$};
	\draw (324,220.09) node [anchor=north west][inner sep=0.75pt]    {$x_{\ell ,4j+2}$};
	\draw (172,219.67) node [anchor=north west][inner sep=0.75pt]    {$x_{\ell ,4j-2}$};
	\end{tikzpicture}
	\end{minipage}
	}
	\caption{Left: An initial partition of $[0,1]$ at $\ell=1$ with $N_0=8$ and its subsequent refinement at $\ell=2$ with $s=1$. Right: The relationship between an interior large hat function $\varphi_{\ell-1,j}$ on $[x_{\ell-1,j-1},x_{\ell-1,j+1}]$ and smaller hat functions $\varphi_{\ell,2^{s}j+k}$, $-2^{s}+1\le k \le 2^{s}-1$, with $s=2$.}
	\label{fig:DAT}
\end{figure}

\begin{algorithm}\label{alg:dat}
	\normalfont
	Consider the $1$-level partition $\{x_{1,j}\}_{j=0}^{N_0}$ given by $0=x_{1,-1}=x_{1,0}<x_{1,1}<\cdots<x_{1,N_0-1}<x_{1,N_0}=x_{1,N_0+1}=1$, and let $\{\varphi_{1,j}\}_{j=0}^{N_0}$ be the associated partition of unity such that $\supp(\varphi_{1,j}) \subseteq [x_{1,j-1},x_{1,j+1}]$ with $\varphi_{1,j}(x_{1,j})=1$. Pick $L,s\in \N$ such that $L$ is the tree level and any subinterval at a level is divided equally into $2^s$ small subintervals at the next level. For each tree level $\ell=2,\ldots,L$, let $\{x_{\ell,j}\}_{j=0}^{2^{\ell s} N_0}$ be a refinement partition of the grid $\{x_{\ell-1,j}\}_{j=0}^{2^{(\ell-1)s} N_0}$ such that $\{x_{\ell-1,j}\}_{j=0}^{2^{(\ell-1)s}N_0} \subset \{x_{\ell,j}\}_{j=0}^{2^{\ell s} N_0}$ with
	 $x_{\ell-1,j}=x_{\ell,2^{s}j}$, $x_{\ell,-1}:=x_{\ell,0}=0$, and $x_{\ell,2^{\ell s} N_0+1}:=x_{\ell,2^{\ell s} N_0}=1$.
Fix $N \in \mathbb{N}$ such that $N$ is the total number of points on the finest grid in the tree and $\{x_{L,j}\}_{j=0}^{2^{Ls}N_0} \subset \{x_{j}\}_{j=0}^{N}$. The local problems will be solved on this fixed fine grid. Note that $\supp(\varphi_{\ell,j}) \subseteq [x_{\ell,j-1},x_{\ell,j+1}]$ with $\varphi_{\ell,j}(x_{\ell,j})=1$ for $0\le j \le 2^{(\ell-1) s}N_0$ and
	\begin{align*}	
	 \varphi_{\ell-1,0}=\sum_{k=0}^{2^{s}-1} \varphi_{\ell,k}, \quad \varphi_{\ell-1,2^{(\ell-2)s}N_0}=\sum_{k=-2^{s}+1}^{0} \varphi_{\ell,2^{(\ell-1) s}N_0+k}, \quad
	 \varphi_{\ell-1,j}=\sum_{k=-2^{s}+1}^{2^{s}-1} \varphi_{\ell,2^{s}j+k}
	\end{align*}
	for $j=1, \dots, 2^{(\ell-2)s}N_0-1$. See \cref{fig:DAT} for an illustration of the setting above.

	\begin{itemize}
		\item[(S1)]
		Solve the following (regular and Dirac assisted) local problems at tree level $L$ in a parallel fashion using any chosen discretization method.
		\begin{align}
			& \begin{cases}	 &\mathcal{L}{u_{L,j}}=f_{j}=:f\varphi_{L,j}, \qquad x \in [x_{L,j-1},x_{L,j+1}], \qquad j=0,\dots, 2^{(L-1)s}N_0,\\
				 &(\mathcal{B}_0 u_{L,j}(0)-g_0\delta_{0,j})\delta_{0,x_{L,j-1}}+
				 (1-\delta_{0,x_{L,j-1}})u_{L,j}(x_{L,j-1})=0,\\
				 &(\mathcal{B}_1 u_{L,j}(1)-g_1\delta_{2^{Ls}N_0,j})\delta_{1,x_{L,j+1}}+
				 (1-\delta_{1,x_{L,j+1}})u_{L,j}(x_{L,j+1})=0,
			\end{cases} \label{localproblem:regular}\\
			&	 \begin{cases} &\mathcal{L}{v_{L,j}}=\delta_{x_{L,j}}, \qquad x \in [x_{L,j-1},x_{L,j+1}], \qquad j=1,\dots, 2^{(L-1)s}N_0-1,\\
				 &\mathcal{B}_0 v_{L,j}(0)\delta_{0,x_{L,j-1}}+
				 (1-\delta_{0,x_{L,j-1}})v_{L,j}(x_{L,j-1})=0,\\
				 &\mathcal{B}_1 v_{L,j}(1)\delta_{1,x_{L,j+1}}+
				 (1-\delta_{1,x_{L,j+1}})v_{L,j}(x_{L,j+1})=0. \end{cases}\label{localproblem:dirac}
		\end{align}
		For all $j=0,\dots, 2^{(L-1)s}N_0$, extend $u_{L,j}$ by zero outside of $[x_{L,j-1},x_{L,j+1}]$ and denote it by $\tilde{u}_{L,j}$. Similarly, for all $j=1,\dots, 2^{(L-1)s}N_0-1$, extend $v_{L,j}$ by zero outside of $[x_{L,j-1},x_{L,j+1}]$ and denote it by $\tilde{v}_{L,j}$.
		
		\item[(S2)] Let $\ell=L,\ldots,2$ decreasingly. Consider the artificial Dirac distributions at each grid point and find the appropriate linear combination of Dirac local problems to offset it. This allows us to recover the solutions to local problems at level $\ell-1$ from those at level $\ell$. More explicitly, define an $n_{2} \times n_{2}$ tridiagonal matrix and an $n_{2}$ column vector as follows
		\begin{align*}
			 T_{\ell,n_{1},n_{2}} & :=\text{tridiag}(
			 \{d_{x_{\ell,n_{1}+m}}(\tilde{v}_{\ell,n_{1}+m-1})\}_{m=2}^{n_{2}},
			 \{1\}_{m=1}^{n_2},
			 \{d_{x_{\ell,n_{1}+m}}(\tilde{v}_{\ell,n_{1}+m+1})\}_{m=1}^{n_2-1}
			),\\
			 \gamma_{\ell,n_{1},n_{2}} & :=  [-d_{x_{\ell,n_{1}+m}}(\tilde{u}_{\ell,n_{1}+m-1})(1-\dirac_{1,m}(1-\dirac_{0,n_1}(1-\dirac_{1,\ell})))\\
			& \qquad \qquad -d_{x_{\ell,n_{1}+m}}(\tilde{u}_{\ell,n_{1}+m+1})(1-\dirac_{n_2,m}(1-\dirac_{2^{(\ell-1)s}N_0-2^{s},n_1}(1-\dirac_{1,\ell})))]_{1\le m \le n_2},
		\end{align*}
		where $n_{1}, n_{2} \in \mathbb{N} \cup \{0\}$, and the first, second, and third arguments of $\text{tridiag}(\cdot,\cdot,\cdot)$ correspond to the entries in the lower, main, and upper diagonals. Given a column vector $\mu$, we denote the $k$th component of $\mu$ by $(\mu)_k$. For each $\ell$, solve the linking problems obtained from steps (a)-(d) below in a parallel fashion.
		\begin{enumerate}
			\item[(a)] (Left-most element of partition of unity) Construct a $(2^{s}-1) \times (2^{s}-1)$ tridiagonal matrix $T_{\ell,0,2^s-1}$ and a $(2^{s}-1)$ column vector $\gamma_{\ell,0,2^s-1}$. Set $\mu_{\ell,0} = (T_{\ell,0,2^s-1})^{-1}\gamma_{\ell,0,2^s-1}$.
			\item[(b)] (Interior elements of partition of unity) For all $j=1,\dots,2^{(\ell-2)s}N_0-1$, construct a $(2^{s+1}-1) \times (2^{s+1}-1)$ tridiagonal matrix $T_{\ell,2^{s}(j-1),2^{s+1}-1}$ and a $(2^{s+1}-1)$ column vector $\gamma_{\ell,2^{s}(j-1),2^{s+1}-1}$. Let $e_{2^{s}}$ is a $(2^{s+1}-1)$ vector with $1$ in the $2^{s}$th entry and $0$ in the remaining entries. Set
			 $\mu_{\ell,j} = (T_{\ell,2^{s}(j-1),2^{s+1}-1})^{-1}\gamma_{\ell,2^{s}(j-1),2^{s+1}-1}$ and $\nu_{\ell,j}=(T_{\ell,2^{s}(j-1),2^{s+1}-1})^{-1} e_{2^{s}}$.
			\item[(c)] (Right-most element of partition of unity) Construct a $(2^{s}-1) \times (2^{s}-1)$ tridiagonal matrix $T_{\ell,2^{(\ell-1) s}N_0-2^{s},2^s-1}$ and a $(2^{s}-1)$ column vector $\gamma_{\ell,2^{(\ell-1) s}N_0-2^{s},2^s-1}$. Set $\mu_{\ell,2^{(\ell-2)s}N_0} = (T_{\ell,2^{(\ell-1) s}N_0-2^{s},2^s-1})^{-1}\gamma_{\ell,2^{(\ell-1) s}N_0-2^{s},2^s-1}$.
			\item[(d)] (Construct the solutions to local problems at level $\ell-1$)
			For $j=0,2^{(\ell-2)s}N_0$ (left-most and right-most elements respectively), set
			 \begin{align*}
				& \tilde{u}_{\ell-1,0} = \sum_{k=0}^{2^{s}-1} \tilde{u}_{\ell,k} + \sum_{k=1}^{2^{s}-1} (\mu_{\ell,0})_{k} \tilde{v}_{\ell,k}, \\
				& \tilde{u}_{\ell-1,2^{(\ell-2)s}N_0} = \sum_{k=-2^{s}+1}^{0} \tilde{u}_{\ell,2^{(\ell-1) s}N_0+k} + \sum_{k=1}^{2^{s}-1} (\mu_{\ell,2^{(\ell-2)s}N_0})_{k} \tilde{v}_{\ell,2^{(\ell-1) s}N_0-2^{s}+k}.
			 \end{align*}
			For $j=1, \dots, 2^{(\ell-2)s}N_0-1$ (interior elements), set
			\[
			\qquad \qquad \qquad \tilde{u}_{\ell-1,j} = \sum_{k=-2^{s}+1}^{2^{s}-1} \tilde{u}_{\ell,2^{s}j+k} + \sum_{k=1}^{2^{s+1}-1} (\mu_{\ell,j})_{k} \tilde{v}_{\ell,2^{s}j-2^s+k}, \qquad
			 \tilde{v}_{\ell-1,j} = \sum_{k=1}^{2^{s+1}-1} (\nu_{\ell,j})_{k} \tilde{v}_{\ell,2^{s}j-2^s+k}.
			\]
		\end{enumerate}
		\item[(S3)]
		Construct an $(N_0-1) \times (N_0-1)$ tridiagonal matrix $T_{1,0,N_0-1}$ and an $(N_0-1)$ column vector $\gamma_{1,0,N_0-1}$. Set $\mu_{1,0}=(T_{1,0,N_0-1})^{-1} \gamma_{1,0,N_0-1}$.
		Finally, the approximated solution of the problem  \eqref{heqnL}-\eqref{heqnbdry} is given by
$u=\sum_{k=0}^{N_0} \tilde{u}_{1,k} + \sum_{k=1}^{N_0-1} (\mu_{1,0})_{k} \tilde{v}_{0,k}$.
	\end{itemize}
\end{algorithm}

As an illustrative example, for an equispaced grid on $[0,1]$ with $N_0=4$, $h=2^{-n}$, $L=n-2$, $s=1$ and $n\in \N$, the size of each linking and local problem (with the exception of those near the boundaries) is a $3 \times 3$ matrix equation. This exactly describes the situation in \cref{ex:constant,ex:akvary} in \cref{sec:examples}.

Thus far, we have described DAT for the linear differential operator $\mathcal{L}$ defined in \eqref{heqnL}. The DAT method with appropriate modification can be generalized to general 1D linear differential operators $\mathcal{L}$ (e.g., the biharmonic equation involving higher order derivatives). We would need to modify \cref{thm:dat} about how we patch the local problems in \eqref{heqnL:local} by means of the Dirac assisted local problems in \eqref{heqnL:local:delta} equipped with suitable boundary conditions. In addition to \eqref{heqnL:local:delta}, we may have to solve additional Dirac assisted linking problems in \eqref{heqnL:local:delta} using higher order distributional derivatives of $\dirac_{x_j}$.

As it currently stands, DAT can handle multidimensional problems that can be decomposed into a series of 1D problems (e.g., by the separation of variables). We shall provide a few relevant 2D numerical examples in \cref{subsec:2D}. Generalizing DAT for the purpose of solving general 2D/3D problems is a challenging multifaceted problem and demands ingenious ideas. There are two interconnected critical issues that need to be properly resolved. For the sake of discussion, let us restrict ourselves to 2D. The source term would still be partitioned by shifted square hat functions and their refinability can still be used to give rise to the tree structure. The first issue comes from the 2D Dirac assisted local problems.
In contrast to the 1D Dirac assisted local problem whose source term is a Dirac distribution at a single point, the 2D Dirac assisted local problems would have the source term consisting of weighted Dirac distributions defined along the boundary of a rectangular subdomain. Obtaining a highly accurate numerical solution of such 2D Dirac assisted local problem and accurately estimating its outward fluxes are challenging, because its weak solution involves highly singular functions caused by the singular distribution source term. The second issue is about how to formulate the linking problems and generalize \cref{thm:dat} for stitching all the local solutions into a global solution appropriately. This second issue in multiple dimensions is considerably more difficult than in 1D, due to more complicated topology and boundaries of multidimensional subdomains.

\section{Compact Finite Difference Schemes with Arbitrarily High Accuracy Orders}\label{sec:fdm}

To numerically solve the heterogeneous Helmholtz equation in \eqref{heqnL}--\eqref{heqnbdry} with piecewise smooth coefficients $a,\kappa^2$ and source term $f$, in this section we shall study compact finite difference schemes with arbitrarily high accuracy and numerical dispersion orders. Such compact finite difference schemes are important for accurately solving local problems stemming from DAT in the foregoing section.

\subsection{Compact stencils for interior points} \label{subsec:fdm:interior}

We start by stating a simple observation, which is critical for proving the existence of a 1D finite difference scheme with arbitrarily high accuracy order. The following observation uses an analyticity assumption for its theoretical analysis; however, we only require the coefficients to be differentiable up to a certain order as we shall see later in this section.

\begin{prop} \label{prop:span}
Let $a, \kappa^{2}, f$ in \eqref{heqnL} be analytic functions
and let $u$ be an analytic function satisfying $[a(x)u'(x)]'+\kappa^2(x) u(x)=f(x)$ with $a(x)>0$ for all $x\in (0,1)$.
For any point $\xb\in (0,1)$, we have
\be \label{uj}
u^{(j)}(\xb)=E_{j,0}u(\xb) + E_{j,1}u'(\xb) + \sum_{\ell=0}^{j-2} F_{j,\ell} f^{(\ell)}(\xb), \qquad j\ge 2,
\ee
where the quantities $E_{j,0}, E_{j,1}, F_{j,\ell}$ only depend on the values $a(\xb), a'(\xb),\ldots,a^{(j-1)}(\xb)$ and $\kappa^2(\xb)$, $[\kappa^2]'(\xb)$,$\ldots,[\kappa^{2}]^{(j-2)}(\xb)$ for $j\ge 2$ and $\ell\in \NN$. Consequently, for sufficiently small $h$,
\be \label{utay}
u(\xb+h)=
u(\xb) E_0(h) +u'(\xb) h E_1(h) +
\sum_{\ell=0}^\infty h^{\ell+2} f^{(\ell)}(\xb)F_\ell(h),
\ee
where $E_0(h), E_1(h)$ and $F_\ell(h), \ell\in \NN$ are defined to be
\be \label{E12}
E_0(h):=1+\sum_{j=2}^\infty \frac{E_{j,0}}{j!} h^j,\quad
E_1(h):=1+\sum_{j=2}^\infty \frac{E_{j,1}}{j!}h^{j-1},\quad F_\ell(h):=\sum_{j=\ell+2}^\infty \frac{F_{j,\ell}}{j!} h^{j-\ell-2}.
\ee
\end{prop}

\begin{proof}
We prove the claim in \eqref{uj} using mathematical induction on $j$. Consider the base case with $j=2$.
Since $a(x)>0$, we deduce from $[a(x)u'(x)]'+\kappa^2(x) u(x)=f(x)$
that
\be \label{recursive}
u^{(2)}(x)=-\tfrac{\kappa^2(x)}{a(x)} u(x)-\tfrac{a'(x)}{a(x)}u'(x)+\tfrac{f(x)}{a(x)},
\qquad x\in (0,1).
\ee
Hence, setting $x=\xb$ in the above identity \eqref{recursive}, we conclude that \eqref{uj} holds for $j=2$.

Suppose that the claim in \eqref{uj} holds for some $j\ge 2$. We now prove that \eqref{uj} must hold for $j+1$.
Applying the $(j-1)$th derivative to both sides of the identity in \eqref{recursive}, we observe that \[
u^{(j+1)}(x)=
-\left[\tfrac{\kappa^2(x)}{a(x)} u(x)\right]^{(j-1)}
-\left[\tfrac{a'(x)}{a(x)}u'(x)\right]^{(j-1)}
+\left[\tfrac{f(x)}{a(x)}\right]^{(j-1)}.
\]
Applying the Leibniz differentiation formula to the above identity, we conclude that the quantity $u^{(j+1)}(x)$ can be written as a linear combination of $f(x),f'(x), \ldots, f^{(j-1)}(x)$ and $u(x),u'(x),\ldots, u^{(j)}(x)$ with all combination coefficients being analytic functions of $x$ depending only on $a(x), a'(x),\ldots, a^{(j)}(x)$ and $\kappa^2(x), [\kappa^2(x)]',\ldots,[\kappa^2(x)]^{(j-1)}$. Now by induction hypothesis, we conclude that \eqref{uj} holds for $j+1$.
This proves \eqref{uj} by mathematical induction on $j$.

On the other hand, since $u$ is analytic in a neighborhood of $\xb$, the Taylor series of $u$ at the base point $\xb$ is $u(\xb+h)=u(\xb)+u'(\xb)h+\sum_{j=2}^\infty
\frac{u^{(j)}(\xb)}{j!} h^j$. Therefore, we deduce from \eqref{uj} that
\begin{align*}
u(\xb+h)&=u(\xb)+u'(\xb)h +\sum_{j=2}^\infty \frac{h^{j}}{j!}\left(E_{j,0} u(\xb)+E_{j,1}u'(\xb)+\sum_{\ell=0}^{j-2}F_{j,\ell} f^{(\ell)}(\xb)\right)\\
&=u(\xb)\left(1+\sum_{j=2}^\infty \frac{E_{j,0}}{j!} h^j\right)
+u'(\xb)\left(h +\sum_{j=2}^\infty \frac{E_{j,1}}{j!} h^j\right)+
\sum_{j=2}^\infty \sum_{\ell=0}^{j-2} \frac{F_{j,\ell}}{j!} h^j f^{(\ell)}(\xb)\\
&=u(\xb) E_0(h)+u'(\xb) h E_1(h)+
\sum_{\ell=0}^\infty
\sum_{j=\ell+2}^\infty
\frac{F_{j,\ell}}{j!} h^j f^{(\ell)}(\xb),
\end{align*}
from which we obtain
\eqref{utay}.
\end{proof}

Let us now consider compact finite difference schemes with high accuracy and numerical dispersion orders for the heterogeneous Helmholtz equation in \eqref{heqnL} with smooth coefficients $a,\kappa^2$ and source term $f$. Suppose the discretization stencil is centered at an interior point $\xb$ with mesh size $0<h<1$. That is, we fix the base point $\xb$ to be in $(0,1)$ such that $(\xb-h,\xb+h)\subset (0,1)$.

\begin{theorem}\label{thm:interior}
Suppose that $a, \kappa^{2}, f$ in \eqref{heqnL} are smooth functions. Let $M,\tilde{M}$ be positive integers with $M\ge \tilde{M}$. Let $0<h<1$ and $\xb\in (0,1)$ such that $(\xb-h,\xb+h)\subset (0,1)$.
Consider the discretization stencil of a compact finite difference scheme for
$\mathcal{L}u:=[a(x)u'(x)]'+\kappa^2(x)u(x)=f(x)$ (i.e., \eqref{heqnL})
at the base point $\xb$ below
\begin{equation} \label{Lhfd}
\mathcal{L}_{h}u(\xb):= h^{-2}[c_{-1}(h) u(\xb-h)+c_0(h)u(\xb)+c_1(h)u(\xb+h)] - \sum_{\ell=0}^{\tilde{M}-1} d_{\ell}(h) h^{\ell} f^{(\ell)}(\xb),
\end{equation}
where $c_{-1}, c_0, c_1$ and $d_\ell$ are smooth functions of $h$ for $\ell=0,\ldots,\tilde{M}-1$. Suppose that
\be\label{cbo}
\begin{split}
&c_{1}(h)=\frac{\alpha(h)}{ E_1(h)}+\bo(h^{M+1}),\quad
c_{-1}(h)=\frac{\alpha(h)}{ E_1(-h)}+\bo(h^{M+1}),\\
&c_0(h)=-c_1(h) E_0(h)-c_{-1}(h) E_0(-h)+\bo(h^{M+2})
\end{split}
\ee
and
\be \label{dcoeff}
d_{\ell}(h)= - \delta_{0,\ell} +
c_1(h) F_\ell(h)+(-1)^\ell c_{-1}(h)
F_\ell(-h)
+\bo(h^{\tilde{M}-\ell}), \quad \ell=0,\dots,\tilde{M}-1,
\ee
as $h\to 0$, where $\alpha$ is a smooth function of $h$ with $\alpha(0)\ne 0$ and
$E_0(h), E_1(h)$ and
$F_{\ell}(h), \ell\in \NN$ are defined uniquely in \eqref{E12} of \cref{prop:span}.
Then the discretization stencil of the compact finite difference scheme
has numerical dispersion order $M$ at the base point $\xb$, that is,
\be \label{numdispersion}
h^{-2}[c_{-1}(h) u(\xb-h)+c_0(h) u(\xb)+c_1(h) u(\xb+h)]=\bo(h^{M}),\qquad h\to 0,
\ee
for every solution $u$ of $\mathcal{L}u=0$, and has accuracy order $\tilde{M}$ at the base point $\xb$, that is,
\be \label{accuracy}
\mathcal{L}_{h}u(\xb)-f(\xb)=\bo(h^{\tilde{M}}),\qquad h\to 0,
\ee
for every solution $u$ of $\mathcal{L}u=f$.
\end{theorem}

\begin{proof}
By \cref{prop:span} and \eqref{uj},
all the quantities $E_{j,0}, E_{j,1}, F_{j,\ell}$ depend only on the values $a(\xb)$, $a'(\xb),\ldots,a^{(j-1)}(\xb)$ and $\kappa^2(\xb), [\kappa^2]'(\xb),\ldots,[\kappa^{2}]^{(j-2)}(\xb)$ for $j\ge 2$ and $\ell\in \NN$.
For simplicity, we define $u_0:=u(\xb), u_1:=u'(\xb)$ and $f_\ell:=f^{(\ell)}(\xb)$ for $\ell\in \NN$.
Thus, by \eqref{utay}, we deduce
\begin{align}\nonumber
h^{2} &\mathcal{L}_{h} u(\xb)-  h^{2} f_0  = u_{0} \Big(c_{-1}(h) E_0(-h) + c_{0}(h) + c_{1}(h) E_0(h) \Big)+ u_{1} h\Big(-c_{-1}(h) E_1(-h) + c_{1}(h) E_1(h)\Big)\\
& \quad + \sum_{\ell=0}^{\infty} h^{\ell+2}  f_\ell
\left(c_1(h) F_\ell(h)+(-1)^\ell  c_{-1}(h) F_\ell(-h)\right)
- (d_{0}(h)+1)h^{2} f_{0} -\sum_{\ell=1}^{\tilde{M}-1} d_{\ell}(h) h^{\ell+2} f_{\ell},\label{Lhuf}
\end{align}
where $E_0(h)$ and $E_1(h)$ are defined in \eqref{E12}.

On the other hand, from $f(x)=[a(x) u'(x)]'+\kappa^2(x) u(x)$ we trivially observe that $f^{(\ell)}$ can be written as a linear combination of $u, u',\ldots, u^{(\ell+2)}$ as well.
Consequently,
\eqref{numdispersion} holds (or equivalently, \eqref{accuracy} holds with $f=0$ and $\tilde{M}=M$) for numerical dispersion order $M$ if and only if the coefficients of $u_{0}$ and $u_{1}$ in the above identity are $\bo(h^{M+2})$ as $h\to 0$.
That is, \eqref{numdispersion} is equivalent to
\begin{align*}
&c_{-1}(h) E_0(-h)+c_0(h)+c_1(h) E_0(h)=\bo(h^{M+2}),\qquad h \to 0,\\
&-c_{-1}(h) E_1(-h)+c_1(h) E_1(h)=\bo(h^{M+1}),\qquad h \to 0.
\end{align*}
Solving the above equation and noting that $E_0(0)=E_1(0)=1$, we conclude that \eqref{cbo} holds if and only if \eqref{numdispersion} holds.
Thus, \eqref{numdispersion} holds for numerical dispersion order $M$ if and only if \eqref{cbo} holds.

We now prove \eqref{accuracy}.
Since we proved \eqref{numdispersion} and $M\ge \tilde{M}$, \eqref{accuracy} holds for accuracy order $\tilde{M}$ if and only if
all $f_{0},\dots, f_{\tilde{M}-1}$ must be $\bo(h^{\tilde{M}+2})$ as $h\to 0$. Rearranging the last line of \eqref{Lhuf}, we conclude
\begin{align*}
	\sum_{\ell=0}^{\infty} & h^{\ell+2}\left(
c_1(h)F_\ell(h)+(-1)^\ell c_{-1}(h) F_\ell(-h) \right) f_{\ell}
	- (d_{0}+1)h^{2} f_{0} -\sum_{\ell=1}^{\tilde{M}-1} d_{\ell} h^{\ell+2} f_{\ell}\\
&= \sum_{\ell=0}^{\tilde{M}-1} h^{\ell+2}
\left(-(d_{\ell}+\delta_{0,\ell}) +c_1(h)F_\ell(h)+(-1)^\ell c_{-1}(h)F_\ell(-h)
\right)f_{\ell}+\bo(h^{\tilde{M}+2}).
\end{align*}
Since \eqref{numdispersion} holds, \eqref{accuracy} now is equivalent to that all the coefficients $f_{\ell}$, $\ell=0,\dots,\tilde{M}-1$ in the above identity must be $\bo(h^{\tilde{M}+2})$, that is,
\[
- (d_{\ell} + \delta_{0,\ell})+
c_1(h)F_\ell(h)+(-1)^\ell c_{-1}(h) F_\ell(-h)
=\bo(h^{\tilde{M}-\ell}), \qquad \ell=0,\ldots,\tilde{M}-1.
\]
Solving the above linear equations for $d_0,\ldots, d_{\tilde{M}-1}$ and using \eqref{cbo}, we obtain
\eqref{dcoeff}.
This proves \eqref{accuracy} for accuracy order $\tilde{M}$.
\end{proof}

We make some remarks on \cref{thm:interior}.
First, from the proof of \cref{thm:interior}, we see that \cref{thm:interior} finds all the possible compact FDMs with accuracy order $\tilde{M}$ and numerical dispersion order $M$.
Because \eqref{accuracy} for accuracy order $\tilde{M}$ automatically implies \eqref{numdispersion} for numerical dispersion order $M$ with $M=\tilde{M}$, we often take $M=\tilde{M}$.
Also, $\kappa^2$ can be replaced by $-\kappa^2$ (i.e., $\kappa$ can be complex-valued).
Second, if $E_0$ and $E_1$ in \eqref{E12} have closed forms,
then we can have numerical dispersion order $M=\infty$ for a pollution free scheme
by selecting $c_1(h)=-1/E_1(h), c_{-1}(h)=-1/E_1(-h)$ and
$c_0(h)=E_0(h)/E_1(h)+E_0(-h)/E_1(-h)$
in \eqref{cbo}.
In particular, for constant functions $a$ and $\kappa^2$, we
observe
\be \label{constantcoeff:E}
E_0(h)=\cos(\tilde{h}),\quad E_1(h)=\tilde{h}^{-1}\sin(\tilde{h})
\quad \mbox{with}\quad \tilde{h}:=h\kappa/\sqrt{a}
\ee
and for $\ell\in \NN$,
\be \label{constantcoeff:F}
F_{2\ell}(h)=\frac{\cos(\tilde{h})-
\sum_{j=0}^\ell \frac{(-1)^j}{(2j)!}\tilde{h}^{2j}}{
(-1)^{\ell+1}  \tilde{h}^{2\ell+2}a},\quad F_{2\ell+1}(h)=\frac{\sin(\tilde{h})-\sum_{j=0}^\ell \frac{(-1)^j}{(2j+1)!} \tilde{h}^{2j+1}}{
(-1)^{\ell+1} \tilde{h}^{2\ell+3}a}.
\ee
This pollution free scheme coincides with that of \cite{WL11,WW14}. In the literature, a dispersion correction procedure for 1D homogeneous Helmholtz equation with (piecewise) constant wave numbers also exists \cite{cgx18, CGX20, EG13}. However, the procedure as presented uses the standard second order FDM, which itself is not a pollution free scheme. The correction solely comes from modifying the original wave number. More specifically, the method involves inserting the exact homogeneous solution of the 1D Helmholtz equation on the real line into the standard second order FDM to obtain the modified wave number.


\subsection{Compact stencils for boundary points}
\label{subsec:fdm:bdry}

We now handle the case that the base point is one of the endpoints. It is important that a compact FDM should achieve the same accuracy order and numerical dispersion order at the endpoints as it does at interior points. The following result answers this question. For simplicity, we only handle
the boundary condition at a base point $\xb$ from its right side, while the treatment for the boundary condition at $\xb$ from its left side is similar through symmetry. Because we shall handle piecewise smooth coefficients, let us consider a general boundary condition at $\xb\in [0,1)$ from its right side.
For $j\in \NN$ and a function $f(x)$, $f^{(j)}(\xb+):=\lim_{x\to \xb^+} f^{(j)}(x)$ and $f^{(j)}(\xb-):=\lim_{x\to \xb^-} f^{(j)}(x)$ for one-sided derivatives.

\begin{theorem}\label{thm:bdry}
Suppose that $a, \kappa^{2}, f$ in \eqref{heqnL} are smooth functions. Let $M,\tilde{M}$ be positive integers with $M\ge \tilde{M}$. Let $0<h<1$
and the boundary condition at $\xb\in [0,1)$ with $(\xb,\xb+h)\subset (0,1)$.
Suppose that the boundary condition at $\xb$ for the right side of $\xb$ is given by
\be \label{bdryatxb}
\mathcal{B}^+ u(\xb):=\gl_0 u(\xb+)+\gl_1 u'(\xb+) \quad \mbox{with}\quad \gl_0, \gl_1\in \C.
\ee
Consider the discretization stencil of a compact FDM for
$\mathcal{L}u:=[a(x)u'(x)]'+\kappa^2(x)u(x)=f(x)$
at the base point $\xb$, from the right side of $\xb$ with the above boundary condition, below
\begin{equation} \label{Lhfd:bdry} \mathcal{L}^{\mathcal{B}^+}_{h}u(\xb):= h^{-1}[c_{0}^{\mathcal{B}^+}(h) u(\xb)+c_1^{\mathcal{B}^+}(h)u(\xb+h)] - \sum_{\ell=0}^{\tilde{M}-2} d_{\ell}^{\mathcal{B}^+}(h) h^{\ell+1} f^{(\ell)}(\xb+),
\end{equation}
where $c_0^{\mathcal{B}^+}, c_1^{\mathcal{B}^+}$ and $d_\ell^{\mathcal{B}^+}$ are smooth functions of $h$ for $\ell=0,\ldots,\tilde{M}-2$. Suppose that
\begin{equation}\label{cbo:bdry}
c_{1}^{\mathcal{B}^+}(h)=\frac{\gl_1}{E_1(h)}+\bo(h^M),
\quad
c_{0}^{\mathcal{B}^+}(h)=h \gl_0-
c_1^{\mathcal{B}^+}(h)E_0(h)+ \bo(h^{M+1})
\end{equation}
and
\be \label{dcoeff:bdry}
d_\ell^{\mathcal{B}^+} (h)= c_1^{\mathcal{B}^+}(h)
F_\ell(h)
+\bo(h^{\tilde{M}-\ell-1}),\qquad
\ell=0,\dots,\tilde{M}-2,
\ee
as $h\to 0$, where $E_{0}, E_{1}, F_{\ell},\ell\in \NN$ are given in \eqref{E12} of \cref{prop:span} and are determined by $a^{(n)}(\xb+)$ and $[\kappa^2]^{(n)}(\xb+)$ for $n\in \NN$.
Then the discretization stencil of the compact finite difference scheme at the base point $x_b$ with the boundary condition in \eqref{bdryatxb} from the right side of $x_b$  satisfies
%
\be \label{numdispersion:bdry}
c_0^{\mathcal{B}^+}(h) u(\xb)+c_1^{\mathcal{B}^+}(h) u(\xb+h)= \bo(h^{M}),\qquad h\to 0
\ee
for every solution $u$ of $\mathcal{L}u=0$, and
%
\be \label{accuracy:bdry}
\mathcal{L}^{\mathcal{B}^+}_{h}u(\xb)
-\mathcal{B}^+ u(\xb)=
\bo(h^{\tilde{M}}),\qquad h\to 0
\ee
for every solution $u$ of $\mathcal{L}u=f$.
\end{theorem}

\begin{proof}
The proof is similar to but easier than the proof of \cref{thm:interior} by using \eqref{utay}, which implies
\be \label{u:deriv:0}
u'(\xb)=\frac{1}{h E_1(h)} u(\xb+h)-\frac{E_0(h)}{h E_1(h)} u(\xb)-\sum_{\ell=0}^\infty \frac{h^{\ell+1}}{E_1(h)} f^{(\ell)}(\xb +) F_\ell(h).
\ee
Since $\mathcal{B}^+ u(\xb)=\gl_0 u(\xb+)+\gl_1 u'(\xb+)$, using \eqref{u:deriv:0} we obtain
\be \label{B+}
\mathcal{B}^+ u(\xb)
=h^{-1} \left[
\frac{\gl_1}{E_1(h)} u(\xb+h)+\left(\gl_0 h-\frac{\gl_1 E_0(h)}{E_1(h)}\right) u(\xb)\right]-\sum_{\ell=0}^\infty \frac{\gl_1}{E_1(h)} h^{\ell+1} f^{(\ell)}(\xb +) F_\ell(h).
\ee
Now the claim follows directly by using \eqref{B+} and \eqref{Lhfd:bdry}.
\end{proof}

If $E_0$ and $E_1$ in \eqref{E12} have closed forms, then we can achieve numerical dispersion order $M=\infty$  for pollution free by selecting $c_1(h)=\frac{\gl_1}{E_1(h)}$ and
$c_0(h)=h\gl_0-\frac{E_0(h)}{E_1(h)}\gl_1$ in \eqref{cbo:bdry} of \cref{thm:bdry}.
In particular, for
constant functions $a$ and $\kappa^2$,
$E_0, E_1, F_\ell$ are given in
\eqref{constantcoeff:E} and \eqref{constantcoeff:F}. As before, this pollution free scheme coincides with that of \cite{WL11,WW14} when the boundary condition takes the form of $\lambda_0=1$ and $\lambda_1=0$ in \eqref{bdryatxb} or  $\lambda_0=-i\kappa$ and $\lambda_1=-1$ in \eqref{bdryatxb}.

\subsection{Compact stencils for piecewise smooth coefficients}
\label{subsec:fdm:piece}

We now discuss
piecewise smooth coefficients $a,\kappa^2$ and $f$. Assume that $a,\kappa^2,f$ may
have a breaking/branch point $x_c:=
\xb-\theta h \in (\xb-h,\xb+h)$ with $\theta\in (-1,1)$ such that they
are smooth on $(\xb-h,x_c)$ and $(x_c,\xb+h)$, but they may be discontinuous at $x_c$. We also assume that all the one-sided derivatives of $a,\kappa^2,f$ exist at $x_c$ and assume $\theta\in [0,1)$ for simplicity.
To solve the Dirac assisted local problems in \eqref{localproblem:dirac} of \cref{alg:dat} for DAT, we also assume that $w(\delta_{x_c})$ is the weight of the Dirac distribution $\delta_{x_c}$ in the source term $f$.
We can generalize
\cref{thm:interior} by considering the following discretization stencil at $\xb$:
{\footnotesize{\[
\mathring{\mathcal{L}}_{h}u(\xb):= h^{-2}[c_{-1}(h) u(\xb-h)+c_0(h)u(\xb)+c_1(h)u(\xb+h)]
-h^{-1} d_w(h) w(\delta_{x_c})
- \sum_{\ell=0}^{\tilde{M}-1} h^{\ell} \left(d_{\ell}^{+}(h) f^{(\ell)}(x_c+) +  d_{\ell}^{-}(h) f^{(\ell)}(x_c-)\right),
\]}}
where $c_{-1},c_0,c_1,d_w$, and $d_\ell^{+},d_\ell^{-},\ell=0,\ldots,\tilde{M}-1$ are smooth functions of $h$ satisfying
\begin{align*}
&c_{-1}(h) E_{0,-}((\theta -1)h)+c_0(h) E_{0,+}(\theta h)+c_1(h) E_{0,+}((1+\theta)h)=\bo(h^{M+2}),\\
&c_{-1}(h) (\theta-1) E_{1,-}((\theta-1)h)\tfrac{a(x_c+)}{a(x_c-)}
+c_0(h)\theta E_{1,+}(\theta h)+
c_1(h) (\theta+1) E_{1,+}((1+\theta)h)=\bo(h^{M+1}),\\
&d_w(h)=c_{-1}(h) (1-\theta) E_{1,-}((\theta-1)h) \tfrac{1}{a(x_c-)}
+\bo(h^{M+1}),\quad h\to 0,
\end{align*}
and for $\ell=0,\ldots, \tilde{M}-1$,
\begin{align*}
&d_\ell^+(h)= c_0(h) \theta^{\ell+2} F_{\ell,+}(\theta h)+c_1(h) (\theta+1)^{\ell+2} F_{\ell,+}((\theta+1)h)+\bo(h^{\tilde{M}-\ell}),\\
&d_\ell^-(h)=c_{-1}(h) (\theta-1)^{\ell+2} F_{\ell,-}((\theta-1)h)+\bo(h^{\tilde{M}-\ell}),\quad h\to 0,
\end{align*}
where $E_{0,\pm}, E_{1,\pm}$ and $F_{\ell,\pm}$ are given in \cref{prop:span} at the point $x_c$ (instead of $\xb$) using $a^{(j)}(x_c\pm)$, $[\kappa^2]^{(j)}(x_c\pm)$ and $f^{(j)}(x_c\pm)$ accordingly.
Then the above discretization stencil has numerical dispersion order $M$ at $\xb$ by satisfying \eqref{numdispersion} and has accuracy order $\tilde{M}$ at $\xb$ by satisfying $\mathring{\mathcal{L}}_{h}u(\xb)=\bo(h^{\tilde{M}})$ as $h\to 0$ for every solution $u$ of $\mathcal{L} u=f$.
The proof of the above equations is very similar to that of \cref{thm:interior} but we expand $u(\xb-h), u(\xb)$ and $u(\xb+h)$ through \cref{prop:span} at $x_c$ instead of $\xb$ by noting $u(\xb-h)=u(x_c+(\theta-1)h)$, $u(\xb)=u(x_c+\theta h)$ and $u(\xb+h)=u(x_c+(\theta+1)h)$.
Then we link the two sides of $x_c$ through the transmission conditions $u(x_c+)=u(x_c-)$ and $a(x_c+) u'(x_c+)-a(x_c-)u'(x_c-)=w(\delta_{x_c})$.
If all coefficients $a,\kappa^2,f$ are smooth inside $(\xb-h,\xb+h)$, then $\mathcal{L}_h u(\xb)$ in \eqref{Lhfd} of
\cref{thm:interior} can be recovered through $\mathcal{L}_h(\xb)=\mathring{\mathcal{L}}_h u(\xb)+f(\xb)$ using $x_c=\xb$.
We shall not pursue this general issue further.
In the present paper, it suffices for us to only consider the special case that $x_c=\xb$, i.e., $\theta=0$. For $x_c=\xb$,  using the following special boundary operators at $\xb$:
\be \label{bdry:breaking}
\mathcal{B}^+ u(\xb):=u'(\xb+) \quad \mbox{and}\quad \mathcal{B}^- u(\xb):=u'(\xb-),
\ee
instead of using $\mathring{\mathcal{L}}_h u(\xb)$ we can deduce a compact stencil at the base point $\xb$ from \cref{thm:bdry} for \eqref{bdry:breaking} that
\be \label{join:cond}
\mathcal{L}_h u=
\tfrac{2a(\xb-)}{a(\xb+)+a(\xb-)}\mathcal{L}^{\mathcal{B}^-}_h u(\xb)-\tfrac{2a(\xb+)}{a(\xb+)+a(\xb-)}\mathcal{L}^{\mathcal{B}^+}_h u(\xb)
=-\tfrac{2w(\delta_{\xb})}{a(\xb+)+a(\xb-)}
\ee
at the base point $\xb$, where
$w(\delta_{\xb})$ is the weight of the Dirac distribution $\delta_{\xb}$ in the source term $f$.

\subsection{A concrete example of finite difference schemes for $M=8$}\label{subsec:fdm:6}

For the convenience of the reader, here we provide details about how to obtain concrete compact finite difference schemes with $M$th order accuracy and numerical dispersion as discussed in \cref{subsec:fdm:interior,subsec:fdm:bdry,subsec:fdm:piece}. In particular, we provide details for $M=2,4,6,8$.
We first discuss how to compute $E_0, E_1, F_\ell, \ell\in \N\cup\{0\}$ as defined in \eqref{E12}.
Using \eqref{recursive} and taking derivative on both sides of \eqref{uj},
we observe that the coefficients $E_{j,0}$, $E_{j,1}$ and $F_{j,\ell},\ell=0,\ldots,j-2$ at a base point $\xb$ in \cref{prop:span} can be recursively obtained by
\[
E_{j+1,0}=E_{j,0}'-\frac{\kappa^2}{a} E_{j,1},\quad
E_{j+1,1}=E_{j,0}+E_{j,1}'-\frac{a'}{a}E_{j,1},
\quad
F_{j+1,\ell}=F_{j,\ell}'+F_{j,\ell-1},\quad j\ge 2, \quad \ell=0,\ldots,j-1
\]
with the initial values
\be \label{Ej:initial}
E_{2,0}:=-\frac{\kappa^2}{a},\quad E_{2,1}:=-\frac{a'}{a}, \quad \mbox{and} \quad F_{2,0}:=\frac{1}{a},
\ee
where we used the convention that $F_{j,-1}:=\frac{E_{j,1}}{a}$ and $F_{j,\ell}:=0$ for all $\ell>j-2$.
Note that $E_0(0)=E_1(0)=1$ and $F_\ell(0)=F_{\ell+2,\ell}=F_{2,0}=\frac{1}{a(\xb)}$ for all $\ell\in \NN$.

Let $M=\tilde{M}\in 2\N$. At an interior point $\xb$ we obtain from  \eqref{Lhfd} of \cref{thm:interior} that
\[
c_{-1}(h) u(\xb-h)+c_0(h)u(\xb)+c_1(h) u(\xb+h)=h^2 f(\xb)+\sum_{\ell=0}^{M-2} d_\ell(h) h^{\ell+2} f^{(\ell)}(\xb) + \bo(h^{M+2}),
\]
as $h\to 0$, where one particular choice of $c_{-1},c_0, c_1$ satisfying \eqref{cbo} of \cref{thm:interior} is given by
\be \label{cbo:special}
c_{-1}(h):=-\EE_1^{M-1}(-h),\;\;
c_{1}(h):=-\EE_1^{M-1}(h),\;\;
c_{0}(h):= \EE_1^{M-1}(h)\EE_0^{M}(h)
+\EE_1^{M-1}(-h)\EE_0^{M}(-h),
\ee
and
the corresponding $d_{\ell}, \ell=0,\ldots,M-2$ in \eqref{dcoeff} are given by
\be \label{dcoeff:special}
d_{\ell}(h):=- \delta_{0,\ell}
-\EE_1^{M-1}(h) \EF_\ell^{M-\ell-2}(h)
-(-1)^\ell \EE_1^{M-1}(-h) \EF_\ell^{M-\ell-2}(-h),
\ee
where
$\EE_0^n, \EE_1^n, F_\ell^n, n\in \N$ are the unique polynomials (in terms of $h$) of degree $n$  satisfying
\be \label{E01M}
\EE_0^n(h)=E_0(h)+\bo(h^{n+1}),\quad
\EE_1^n(h):=1/E_1(h)+\bo(h^{n+1}),\quad
\EF_\ell^n(h)=F_\ell(h)+\bo(h^{n+1}),\quad
\ee
as $h\to 0$. Notice that $d_{M-1}(h)=0$ and $E_1(0)=1$.
Observing $c_1(0)+(-1)^{M-1} c_{-1}(0)=0$ for even $M\in 2\N$,
one can directly check that the above choice in \eqref{cbo:special} and \eqref{dcoeff:special} satisfies all the conditions in
\eqref{cbo} and \eqref{dcoeff} with
$\alpha(h)=1-\beta h^M$, where $\beta$ is the coefficient of $h^M$ in the Taylor series of $\frac{1}{E_1(h)}$ at $h=0$.
Note that $c_{-1},c_0, c_1$ in \eqref{cbo:special}
and
$d_{\ell}, \ell=0,\ldots,M-2$ in
\eqref{dcoeff:special} only depend on $a,a',\ldots,a^{(M-1)}$,
$\kappa^2,[\kappa^2]',\ldots,[\kappa^2]^{(M-2)}$ and $f, f',\ldots,f^{(M-2)}$.
We can also obtain stencils for odd integers $M=\tilde{M}$; however
$c_1(0)+(-1)^{M-1} c_{-1}(0)\ne 0$ and consequently, we have to use $\EE_1^{M}$ instead of $\EE_1^{M-1}$, $\EE_0^{M+1}$ instead of $\EE_0^{M}$, and $\EF_\ell^{M-\ell-1}$ instead of $\EF_\ell^{M-\ell-2}$ in \eqref{cbo:special}-\eqref{dcoeff:special}.

Define $a_j:=a^{(j)}(x_b)$, $\kappa_j:=[\kappa^2]^{(j)}(x_b)$ and
$f_j:=f^{(j)}(x_b)$.
For $M=\tilde{M}=8$, we explicitly have
{\scriptsize
\begin{align*}
\EE_{1}^{7} & = 1 + \tfrac{ha_1}{2a_0} + \left(2a_0 a_2 + 2 a_0 \kappa_0 -a_1^2\right) \tfrac{h^2}{12 a_0^2} + \left(a_3{a_0}^2+2\kappa_{1}{a_0}^2-2a_1a_2a_0+a_1^3
\right)\tfrac{h^3}{24a_0^3} + \left(6 a_{4}a_{0}^{3} + 18\kappa_{2}a_{0}^{3}
-18a_1 a_3 a_0^2 - 6 a_1 \kappa_1 a_0^2
\right.
\\
& \quad \left.
- 16 a_2^2 a_0^2 - 2 a_2 \kappa_0 a_0^2 + 14 \kappa_0^2 a_0^2 + 46a_1^{2} a_2 a_0- 2 a_1^{2}\kappa_0 a_0 - 19a_1^4\right)\tfrac{h^4}{720a_0^4} +
\left(2a_5 a_0^4 + 8 \kappa_3 a_0^4 - 8 a_1 a_4 a_0^3
- 6 a_1 \kappa_2 a_0^3
- 20 a_2 a_3 a_0^3
\right.
\\
& \quad \left.
- 8 a_2 \kappa_1 a_0^3
- 2 a_3 \kappa_0 a_0^3 + 28 \kappa_0 \kappa_1 a_0^3
+ 29 a_1^2 a_3 a_0^2 +4 a_1^2 \kappa_1 a_0^2
+ 48 a_1 a_2^2 a_0^2 -2 a_1 a_2 \kappa_0 a_0^2
- 14 a_1 \kappa_0^2 a_0^2 - 78 a_1^3 a_2 a_0 + 4 a_1^3 \kappa_0 a_0
\right.
\\
& \quad \left.
 +27a_1^5
\right)\tfrac{h^5}{1440 a_0^5} +
\left(
12a_6 a_0^5 + 60 \kappa_4 a_0^5
-60a_1 a_5 a_0^4 - 72a_1 \kappa_3 a_0^4 - 192a_2 a_4 a_0^4
-156a_2 \kappa_2 a_0^4
-135a_3^2 a_0^4 -120a_3 \kappa_1 a_0^4
-24a_4 \kappa_0 a_0^4
\right.
\\
& \quad \left.
+348\kappa_0 \kappa_2 a_0^4
+300\kappa_1^2 a_0^4 +282a_1^2 a_4 a_0^3
+114a_1^2 \kappa_2 a_0^3 +204a_1 a_2 \kappa_1 a_0^3
+1296a_1 a_2 a_3 a_0^3
-708a_1 \kappa_0 \kappa_1 a_0^3 +352a_2^3 a_0^3
-12a_2^2 \kappa_0 a_0^3
\right.
\\
& \quad \left.
 -240a_2 \kappa_0^2 a_0^3 +124\kappa_0^3 a_0^3 -1056a_1^3 a_3 a_0^2 -66a_1^3 \kappa_1 a_0^2 -2544a_1^2 a_2^2 a_0^2 +198a_1^2 a_2 \kappa_0 a_0^2
 +354a_1^2 \kappa_0^2 a_0^2 +2910a_1^4 a_2 a_0
 -150a_1^4 \kappa_0 a_0
\right.
\\
& \quad \left.
 -863a_1^6
\right)\tfrac{h^6}{60480a_0^6} + \left(
3 a_7 a_0^6 +18\kappa_5 a_0^6
- 18a_1 a_6 a_0^5 -30 a_1 \kappa_4 a_0^5
-70a_2 a_5 a_0^5 -88 a_2 \kappa_3 a_0^5 -126 a_3 a_4 a_0^5 -108 a_3 \kappa_2 a_0^5
\right.
\\
& \quad \left.
-60 a_4 \kappa_1 a_0^5 -10 a_5 \kappa_0 a_0^5 +152 \kappa_0 \kappa_3 a_0^5
+360 \kappa_1 \kappa_2 a_0^5
+104 a_1^2 a_5 a_0^4 +74 a_1^2 \kappa_3 a_0^4 +610 a_1 a_2 a_4 a_0^4 +252 a_1 a_2 \kappa_2 a_0^4
+423 a_1 a_3^2 a_0^4
\right.
\\
& \quad \left.
 + 156 a_1 a_3 \kappa_1 a_0^4 +10 a_1 a_4 \kappa_0 a_0^4
 -468 a_1 \kappa_0 \kappa_2 a_0^4
 -420 a_1 \kappa_1^2 a_0^4 +686 a_2^2 a_3 a_0^4 +108 a_2^2 \kappa_1 a_0^4 +4 a_2 a_3 \kappa_0 a_0^4 -600 a_2 \kappa_0 \kappa_1 a_0^4
\right.
\\
& \quad \left.
 -142 a_3 \kappa_0^2 a_0^4
+372 \kappa_0^2 \kappa_1 a_0^4
 -500 a_1^3 a_4 a_0^3
-126 a_1^3 \kappa_2 a_0^3 -3316 a_1^2 a_2 a_3 a_0^3 -240 a_1^2 a_2 \kappa_1 a_0^3 +86 a_1^2 a_3 \kappa_0 a_0^3 +948 a_1^2 \kappa_0 \kappa_1 a_0^3
\right.
\\
& \quad \left.
-1760 a_1 a_2^3 a_0^3
 +168 a_1 a_2^2 \kappa_0 a_0^3 +600 a_1 a_2 \kappa_0^2 a_0^3
-248 a_1 \kappa_0^3 a_0^3 +1899 a_1^4 a_3 a_0^2 +48 a_1^4 \kappa_1 a_0^2 +6000 a_1^3 a_2^2 a_0^2
-522 a_1^3 a_2 \kappa_0 a_0^2
\right.
\\
& \quad \left.
-474 a_1^3 \kappa_0^2 a_0^2 -5310 a_1^5 a_2 a_0 +264 a_1^5 \kappa_0 a_0 +1375 a_1^7
\right)\tfrac{h^7}{120960 a_0^7},\\
\EE_{0}^{8} & = 1 - \tfrac{\kappa_{0}h^{2}}{2a_{0}}
+\left(-\kappa_{{1}}a_{{0}}
+2\,a_{{1}}\kappa_{{0}}\right){\tfrac {{h}^{3}}{6{a_{{0}}}^{2}}}
+\left(-\kappa_{{2}}{a_{{0}}}^{2}
+3\,a_{{1}}\kappa_{{1}}a_{{0}}
+3\,a_{{2}}\kappa_{{0}}a_{{0}}
+{\kappa_{{0}}}^{2}a_{{0}}
-6\,{a_{{1}}}^{2}\kappa_{{0}}\right){\tfrac {{h}^{4}}{24{a_{{0}}}^{3}}}
+ \left(-\kappa_{{3}}{a_{{0}}}^{3}
+4\,a_{{1}}\kappa_{{2}}{a_{{0}}}^{2}
\right.
\\
& \quad
\left.
+6\,a_{{2}}\kappa_{{1}}{a_{{0}}}^{2}
+4\,a_{{3}}\kappa_{{0}}{a_{{0}}}^{2}
+4\,\kappa_{{0}}\kappa_{{1}}{a_{{0}}}^{2}
-12\,{a_{{1}}}^{2}\kappa_{{1}}a_{{0}}
-24\,a_{{1}}a_{{2}}\kappa_{{0}}a_{{0}}
-6\,a_{{1}}{\kappa_{{0}}}^{2}a_{{0}}
+24\,{a_{{1}}}^{3}\kappa_{{0}}
\right) {\tfrac {{h}^{5}}{120\,{a_{{0}}}^{4}}}
 + \left(-\kappa_{{4}}{a_{{0}}}^{4}
+5\,a_{{1}}\kappa_{{3}}{a_{{0}}}^{3}
\right.
\\
& \left. \quad
+10\,a_{{2}}\kappa_{{2}}{a_{{0}}}^{3}
+10\,a_{{3}}\kappa_{{1}}{a_{{0}}}^{3}
+5\,a_{{4}}\kappa_{{0}}{a_{{0}}}^{3}
+7\,\kappa_{{0}}\kappa_{{2}}{a_{{0}}}^{3}
+4\,{\kappa_{{1}}}^{2}{a_{{0}}}^{3}
-20\,{a_{{1}}}^{2}\kappa_{{2}}{a_{{0}}}^{2}
-60\,a_{{1}}a_{{2}}\kappa_{{1}}{a_{{0}}}^{2}
-40\,a_{{1}}a_{{3}}\kappa_{{0}}{a_{{0}}}^{2}
\right.
\\
& \quad \left.
-31\,a_{{1}}\kappa_{{0}}\kappa_{{1}}{a_{{0}}}^{2}
-30\,{a_{{2}}}^{2}\kappa_{{0}}{a_{{0}}}^{2}
-13\,a_{{2}}{\kappa_{{0}}}^{2}{a_{{0}}}^{2}
-{\kappa_{{0}}}^{3}{a_{{0}}}^{2}
+60\,{a_{{1}}}^{3}\kappa_{{1}}a_{{0}}
-120\,{a_{{1}}}^{4}\kappa_{{0}}
+180\,{a_{{1}}}^{2}a_{{2}}\kappa_{{0}}a_{{0}}
+36\,{a_{{1}}}^{2}{\kappa_{{0}}}^{2}a_{{0}}
\right){\tfrac {{h}^{6}}{720\,{a_{{0}}}^{5}}}
\\
& \quad
+ \left(
-\kappa_5 a_0^5 +6a_1 \kappa_4 a_0^4
+15a_2 \kappa_3 a_0^4
+20a_3 \kappa_2 a_0^4
+15a_4 \kappa_1 a_0^4
+6a_5 \kappa_0 a_0^4
+11\kappa_0 \kappa_3 a_0^4 +15\kappa_1 \kappa_2 a_0^4
-30a_1^2 \kappa_3 a_0^3 -120a_1 a_2 \kappa_2 a_0^3
\right.
\\
& \quad \left.
-120a_1 a_3 \kappa_1 a_0^3
-60a_1 a_4 \kappa_0 a_0^3 -66a_1 \kappa_0 \kappa_2 a_0^3 -39a_1 \kappa_1^2 a_0^3
-90a_2^2 \kappa_1 a_0^3
-120a_2 a_3 \kappa_0 a_0^3 -81a_2 \kappa_0 \kappa_1 a_0^3
-24a_3 \kappa_0^2 a_0^3 -9\kappa_0^2 \kappa_1 a_0^3
\right.
\\
& \quad \left.
+120a_1^3 \kappa_2 a_0^2
+540a_1^2 a_2 \kappa_1 a_0^2 +360a_1^2 a_3 \kappa_0 a_0^2
+228a_1^2 \kappa_0 \kappa_1 a_0^2
+540a_1 a_2^2 \kappa_0 a_0^2
+192 a_1 a_2 \kappa_0^2 a_0^2 +12a_1 \kappa_0^3 a_0^2
-360a_1^4 \kappa_1 a_0
\right.
\\
& \quad \left.
-1440a_1^3 a_2 \kappa_0 a_0 -240a_1^3 \kappa_0^2 a_0 +720a_1^5 \kappa_0
\right)\tfrac{h^7}{5040a_0^6} +
\left(
-\kappa_6 a_0^6
+7a_1 \kappa_5 a_0^5 +21a_2 \kappa_4 a_0^5 +35a_3 \kappa_3 a_0^5 +35a_4 \kappa_2 a_0^5 +21a_5 \kappa_1 a_0^5
\right.
\\
& \quad \left.
+7a_6 \kappa_0 a_0^5 +16\kappa_0 \kappa_4 a_0^5
+26\kappa_1 \kappa_3 a_0^5
+15\kappa_2^2 a_0^5
-42a_1^2 \kappa_4 a_0^4
-210a_1 a_2 \kappa_3 a_0^4 -280a_1 a_3 \kappa_2 a_0^4 -210a_1 a_4 \kappa_1 a_0^4
-84a_1 a_5 \kappa_0 a_0^4
\right.
\\
& \quad \left.
-122a_1 \kappa_0 \kappa_3 a_0^4
-174a_1 \kappa_1 \kappa_2 a_0^4
-210a_2^2 \kappa_2 a_0^4 -420a_2 a_3 \kappa_1 a_0^4 -210a_2 a_4 \kappa_0 a_0^4
-202a_2 \kappa_0 \kappa_2 a_0^4
-120a_2 \kappa_1^2 a_0^4 -140a_3^2 \kappa_0 a_0^4
\right.
\\
& \quad \left.
-40a_4 \kappa_0^2 a_0^4
-174a_3 \kappa_0 \kappa_1 a_0^4
  -22\kappa_0^2 \kappa_2 a_0^4
-28\kappa_0 \kappa_1^2 a_0^4
+210a_1^3 \kappa_3 a_0^3 +1260a_1^2 a_2 \kappa_2 a_0^3 +1260a_1^2 a_3 \kappa_1 a_0^3
+630a_1^2 a_4 \kappa_0 a_0^3
\right.
\\
& \quad \left.
+345a_1^2 \kappa_1^2 a_0^3
+572a_1^2 \kappa_0 \kappa_2 a_0^3
+1890a_1 a_2^2 \kappa_1 a_0^3 +2520a_1 a_2 a_3 \kappa_0 a_0^3 +1422a_1 a_2 \kappa_0 \kappa_1 a_0^3
+418a_1 a_3 \kappa_0^2 a_0^3 +130a_1 \kappa_0^2 \kappa_1 a_0^3
\right.
\\
& \quad \left.
+630a_2^3 \kappa_0 a_0^3
+303a_2^2 \kappa_0^2 a_0^3 +34a_2 \kappa_0^3 a_0^3
+\kappa_0^4 a_0^3 -840a_1^4 \kappa_2 a_0^2
-5040a_1^3 a_2 \kappa_1 a_0^2 -3360a_1^3 a_3 \kappa_0 a_0^2
-1800a_1^3 \kappa_0 \kappa_1 a_0^2
\right.
\\
& \quad \left.
-7560a_1^2 a_2^2 \kappa_0 a_0^2
-2280a_1^2 a_2 \kappa_0^2 a_0^2 -120a_1^2 \kappa_0^3 a_0^2
+2520a_1^5 \kappa_1 a_0 +12600a_1^4 a_2 \kappa_0 a_0 +1800a_1^4 \kappa_0^2 a_0
-5040a_1^6 \kappa_0
\right)\tfrac{h^{8}}{40320 a_0^7},\\
%
%
%
\EF_0^6(h) & = \tfrac{1}{2a_0} - \tfrac{a_1 h}{3 a_0^2} + \left(-3a_2 a_0 -\kappa_0 a_0 +6a_1^2\right)\tfrac{h^2}{24a_0^3} + \left(-4a_3 a_0^2 -3\kappa_1 a_0^2 +24a_1 a_2 a_0 +6a_1 \kappa_0 a_0 -24a_1^3\right)\tfrac{h^3}{120a_0^4}
+ \left(-5a_4 a_0^3 -6\kappa_2 a_0^3
\right.
\\
& \quad \left.
+40a_1 a_3 a_0^2 +23a_1 \kappa_1 a_0^2 +30a_2^2 a_0^2 +13a_2 \kappa_0 a_0^2 +\kappa_0^2 a_0^2 -180a_1^2 a_2 a_0 -36a_1^2 \kappa_0 a_0 +120a_1^4\right)\tfrac{h^{4}}{720 a_0^5} + \left(-3a_0^4 a_5 -5a_0^4 \kappa_3
\right.
\\
& \quad \left.
 +30a_0^3 a_1 a_4 +28a_0^3 a_1 \kappa_2
+60a_0^3 a_2 a_3 +30a_0^3 a_2 \kappa_1
+12a_0^3 a_3 \kappa_0 +4a_0^3 \kappa_0 \kappa_1
-180a_0^2 a_1^2 a_3 -84a_0^2 a_1^2 \kappa_1 -270a_0^2 a_1 a_2^2
-6a_0^2 a_1 \kappa_0^2
\right.
\\
& \quad \left.
-96a_0^2 a_1 a_2 \kappa_0
 +720a_0 a_1^3 a_2 +120a_0 a_1^3 \kappa_0 -360a_1^5\right)\tfrac{h^5}{2520 a_0^6} + \left(-7a_0^5 a_6 -15a_0^5 \kappa_4
+84a_0^4 a_1 a_5 +110a_0^4 a_1 \kappa_3
 +210a_0^4 a_2 a_4
\right.
\\
& \quad \left.
+171a_0^4 a_2 \kappa_2 +140a_0^4 a_3^2
+129a_0^4 a_3 \kappa_1 +40a_0^4 a_4 \kappa_0
+21a_0^4 \kappa_0 \kappa_2 +18 a_0^4 \kappa_1^2 -630a_0^3 a_1^2 a_4
-482a_0^3 a_1^2 \kappa_2
-2520a_0^3 a_1 a_2 a_3
\right.
\\
& \quad \left.
-1047a_0^3 a_1 a_2 \kappa_1
-418a_0^3 a_1 a_3 \kappa_0
-115a_0^3 a_1 \kappa_0 \kappa_1-630a_0^3 a_2^3 -303a_0^3 a_2^2 \kappa_0 -34a_0^3 a_2 \kappa_0^2 -a_0^3 \kappa_0^3 +3360a_0^2 a_1^3 a_3
+1320a_0^2 a_1^3 \kappa_1
\right.
\\
& \quad \left.
+7560a_0^2 a_1^2 a_2^2
+2280a_0^2 a_1^2 a_2 \kappa_0 +120a_0^2 a_1^2 \kappa_0^2 -12600a_0 a_1^4 a_2 -1800a_0 a_1^4 \kappa_0 +5040a_1^6\right)\tfrac{h^{6}}{40320 a_0^7},\\
\EF_1^5(h) & = \tfrac{1}{6 a_0} - \tfrac{a_1 h}{8 a_0^2} + \left(-6a_2 a_0 - \kappa_0 a_0 + 12 a_1^2\right)\tfrac{h^2}{120 a_0^3} + \left(-5a_3 a_0^2 -2\kappa_1 a_0^2 +30a_1 a_2 a_0 +4a_1 \kappa_0 a_0 -30a_1^3\right)\tfrac{h^3}{360 a_0^4} +
\left(-15a_0^3 a_4
\right.
\\
& \quad \left.
 -10a_0^3 \kappa_2 +120a_0^2 a_1 a_3 +39a_0^2 a_1 \kappa_1 +90a_0^2 a_2^2 +21a_0^2 a_2 \kappa_0
+a_0^2 \kappa_0^2 -540a_0 a_1^2 a_2 -60a_0 a_1^2 \kappa_0 +360 a_1^4\right)\tfrac{h^{4}}{5040 a_0^5}
\\
& \quad
+ \left(-21a_0^4 a_5 -20a_0^4 \kappa_3 +210a_0^3 a_1 a_4 +115a_0^3 a_1 \kappa_2
+420a_0^3 a_2 a_3 +120a_0^3 a_2 \kappa_1
+45a_0^3 a_3 \kappa_0
+10a_0^3 \kappa_0 \kappa_1 -1260a_0^2 a_1^2 a_3
\right.
\\
& \quad \left.
-345a_0^2 a_1^2\kappa_1 -1890a_0^2 a_1 a_2^2
-375a_0^2 a_1 a_2 \kappa_0 -15a_0^2 a_1 \kappa_0^2 +5040a_0 a_1^3 a_2 +480a_0 a_1^3 \kappa_0
 -2520 a_1^5\right)\tfrac{h^{5}}{40320 a_0^6},\\
\EF_2^4(h) & = \tfrac{1}{24 a_0} - \tfrac{a_1 h}{30 a_0^2} + \left(-10a_2 a_0 - \kappa_0 a_0 + 20 a_1^2\right)\tfrac{h^2}{720 a_0^3} + \left(-4a_0^2 a_3 -a_0^2 \kappa_1 +24a_0 a_1 a_2 +2a_0 a_1 \kappa_0
-24 a_1^3\right)\tfrac{h^{3}}{1008a_0^4}
+ \left(-35a_0^3 a_4
\right.
\\
& \quad \left.
-15a_0^3 \kappa_2 +280a_0^2 a_1 a_3 +59a_0^2 a_1 \kappa_1 +210a_0^2 a_2^2 +31a_0^2 a_2 \kappa_0
+a_0^2 \kappa_0^2 -1260a_0 a_1^2 a_2 -90a_0 a_1^2 \kappa_0
+840 a_1^4\right)\tfrac{h^{4}}{40230 a_0^5},
\\
\EF_3^3(h) & = \tfrac{1}{120 a_0}  - \tfrac{a_1 h}{144 a_0^2}
+\left(-15a_0 a_2 -a_0 \kappa_0 +30 a_1^2
\right)\tfrac{h^2}{5040 a_0^3} + \left(
-35a_0^2 a_3 -6a_0^2 \kappa_1 +210a_0 a_1 a_2 +12a_0 a_1 \kappa_0
-210 a_1^3\right)\tfrac{h^{3}}{40320 a_0^4},\\
\EF_4^2(h) & = \tfrac{1}{720 a_0} - \tfrac{a_1 h}{840 a_0^2}
+ \left(-21a_0 a_2 -a_0 \kappa_0 +42 a_1^2\right)\tfrac{h^2}{40320 a_0^3},
\qquad \EF_5^1(h)  = \tfrac{1}{5040 a_0} - \tfrac{a_1 h}{5760 a_0^2},\qquad
\EF_6^0(h)  = \tfrac{1}{40320 a_0}.
\end{align*}
}

Define $\EE_{0,\pm}^n, \EE_{1,\pm}^n, \EF_{\ell,\pm}^n$ for $n\in \NN$ to be just $\EE_{0}^n, \EE_{1}^n, \EF_{\ell}^n$ as in \eqref{E01M}, respectively but using $a_j=a^{(j)}(\xb\pm)$,
$\kappa_j=[\kappa^2]^{(j)}(\xb\pm)$ and  $f_j=f^{(j)}(\xb\pm)$.
Let $M=\tilde{M}\in 2\N$. For the left boundary condition $\mathcal{B}^+u(\xb)=\gl_0 u(\xb+)+\gl_1 u'(\xb+)$, we deduce from \eqref{Lhfd:bdry} of \cref{thm:bdry} that
\[
c_0^{\mathcal{B}^+}(h) u(\xb)+c_1^{\mathcal{B}^+}(h) u(\xb+h)=
h \mathcal{B}^+ u(\xb)+\sum_{\ell=0}^{M-2} d^{\mathcal{B}^+}_\ell(h) h^{\ell+2} f^{(\ell)}(\xb+)+\bo(h^{M+1}),
\]
as $h\to 0$, where
one particular choice of $c^{\mathcal{B}^+}_1$, $c^{\mathcal{B}^+}_0$, and $d^{\mathcal{B}^+}_\ell$ for $\ell=0,\ldots, M-2$ satisfying \eqref{cbo:bdry} and \eqref{dcoeff:bdry} of \cref{thm:bdry} are given by
\be \label{cbo:bdry:special}
\begin{split}
&c^{\mathcal{B}^+}_{1}(h):=\gl_1 \EE_{1,+}^{M-1}(h),\qquad
c^{\mathcal{B}^+}_{0}(h):= h\gl_0
-\gl_1 \EE_{1,+}^{M-1}(h)\EE_{0,+}^{M}(h),\\
&d^{\mathcal{B}^+}_\ell(h):=\gl_1 \EE_{1,+}^{M-1}(h)\EF_{\ell,+}^{M-\ell-2}(h),\qquad \ell=0,\ldots, M-2.
\end{split}
\ee
%
Similarly,
let the boundary condition at $\xb$ for the left side of $\xb$ be given by
\be \label{bdryatxb:left}
\mathcal{B}^- u(\xb):=\gl_0 u(\xb-)+\gl_1 u'(\xb-) \quad \mbox{with}\quad \gl_0, \gl_1\in \C.
\ee
By symmetry, we observe that the discretization
at the base point $\xb$ from the left side of $\xb$ is
\be \label{Lhfd:bdry:left}
c_{-1}^{\mathcal{B}^-}(h)u(\xb-h) +c_{0}^{\mathcal{B}^-}(h) u(\xb)
=h \mathcal{B}^- u(\xb)
+\sum_{\ell=0}^{M-2} d_{\ell}^{\mathcal{B}^-}(h) h^{\ell+2} f^{(\ell)}(\xb-)+\bo(h^{M+1}),
\ee
as $h\to 0$, which satisfies the corresponding relations in \eqref{numdispersion:bdry} and \eqref{accuracy:bdry} if
\be \label{cbo:bdry:special:left}
\begin{split}
&c^{\mathcal{B}^-}_{-1}(h):=-\gl_1 \EE_{1,-}^{M-1}(-h),\qquad
c^{\mathcal{B}^-}_{0}(h):= h\gl_0
+\gl_1 \EE_{1,-}^{M-1}(-h)\EE_{0,-}^{M}(-h),\\
&d^{\mathcal{B}^-}_\ell(h):=(-1)^{\ell+1} \gl_1 \EE_{1,-}^{M-1}(-h)\EF_{\ell,-}^{M-\ell-2}(-h),\qquad \ell=0,\ldots, M-2.
\end{split}
\ee
%
%
%
For the stencil used at the breaking/branch point $\xb$ such that $w(\delta_{\xb})$ is the weight of the Dirac distribution $\delta_{\xb}$ of the source term $f$, we deduce from \eqref{cbo:bdry:special} and \eqref{cbo:bdry:special:left} with $\lambda_0=0$ and $\lambda_1=1$
that
{\footnotesize{
\begin{align}\label{join:cond:stencil}
& -\alpha
\EE_{1,-}^{M-1}(-h) u(\xb-h)+
\left[
\alpha
\EE_{1,-}^{M-1}(-h) \EE_{0,-}^M(-h)+
\beta
\EE_{1,+}^{M-1}(h) \EE_{0,+}^M(h)\right] u(\xb)-
\beta
\EE_{1,+}^{M-1}(h)u(\xb+h)\\ \nonumber
& \quad = -h \gamma  w(\delta_{\xb})
-\sum_{\ell=0}^{M-2} h^{\ell+2}[
\beta \EE_{1,+}^{M-1}(h) \EF_{\ell,+}^{M-\ell-2}(h)
f^{(j)}(\xb+)+\alpha (-1)^\ell \EE_{1,-}^{M-1}(-h) \EF_{\ell,-}^{M-\ell-2}(-h)f^{(j)}(\xb-)] +\bo(h^{M+1}),
\end{align}
}}
as $h\to 0$, where $\alpha:=\tfrac{2a(\xb-)}{a(\xb+)+a(\xb-)}$,
$\beta:=\tfrac{2a(\xb+)}{a(\xb+)+a(\xb-)}$ and
$\gamma:=\frac{2}{a(\xb+)+a(\xb-)}$.

Finite difference schemes with lower accuracy orders $M=2,4,6$ can be easily obtained by truncating the above given $\EE_0^8, \EE_1^7, \EF_\ell^{6-\ell}$ accordingly. In the above compact FDM with accuracy order $M$ with $M=\tilde{M}\in 2\N$, we only need $a,a',\ldots,a^{(M-1)}, \kappa^2, [\kappa^2]',\ldots,[\kappa^2]^{(M-2)}$ and $f,f',\ldots,f^{(M-2)}$.

\section{Convergence of DAT Using Compact FDMs}
\label{sec:convergence}

In this section, we discuss the convergence of DAT in \cref{sec:dat} using compact FDMs described in \cref{thm:interior,thm:bdry} of \cref{sec:fdm}.
Let us first outline the notations and assumptions for our discussion in this section and for our numerical experiments in the next section.
Let $0=b_0<b_1<\cdots<b_{p}<b_{p+1}=1$ with $p\in \N\cup\{0\}$.
The coefficients $a, \kappa^2$ and $f$ in \eqref{heqnL} are piecewise smooth in the sense that they have uniformly continuous derivatives of orders up to $n_M$ on $(b_j, b_{j+1})$ for all $j=0,\ldots, p$ for certain given integer $n_M\in \N$ (see \cref{sec:fdm} for details). Note that these coefficients may be discontinuous on $(0,1)$ and we call the points $b_1,\ldots,b_{p}$ breaking/branch points. For simplicity of discussion, we assume a Dirichlet boundary condition at 0, while a Dirichlet, Neumann, or Robin boundary condition at 1. Let $\eu$ be the exact weak solution to \eqref{heqnL} with the boundary conditions in \eqref{heqnbdry}. Let $N\in \N$ and $0=x_0<x_1<\ldots<x_{N-1}<x_N=1$ for the computational mesh with the average mesh size $h:=N^{-1}$.
Let $\{u_N(x_j)\}_{j=0}^N$ be the approximated solution on knot points $\{x_j\}_{j=0}^N$. To study theoretical convergence rates and to evaluate the performance of DAT, we define
\be \label{def:2:0}
\|u_N-\eu\|_{\infty}:=\max_{0\le j \le N} |u_N(x_j)-\eu(x_j)|,
\quad
\|u_N-\eu\|_{2}^2:=
\sum_{j=0}^{N} h_j |u_N(x_j)-\eu(x_j)|^2
\ee
with  $h_j:=x_{j+1}-x_j$ and $x_{N+1}:=1$. Because $\|u_N-\eu\|_{2}\le \|u_N-\eu\|_{\infty}$ and
$\|u_N'-\eu'\|_{2}\le \|u_N'-\eu'\|_{\infty}$ always hold, we shall only discuss the convergence in $\infty$-norm instead of $2$-norm.
Throughout this section, positive constants $C, C_1, C_2$ are always independent of both matrix size $N$ and mesh size $h$. The computational mesh is assumed to be quasi-uniform, i.e., there exists $C>0$ independent of $h$ such that
$C^{-1} h\le h_j\le C h$ for all $j=0,\ldots,N$.
Note that the weak solution $\eu \in H^1(0,1)$ but $\eu'$ may be discontinuous at branch points on $(0,1)$, because the coefficients $a, \kappa^2, f$ are only piecewise smooth.
For convenience, every branch point $b_j$ is assumed to be a grid/knot point and the mesh on each piece $(b_j, b_{j+1})$ is uniform for all $j=0,\ldots, p$.
These restrictions could be dropped as we already discussed in \cref{subsec:fdm:piece} but they make our discussion here and implementation in \cref{sec:examples} much simpler.

For an $M$th order compact FDM in \cref{sec:fdm} with $\tilde{M}=M\in \N$, the stencil at $x_j$:
\be \label{fdm:stencil}
c_{j,-1}(h) u(x_{j-1})+c_{j,0}(h) u(x_j)+c_{j,1}(h) u(x_{j+1})=F_j(h),\qquad j=1,\ldots,N
\ee
is given in \cref{thm:interior,thm:bdry} as follows:
\begin{enumerate}
	\item[(1)] If $x_j$ is neither a branch point nor a boundary point, then \eqref{fdm:stencil} is given by \eqref{Lhfd} with \eqref{cbo} and \eqref{dcoeff} under the normalization condition $\alpha(0)=-1$. Because we impose a Dirichlet boundary condition at $0$, the known term $c_{1,-1}(h) u(x_0)$ is moved to and combined with $F_j$.
	\item[(2)] If $x_j$ is a branch point, then \eqref{fdm:stencil} is given by \eqref{join:cond} or more explicitly \eqref{join:cond:stencil}.
	\item[(3)] For $x_N$ (right boundary point), \eqref{fdm:stencil} is given by \eqref{Lhfd:bdry:left} with \eqref{cbo:bdry:special:left}. Note that $c_{N,N+1}(h)=0$. If we impose a Dirichlet boundary condition at $1$, the known term $c_{N-1,1}(h) u(x_N)$ is moved to and combined with $F_j$ such that the last term in \eqref{fdm:stencil} is $N-1$ (instead of $N$).
\end{enumerate}

The approximated numerical solution $u_N=\{u_N(x_j)\}_{j=0}^N$ is obtained by solving the linear system in \eqref{fdm:stencil} with $u=u_N$.
By \cref{thm:interior,thm:bdry}, the exact solution $\eu$ must satisfy
\be \label{fdm:stencil:ue}
c_{j,-1}(h) \eu(x_{j-1})+c_{j,0}(h) \eu(x_j)+c_{j,1}(h) \eu (x_{j+1}) =F_j(h)+R_j(h),\qquad j=1,\ldots,N,
\ee
where the local truncation error functions $R_j(h)$ resulted from Taylor approximation satisfy
\be \label{Rj}
|R_j(h)|\le
\begin{cases}
	C h^{M+2}, &\text{if $x_j\not \in \{b_0,\ldots, b_{p+1}\}$, i.e., $x_j$ is an interior point},\\
	C h^{M+1}, &\text{if $x_j\in \{b_0,\ldots, b_{p+1}$\}, i.e., $x_j$ is a branch point or a boundary point},
\end{cases}
\ee
where the constant $C$ is independent of $h$ and only depends on derivatives of $\eu, a,\kappa^2$ and $f$.
The error is then defined by $Q_j:=u_N(x_j)-\eu(x_j)$ at the knot point $x_j$ for $j=1,\ldots,N$.
By \eqref{fdm:stencil:ue} and  \eqref{Rj},
\be \label{stencil:not}
c_{j,-1}(h) Q_{j-1}+c_{j,0}(h) Q_j+c_{j,1}(h) Q_{j+1}= R_j(h),\qquad j=1,\ldots,N,
\ee
which can be put together into the following matrix form
\be \label{error:est}
A(h) \vec{Q} = \vec{R}(h)
\quad \mbox{with}\quad \vec{Q}:=[Q_1,\ldots,Q_N]^\tp, \vec{R}(h):=[R_1,\ldots,R_N]^\tp,
\ee
where $A(h)$ is an $N\times N$ tridiagonal matrix defined by
\begin{equation} \label{Ah}
A(h)=\text{tridiag}(\{c_{j,-1}(h)\}_{j=2}^{N},\{c_{j,0}(h)\}_{j=1}^{N},\{c_{j,1}(h)\}_{j=1}^{N-1}).
\end{equation}
Setting $h=0$, we have
a related $N\times N$ constant tridiagonal matrix $A(0)$ given by
\begin{equation} \label{A0}
A(0)=\text{tridiag}(\{c_{j,-1}(0)\}_{j=2}^{N},\{c_{j,0}(0)\}_{j=1}^{N},\{c_{j,1}(0)\}_{j=1}^{N-1}),
\end{equation}
where $\text{tridiag}(\cdot,\cdot,\cdot)$ is defined in (S2) of \cref{alg:dat} and
the entries of $A(0)$ are given as follows:
\[
\begin{cases}
	c_{j,-1}(0) = c_{j,1}(0) = -1, \quad c_{j,0}(0)=2 & \text{if $x_{j}$ is an interior point},\\
c_{j,0}(0)=2,\quad	 c_{j,-1}(0)=-\frac{2a(b_{j}-)}{a(b_{j}+) + a(b_{j}-)}, \quad c_{j,1}(0)=-2-c_{j,-1}(0),
& \text{if $x_{j}$ is a branch point},\\
	c_{0,0}(0) = 2, \quad c_{0,1}(0) = -1 & \text{if $\lambda_{1}^{L}=0$ and $\lambda_{0}^{L}=1$ in \eqref{heqnbdry}},\\
	c_{N,-1}(0) = -1, \quad c_{N,0}(0) =1, & \text{if $\lambda_{1}^{R} \neq 0$ in \eqref{heqnbdry}}.
\end{cases}
\]
If we impose a Dirichlet boundary condition at $1$, then we replace $N$ with $N-1$ in \eqref{fdm:stencil:ue}, \eqref{stencil:not}, \eqref{error:est}, \eqref{Ah}, and \eqref{A0}. Furthermore, $c_{N-1,-1}(0) = -1$ and $c_{N-1,0}(0) = 2$ in \eqref{A0}.

Next, we highlight some key issues as to why the theoretical convergence of compact FDMs in \cref{sec:fdm} for 1D heterogeneous Helmholtz equations with various boundary conditions requires a separate comprehensive treatment and will be addressed elsewhere. First, the solution stability of such Helmholtz equations is far from trivial and warrants further investigation. There are some cases in which the stability constant may exponentially rise; i.e., the solution is close to being `unstable' in some sense. In fact, the solution may become highly unstable under perturbation or even with fairly accurate approximation of boundary and source data. See \cite[Section 5.2]{GS20}. In these situations, the convergence of FDM (and any other discretization methods) is severely affected.
For illustration purposes, we mention two such cases by considering the simplest Helmholtz equation: 
\be \label{heqnconst}
u''+\kappa^2 u=f \quad \mbox{on}\quad [0,1] \quad \mbox{with}\quad u(0)=u(1)=0, \quad \mbox{a constant wave number}\; \kappa>0.
\ee
First, it is well known that solving the simplest Helmholtz equation in \eqref{heqnconst} with large wave numbers $\kappa$ is challenging, because the huge stability constant grows quickly with $\kappa^2$ and causes the pollution effect.
This requires the mesh size $h$ to be extremely small for any numerical schemes to start effectively approximating the true solution and exhibiting convergence behavior.
Second, if $\kappa=n\pi$ with $n\in \N$, then the solution $u$ to \eqref{heqnconst} is obviously not unique since $u(x)+\alpha \sin(n\pi x)$ are also solutions to \eqref{heqnconst} for all $\alpha\in \C$. Now consider \eqref{heqnconst} with $\kappa=n\pi\pm \epsilon$ with $n\in \N$ and a very small $\epsilon>0$. Though the solution to \eqref{heqnconst} is now unique and the wave number $\kappa$ is quite small, as we shall explain later, its true solution is highly unstable in some sense. One has to use a small mesh size $h$ in proportion to $\epsilon$ (which may be smaller than machine precision) for any numerical scheme to start effectively approximating the true solution and exhibiting convergence behavior. These phenomena and difficulties call for further investigation of the stability of Helmholtz equations and its relations to convergence properties of FDMs.
Because DAT can break any large problem into very small ones, the above discussion in fact shows the advantages and contributions of DAT for numerical solutions of Helmholtz equations.

Recall that for an $m\times n$ matrix $A$, the $\infty$-norm of $A$ is $\|A\|_\infty:=\sup_{1\le j\le m} \sum_{k=1}^{n} |A_{j,k}|$, which is the operator norm mapping $\ell_\infty^n$ to $\ell_\infty^m$.
In the convergence analysis of FDM, one can deduce from the identity \eqref{error:est} that
\begin{equation}\label{Ahinv}
\|\vec{Q}\|_\infty:=\sup_{1\le j\le N} |Q_j|
\le \| A(h)^{-1} \vec{R}(h)\|_\infty
\le \|A(h)^{-1}\|_\infty \|\vec{R}(h)\|_\infty.
\end{equation}
Hence, how $\|A(h)^{-1}\|_{\infty}$ behaves for small $h$ is a key issue. Even though $A(h)$ in \eqref{error:est} converges entrywise to $A(0)$ in \eqref{A0}, the invertibility of $A(h)$ and the norm estimates of $\|A(h)^{-1}\|_{\infty}$ are not immediately guaranteed by the properties of $A(0)$ in \eqref{A0}, since the size $N$ of $A(h)$ goes to $\infty$ as $h \rightarrow 0$. Furthermore, the structure of $A(h)^{-1}$ may be unknown. In stark contrast to elliptic equations, $A(h)$ may be singular or highly ill-conditioned not only for large wave numbers but also for small wave numbers.
Let us consider the simplest Helmholtz equation in \eqref{heqnconst} again and
use the standard second order FDM. Then at mesh size $h=N^{-1}$, our coefficient matrix is $A(h)=\text{tridiag}(\{-1\}_{j=2}^{N-1},\{2-\kappa^{2}h^{2}\}_{j=1}^{N-1},\{-1\}_{j=1}^{N-2})$, whose eigenvalues are known to take the following form
\[
\sigma_n:=(2-\kappa^{2}h^{2})-2\cos(nh\pi), \quad \forall n=1,\dots,N-1.
\]
Note that the $n$th eigenvalue, $\sigma_n$, vanishes and hence $\det(A(h))=0$
if
\be \label{special:kappa}
\kappa=\kappa_*(h,n), \quad \mbox{where}\quad \kappa_*(h,n):= h^{-1}\sqrt{2(1-\cos(nh \pi))}.
\ee
This situation is not encountered in the elliptic case, since all its eigenvalues $(2+\kappa^{2}h^{2})-2\cos(nh\pi)>0$ for all $n\in \N$.
Consider $\kappa=\kappa_*(2^{-7},3) \approx 9.4226$, that is, $\kappa=3\pi-\epsilon$ for some $0<\epsilon<0.0022$. Then the standard second order FDM fails to produce any solution at $h=2^{-7}$ because $\det(A(h))=0$.
For $\kappa>0$,  we define the distance $\rho_{\kappa}:= \min_{n \in \mathbb{N}}| \kappa -n\pi|$, which can be arbitrarily small for any mesh size $h$. For example,
for the mesh size $h=2^{-19}$, we see that $\rho_{\kappa}\approx 10^{-10}$ with $\kappa:=\kappa_*(2^{-19},3)\approx 9.4248$ but $\det(A(h))=0$ at $h=2^{-19}$.
Note that  the commonly used criterion $\kappa^{3/2} h=\bo(1)$ is satisfied because
$\kappa^{3/2} h \approx 5 \times 10^{-5}$ with $\kappa=\kappa_*(2^{-19},3)$ and $h=2^{-19}$. However, we need to employ an impractically small grid size for the FDM before any convergence is perceived.
The situation is exacerbated if $\kappa$ is very large and $\rho_{\kappa}$ is very small.
The foregoing point first demonstrates how we need to carefully quantify and elaborate on what `sufficiently small $h$' means for some form of convergence in the pre-asymptotic (computationally feasible) range to take place, which theoretically may be challenging (much harder than elliptic equations); and second, it refers back to an earlier key issue regarding the significance of understanding the solution's stability. For the example presented above, one can check by a direct calculation that the energy norm of the true solution is large. The theoretical convergence for 1D heterogeneous Helmholtz equations with piecewise smooth coefficients demands more sophisticated analysis due to its underlying intricacies.

Before we turn to the convergence of DAT, we discuss how to estimate $u'(\xb)$ for $\xb=x_j$ for some $0\le j\le N$ from $u(x_k):=u_N(x_k), k=0,\ldots,N$, since $u_N'$ is used in the linking problems of DAT and in the error  $\|u_N'-\eu'\|_\infty$ for measuring performance.
Assume that the numerical $u$ (i.e., $u_N$) is computed with accuracy order $M$, that is, $|u(x_k)-\eu(x_k)|\le C h^{M}$ for all $k=0,\ldots,N$ for some $C>0$ independent of $N$ and $h$.
We can estimate one-sided derivatives $u'(\xb+)$ and $u'(\xb-)$ with the same accuracy order as well. Basically, let
$\mathcal{L}^{\mathcal{B}^+}_h u(\xb)$ and $\mathcal{L}^{\mathcal{B}^-}_h u(\xb)$ be the stencils with accuracy order $M$ for boundary conditions in \eqref{bdry:breaking} through \cref{thm:bdry}. Then
\be \label{u:deriv}
\eu'(\xb+)=\mathcal{L}^{\mathcal{B}^+}_h u(\xb)+\bo(h^M),
\qquad \eu'(\xb-)=\mathcal{L}^{\mathcal{B}^-}_h u(\xb)+\bo(h^M),\qquad h\to 0,
\ee
which can be also derived from \eqref{u:deriv:0} easily.
Higher order one-sided derivatives at $\xb$ can also be estimated with accuracy order $M$ thanks to \cref{prop:span}.
Moreover, since we can obtain the one-sided derivatives of $\eu$ at all knot points with accuracy order $M$, using interpolation we can obtain a function $u(x), 0\le x\le 1$ from the computed data $\{u(x_j)\}_{j=0}^{N}$ such that $u$ accurately approximates the exact solution $\eu$ in the function setting.
The identities in \eqref{u:deriv} play a critical role in DAT to accurately estimate artificial Dirac distributions in \eqref{heqnL:3} and
\eqref{heqnL:4} for DAT.

Now we are ready to discuss the convergence of DAT. Recall that the average mesh size $h:=N^{-1}$.
Assume that the Helmholtz equation in \eqref{heqnL}--\eqref{heqnbdry} has a unique solution.
Let $N_0$ be a given integer independent of $N$ and $h$.  Now we claim that
\be \label{dat:converg}
\begin{split}
&\mbox{\emph{If all local problems in DAT are at most $N_0 \times N_0$ in size, then all the condition }}\\
&\mbox{\emph{numbers of all local problems in DAT must be uniformly bounded and DAT using the}}\\
&\mbox{\emph{$M$th order compact FDM
exhibits $\bo(h^M)$ convergence for sufficiently small $h$.}}
\end{split}
\ee
The argument is as follows. According to the theory of DAT in \cref{sec:dat}, the accuracy of DAT only depends on the accuracy of the local problem solver and the error accumulated from the tree depth and the linking problems. So, let us look at one typical local problem with grid points $\alpha=x_{L}<x_{L+1}<\cdots<x_{H-1}<x_{H}=\beta$ on $(\alpha,\beta)$.  For small $h$, as explained in \cref{sec:dat} on DAT, the boundary conditions for this typical local problem are either Dirichlet boundary conditions at both $\alpha$ and $\beta$ with at most one branch point inside $(\alpha,\beta)$, or Dirichlet boundary condition at 0 and the prescribed boundary condition as in \eqref{heqnbdry} at 1. Let $m$ be the size of this local problem. Then the relation in \eqref{error:est} still holds with $N=m$, $\vec{U}:=[U_{L+1},\ldots, U_{L+m}]^\tp$ and $\vec{R}(h)=[R_{L},\ldots, R_{L+m}]^\tp$.
Because the size $m\le N_0$,
we have $\lim_{h\to 0} \|A(h)-A(0)\|_\infty=0$, where the $m\times m$ matrix $A(0)$ is given in \eqref{A0}.
If the local mesh $\{x_L,\ldots, x_R\}$ does not contain any branch point, then $A(0)$ must be the standard $m\times m$ tridiagonal matrix generated by $[-1,2,-1]$, probably with $[A(0)]_{m,m}=1$ instead of $2$ depending on the boundary condition at $\beta$. The later matrix $A(0)$ is known to be invertible with $\det(A(0))=m+1$, or $1$ if $[A(0)]_{m,m}=1$.
Suppose now that the local mesh contains a branch point $b_j$ and the $k$th row of $A(0)$ corresponds to this branch point $b_j$.
Then the $k$th row of the standard tridiagonal matrix $A(0)$ with $[-1,2,-1]$ is replaced by
$[\frac{-2a(b_j-)}{a(b_j+)+a(b_j-)}, 2,
\frac{-2a(b_j+)}{a(b_j+)+a(b_j-)}]$.
Then $\text{det}(A(0))=2((1+m-k)a(b_j-) + k a(b_j+))(a(b_j-) + a(b_j+))^{-1}$ if a Dirichlet boundary condition is imposed at $\beta$, or $\text{det}(A(0))=2a(b_j-)(a(b_j-) + a(b_j+))^{-1}$ if a Neumann/Robin boundary condition is imposed at $\beta$. In all cases, the determinant of $A(0)$ is nonzero; thus, $A(0)$ must be an invertible matrix. Because $A(h)$ is at most $N_0\times N_0$, we conclude that
$A(h)$ is invertible for all sufficiently small $h$, $\lim_{h\to 0} \|A(h)^{-1}-A(0)^{-1}\|_\infty=0$,
 and there exists $C_1>0$ independent of $h$ such that $\|A(h)^{-1}\|_\infty \le C_1$ for all small $h>0$. Hence, the condition number of $A(h)$ is uniformly bounded for all local problems and we deduce from \eqref{error:est} and \eqref{Ahinv}  that $\|\vec{Q}\|_\infty \le C_1 \|\vec{R}(h)\|_\infty$. If we use the $M$th order compact FDM, then \eqref{Rj} must hold and hence $\|\vec{R}(h)\|_\infty \le C h^{M+1}$ for all sufficiently small $h$. Putting everything together, we proved that $\|\vec{Q}\|_\infty \le C_1 C h^{M+1}$ for convergence of all local problems in (S2) of \cref{alg:dat}.

For the linking problems, we have to estimate one-sided derivatives $u'$ for approximated solutions $u$ of all local problems. As we discussed before, this can be done by using \eqref{u:deriv} with $M$ being replaced by $M+1$, because the local problems are solved with accuracy order $M+1$ as we discussed a moment ago.
However, we cannot expect from \eqref{u:deriv} to achieve $\|\eu'-u'\|_\infty\le C_2 h^{M+1}$ with a positive constant $C_2$ independent of $h$, where $\eu$ and $u$ stand for the exact solution and approximated solution of a local problem. Note that the constant $C_2$ only depends on $a, \kappa$ and the partitioned source term $f_j=f\varphi_j$, where $\varphi_j$ is the hat function supported on $[\alpha,\beta]$ with $\varphi_j(\gamma)=1$ for some $\gamma\in [\alpha,\beta]$. However, $\beta-\alpha=\bo(h)$ due to $m\le N_0$ and hence, $\|\varphi_j'\|_\infty =\bo(h^{-1})$. Consequently, one can observe that $\|f_j^{(n)}\|_\infty \le C_3 h^{-1}$ for all $n=0,\ldots, M$, where the positive constant $C_3$ only depends on $f$ and is independent of $h$. That is, we can only expect $C_2\le C_3 h^{-1}$ and consequently,
$\|\eu'-u'\|_\infty\le C_2 h^{M+1}\le C_3 h^{M}$.
It is hard to exactly quantify how the error propagates from the deepest tree level to the surface tree level through the linking problems. Our numerical experiments seem to indicate that the linking problems do not further reduce accuracy.
Because the one-sided derivatives $u'$ can be estimated with accuracy $\bo(h^M)$,
the solution $\vec{Q}$ is expected to behave like $\|\vec{Q}\|_\infty\le C C_1 C_2 h^{M+1}\le C C_1C_3 h^M$ for sufficiently small $h$. This leads to the claim in \eqref{dat:converg}.

\section{Numerical Experiments}
\label{sec:examples}

In this section, we present several numerical experiments to illustrate the performance of DAT in \cref{sec:dat} and the developed compact FDMs in \cref{sec:fdm}. Let $\eu$ and $\{u_N(x_j)\}_{j=0}^N$ be the exact (if its analytic expression is known) and approximated solutions on knot points $\{x_j\}_{j=0}^N$ with $0=x_0<x_1<\ldots<x_{N-1}<x_N=1$, respectively.  Because $2$-norm is controlled by $\infty$-norm in \eqref{def:2:0},
we shall measure the performance in $\infty$-norm using relative errors $\frac{\|u_N-\eu\|_\infty}{\|\eu\|_\infty}$ and
$\frac{\|u_N'-\eu'\|_\infty}{\|\eu'\|_\infty}$, where $\{u_N'(x_j)\}_{j=0}^N$ are estimated from $\{u_N(x_j)\}_{j=0}^N$ through \eqref{u:deriv}.
When the analytic expression of the exact solution $\eu$ is unknown, we
calculate the relative error between two consecutive levels. Due to the pollution effect, we know that our grid size has to be at least smaller than $\|\kappa\|_\infty^{-1}$. When we perform our experiments, we initially set our grid size to be approximately $\|\kappa\|_\infty^{-1}$, refine dyadically, and only record the numerical results where a convergent behaviour is present (either with respect to the exact solution or the solution at the subsequent grid refinement). All condition numbers are approximated by using \texttt{condest} in MATLAB, after renormalizing all the diagonal entries to be one in the coefficient matrices.
The columns ``Local CN" and ``Link CN" in all tables in this section list the maximum condition numbers associated with local and link problems in DAT.
The tree level and split parameter used are denoted by $\ell$ and $s$. The default choice is $s=1$. Also, $\ell=0$ means we use FDM without DAT. All linear systems are solved by using MATLAB's backslash command.
For all examples below, we use the $M$th order compact FDMs in \Cref{subsec:fdm:6} with $M=6$ or $M=8$. To visualize the numerical performance, the vertical axis
in each convergence plot uses a base-10 log scale and the horizontal axis uses a base-2 log scale.

\subsection{A comparison with PUFEM}

\begin{example}\label{ex:constant}
	\normalfont
	Consider the model problem \eqref{heqnL}-\eqref{heqnbdry} given by
	$[a(x)u'(x)]'+\kappa^2(x) u(x)=f(x), x\in (0,1)$ with the coefficients $a=1$, $\kappa=10^{6}$, $f=\kappa^{2}\cosh(x)$, and the boundary conditions $u(0)=0$ and $u'(1)-i\kappa u(1)=0$.
	%
	%
	The exact solution has the following analytic expression
	\[
	\eu = \frac{-\kappa \sin(\kappa x)}{\kappa^{2}+1} (\sinh(1) - i \cosh(1) \kappa) e^{i\kappa} + \frac{\kappa^{2}(\cosh(x)-e^{i\kappa x})}{\kappa^{2}+1}.
	\]
	See
	\cref{table:vsPUFEM} for the numerical performance measured by $\frac{\|u_N-\eu\|_\infty}{\|\eu\|_\infty}$ and
	 $\frac{\|u_N'-\eu'\|_\infty}{\|\eu'\|_\infty}$. The errors for PUFEM are evaluated at nodal points. Because the wave number $\kappa=10^6$
	is large, to fairly compare DAT with PUFEM, all inner products in PUFEM are calculated exactly via symbolic computation to minimize possible errors due to numerical quadrature. ``Local CN'' for PUFEM lists the condition number of its coefficient matrix. All local and linking problems in DAT in \cref{table:vsPUFEM} solve at most $4\times 4$ linear systems with uniformly bounded small condition numbers.
	\cref{table:vsPUFEM} demonstrates that DAT can handle very small mesh size and the maximum condition numbers of coefficient matrices coming from all local and linking problems are much smaller than those in FDM and PUFEM by several orders of magnitude.
\end{example}

{\tiny
	\begin{center}
		\begin{tabular}{c c | c c c c | c c c c }
			\hline
			\hline
			$N$ & $\ell$ & $\frac{\|u_N-\eu\|_{\infty}}{\|\eu\|_{\infty}}$ & $\frac{\|u_N'-\eu'\|_{\infty}}{\|\eu'\|_{\infty}}$ & Local CN & Link CN & $\frac{\|u_N-\eu\|_{\infty}}{\|\eu\|_{\infty}}$ & $\frac{\|u_N'-\eu'\|_{\infty}}{\|\eu'\|_{\infty}}$ & Local CN & Link CN \\
			\hline
			\multicolumn{2}{c}{} \vline & \multicolumn{4}{c}{DAT using the compact FDM with order $M=6$} \vline & \multicolumn{4}{c}{DAT using the compact FDM with order $M=8$}\\
			\hline
			$2^{21}$ & $0$ & $1.7276 \times 10^{-1}$ & $4.0087 \times 10^{-1}$ & $1.61 \times 10^{7}$ & $-$ & $4.3849 \times 10^{-4}$ & $1.0173 \times 10^{-3}$ & $2.07 \times 10^{7}$ & $-$\\
			$2^{21}$ & $19$ & $1.7276 \times 10^{-1}$ & $4.0087 \times 10^{-1}$ & $3.23 \times 10^{1}$ & $4.18 \times 10^{2}$ & $4.3849 \times 10^{-4}$ & $1.0173 \times 10^{-3}$ & $3.23 \times 10^{1}$ & $6.27 \times 10^{1}$\\
			\hline
			$2^{22}$ & $0$ & $2.6379 \times 10^{-3}$ & $6.1212 \times 10^{-3}$ & $7.15 \times 10^{7}$ & $-$ & $1.6674 \times 10^{-6}$ & $3.8683 \times 10^{-6}$ & $8.65 \times 10^{7}$ & $-$\\
			$2^{22}$ & $20$ & $2.6379 \times 10^{-3}$ & $6.1212 \times 10^{-3}$ & $4.15 \times 10^{1}$ & $6.37 \times 10^{1}$ & $1.6671 \times 10^{-6}$ & $3.8677 \times 10^{-6}$ & $4.15 \times 10^{1}$ & $6.28 \times 10^{1}$\\
			\hline
			$2^{23}$ & $0$ & $4.0945 \times 10^{-5}$ & $9.5012 \times 10^{-5}$ & $3.51 \times 10^{8}$ & $-$ & $8.1795 \times 10^{-9}$ & $1.8977 \times 10^{-8}$ & $3.51 \times 10^{8}$ & $-$\\
			$2^{23}$ & $21$ & $4.0946 \times 10^{-5}$ & $9.5014 \times 10^{-5}$ & $4.41 \times 10^{1}$ & $6.29 \times 10^{1}$  & $1.1594 \times 10^{-8}$ & $2.6901 \times 10^{-8}$ & $4.41 \times 10^{1}$ & $6.28 \times 10^{1}$\\
			\hline
			\multicolumn{2}{c}{} \vline & \multicolumn{4}{c}{PUFEM in \cite{BS00}} \vline & \multicolumn{4}{c}{}\\
			\hline
			$2^{21}$ & $-$ & $1.2806 \times 10^{-1}$ & $5.7930 \times 10^{-1}$ & $2.09 \times 10^{7}$ & $-$ & & & & \\
			$2^{22}$ & $-$ & $3.2473 \times 10^{-2}$ & $2.1123 \times 10^{-1}$ & $8.69 \times 10^{7}$ & $-$ & & & &\\
			$2^{23}$ & $-$ & $8.1473 \times 10^{-3}$ & $8.9740 \times 10^{-2}$ & $2.38 \times 10^{8}$ & $-$ & & & &\\
			\hline
		\end{tabular}
		\captionof{table}{
			Relative errors for \cref{ex:constant} using DAT with $N_0=4$ and $s=1$ in \cref{alg:dat}, and PUFEM. The grid increment used in $[0,1]$ is $N^{-1}$.
		}
		\label{table:vsPUFEM}
	\end{center}
}

\subsection{Numerical experiments on 1D heterogeneous Helmholtz equations}

\begin{example}\label{ex:piece8}
	\normalfont
	Consider the model problem \eqref{heqnL}-\eqref{heqnbdry} given by
$[a(x)u'(x)]'+\kappa^2(x) u(x)=f(x), x\in (0,1)$ with the following piecewise smooth jumping coefficients having large variation:
	\begin{align*}
	a & =\chi_{[0,\frac{1}{8})}+10^{-1}\chi_{[\frac{1}{8},\frac{2}{8})}+\chi_{[\frac{2}{8},\frac{3}{8})}+10^{-2}\chi_{[\frac{3}{8},\frac{4}{8})}+\chi_{[\frac{4}{8},\frac{5}{8})}+10^{-3}\chi_{[\frac{5}{8},\frac{6}{8})} +\chi_{[\frac{6}{8},\frac{7}{8})}+10^{-4}\chi_{[\frac{7}{8},1]},\\
	\kappa & =10^{4}(\chi_{[0,\frac{1}{8}) \cup [\frac{2}{8},\frac{3}{8}) \cup [\frac{4}{8},\frac{5}{8}) \cup [\frac{6}{8},\frac{7}{8})}) + 500(\chi_{[\frac{1}{8},\frac{2}{8}) \cup [\frac{3}{8},\frac{4}{8}) \cup [\frac{5}{8},\frac{6}{8}) \cup [\frac{7}{8},1]}),\\
	f & = 10^{7}e^{x}(\chi_{[0,\frac{1}{8}) \cup [\frac{2}{8},\frac{3}{8}) \cup [\frac{4}{8},\frac{5}{8}) \cup [\frac{6}{8},\frac{7}{8})})-e^{-2x}(\chi_{[\frac{1}{8},\frac{2}{8}) \cup [\frac{3}{8},\frac{4}{8}) \cup [\frac{5}{8},\frac{6}{8}) \cup [\frac{7}{8},1]}),
	\end{align*}
	and the boundary conditions $u(0)=0$ and $10^{-2}u'(1)-i 500 u(1)=0$.
The exact solution $\eu$ has an analytic expression which is given
on each interval $(2^{-3}(j-1),2^{-3}j)$ for $j=1,\ldots,8$ by
{\footnotesize{\begin{align*}
\eu(x) & = A_j \exp\left(i\frac{\kappa(x)}{\sqrt{a(x)}}x\right) + B_j \exp\left(-i\frac{\kappa(x)}{\sqrt{a(x)}}x\right) + \frac{\exp\left(i\frac{\kappa(x)}{\sqrt{a(x)}}x\right)}{2i\kappa(x) \sqrt{a(x)}} \int_{2^{-3}(j-1)}^x
f(t)  \exp\left(-i\frac{\kappa(t)}{\sqrt{a(t)}}t\right) dt \\
& \quad - \frac{\exp\left({-i\frac{\kappa(x)}{\sqrt{a(x)}}x}\right)}{2i\kappa(x) \sqrt{a(x)}} \int_{2^{-3}(j-1)}^x f(t)  \exp\left({i\frac{\kappa(t)}{\sqrt{a(t)}}t}\right) dt,\qquad x\in (2^{-3}(j-1), 2^{-3} j),
\end{align*}
}}
where all the coefficients $A_j$, $B_j$ for $j=1,\dots,8$ are uniquely determined by solving a system of linear equations that arises from imposing the boundary conditions and the following transmission conditions
\[
\eu(2^{-3}j-)=\eu(2^{-3}j+), \quad a(2^{-3}j-)\eu'(2^{-3}j-)=a(2^{-3}j+)\eu'(2^{-3}j+), \qquad j=1,\dots,8.
\]
See \cref{tab:piece8rel} for the numerical performance measured by $\frac{\|u_N-\eu\|_\infty}{\|\eu\|_\infty}$ and
$\frac{\|u_N'-\eu'\|_\infty}{\|\eu'\|_\infty}$, and \cref{fig:piece8rel} for the convergence plot and approximated solution $u_N$. As can be seen from \cref{tab:piece8rel}, the convergence rates agree with the theoretical discussion in \cref{sec:fdm,sec:convergence}.
\end{example}

	{\tiny
		\begin{center}
			\begin{tabular}{c c|c c c c|c c c c c}
				\hline
                \hline
\multicolumn{2}{c}{} \vline	
&\multicolumn{4}{c}{DAT using the compact FDM with order $M=6$} \vline
&\multicolumn{4}{c}{DAT using the compact FDM with order $M=8$}\\
				\hline
$N$ & $\ell$
& $\frac{\|u_N-\eu\|_{\infty}}{\|\eu\|_{\infty}}$
& $\frac{\|u_N'-\eu'\|_{\infty}}{\|\eu'\|_{\infty}}$
& Local CN & Link CN
& $\frac{\|u_N-\eu\|_{\infty}}{\|\eu\|_{\infty}}$
& $\frac{\|u_N'-\eu'\|_{\infty}}{\|\eu'\|_{\infty}}$
& Local CN & Link CN\\
				\hline
$2^{15}$ & $0$ & $2.0833 \times 10^{-1}$ & $1.2491$
& $1.62 \times 10^{6}$ & $-$
& $8.0406 \times 10^{-3}$ & $4.8823 \times 10^{-2}$
& $1.59 \times 10^{7}$ & $-$\\

				& $7$ & $2.0833 \times 10^{-1}$ & $1.2491$
& $2.27 \times 10^4$ & $2.32 \times 10^3$
& $8.0406 \times 10^{-3}$ & $4.8823 \times 10^{-2}$
& $2.22 \times 10^{4}$ & $2.32 \times 10^{3}$\\

				& $10$ & $2.0833 \times 10^{-1}$ & $1.2491$
& $3.98 \times 10^2$ & $2.32 \times 10^3$
& $8.0406 \times 10^{-3}$ & $4.8823 \times 10^{-2}$
& $3.87 \times 10^{2}$ & $2.32 \times 10^{3}$\\

				\hline
				$2^{16}$ & $0$ & $3.5328 \times 10^{-3}$ & $2.1422 \times 10^{-2}$
& $7.55 \times 10^{6}$ & $-$
& $2.3512 \times 10^{-5}$ & $1.4404 \times 10^{-4}$
& $7.56 \times 10^{6}$ & $-$\\

				& $8$ & $3.5328 \times 10^{-3}$ & $2.1422 \times 10^{-2}$
& $1.66 \times 10^{3}$ & $2.32 \times 10^3$
& $2.3512 \times 10^{-5}$ & $1.4404 \times 10^{-4}$
& $1.66 \times 10^{3}$ & $2.32 \times 10^{3}$\\

				& $11$ & $3.5328 \times 10^{-3}$ & $2.1422 \times 10^{-2}$ & $1.32 \times 10^2$ & $2.32 \times 10^3$
& $2.3512 \times 10^{-5}$ & $1.4404 \times 10^{-4}$  & $1.32 \times 10^{2}$ & $2.32 \times 10^{3}$\\
				\hline

				$2^{17}$ & $0$ & $5.1547 \times 10^{-5}$ & $3.1264 \times 10^{-4}$  & $2.87 \times 10^{7}$ & $-$
& $8.5834 \times 10^{-8}$ & $5.2706 \times 10^{-7}$ & $2.87 \times 10^{7}$ & $-$\\

				& $9$ & $5.1547 \times 10^{-5}$ & $3.1264 \times 10^{-4}$  & $1.18 \times 10^{4}$ & $2.32 \times 10^3$
& $8.5834 \times 10^{-8}$ & $5.2705 \times 10^{-7}$  & $1.18 \times 10^{4}$ & $2.32 \times 10^{3}$\\

				& $12$ & $5.1547 \times 10^{-5}$ & $3.1264 \times 10^{-4}$ & $3.65 \times 10^{1}$ & $2.32 \times 10^3$
& $8.5834 \times 10^{-8}$ & $5.2706 \times 10^{-7}$ & $3.65 \times 10^{1}$ & $2.36 \times 10^{3}$\\
				\hline
				$2^{18}$ & $0$ &$7.9194 \times 10^{-7}$ & $4.8033 \times 10^{-6}$ & $1.14 \times 10^{8}$ & $-$
& $3.2902 \times 10^{-10}$ & $2.0239 \times 10^{-9}$  & $1.14 \times 10^{8}$ & $-$\\
				& $10$ & $7.9194 \times 10^{-7}$ & $4.8033 \times 10^{-6}$ & $1.09 \times 10^{4}$ & $2.32 \times 10^3$
& $9.3775 \times 10^{-10}$ & $2.0278 \times 10^{-9}$ & $1.09 \times 10^{4}$ & $2.32 \times 10^{3}$\\

				& $13$ & $7.9194 \times 10^{-7}$ & $4.8033 \times 10^{-6}$  & $4.28 \times 10^{1}$ & $2.32 \times 10^3$
& $3.2825 \times 10^{-10}$ & $2.0194 \times 10^{-9}$  & $4.28 \times 10^{1}$ & $2.32 \times 10^{3}$\\
				\hline
			\end{tabular}
			 \captionof{table}{Relative errors for \cref{ex:piece8} using DAT with $N_0=32$ and $s=1$ in \cref{alg:dat}. The grid increment used in each sub-interval $[(k-1)2^{-3},k2^{-3}]$ with $1\le k\le 2^3$ is $N^{-1}$.
}
			 \label{tab:piece8rel}
		\end{center}
	}

	\begin{figure}[!htb]
	 \begin{subfigure}[b]{0.3\textwidth}
	\begin{tikzpicture}
		\begin{axis}
		[
		 xlabel={\footnotesize{$N\!\!=$}},
        xlabel style ={xshift=-2.3cm,yshift=0.48cm},
		 ylabel={\footnotesize{Relative errors}},
		height=0.9\textwidth,
		width=\textwidth,
		yminorticks=false,
		ymin=10e-12, ymax=2,
		 ytick={10e-13,10e-11,10e-9,10e-7,10e-5,10e-3,1},,
		xmode=log,
		log basis x={2},
		ymode=log,
		log basis y={10},
		legend style={nodes={scale=0.6, transform shape}},
		legend pos=outer north east,
		]
		 \addplot[thick,mark=*,blue] coordinates
		{(32768,2.0833e-1)
			(65536,3.5328e-3)
			(131072,5.1547e-5)
			 (262144,7.9194e-7)};
		 \addplot[thick,mark=*,red] coordinates
		{(32768,1.2491)
			(65536,2.1422e-2)
			(131072,3.1264e-4)
			 (262144,4.8033e-6)};
		\addplot[thick, dashed, mark=*,blue] coordinates
		{(32768,8.0406e-3)
			(65536,2.3512e-5)
			(131072,8.5834e-8)
			 (262144,3.2825e-10)};
		\addplot[thick, dashed, mark=*,red] coordinates
		{(32768,4.8823e-2)
			(65536,1.4404e-4)
			(131072,5.2706e-7)
			 (262144,2.0194e-9)};
		\end{axis}
		\node (B) at (1.1,0.75) {\tiny \scalebox{0.8}{$\mathcal{O}(N^{-8})$}};
		\node (A) at (1.9,1.6) {};
		\draw[thick,->] (B) edge (A);
		\node (D) at (3.05,3) {\tiny \scalebox{0.8}{$\mathcal{O}(N^{-6})$}};
		\node (C) at (2.35,2.2) {};
		\draw[thick,->] (D) edge (C);
		\end{tikzpicture}
		\end{subfigure}
		 \begin{subfigure}[b]{0.3\textwidth}
			 \raisebox{0.1cm}{\includegraphics[width=\textwidth]{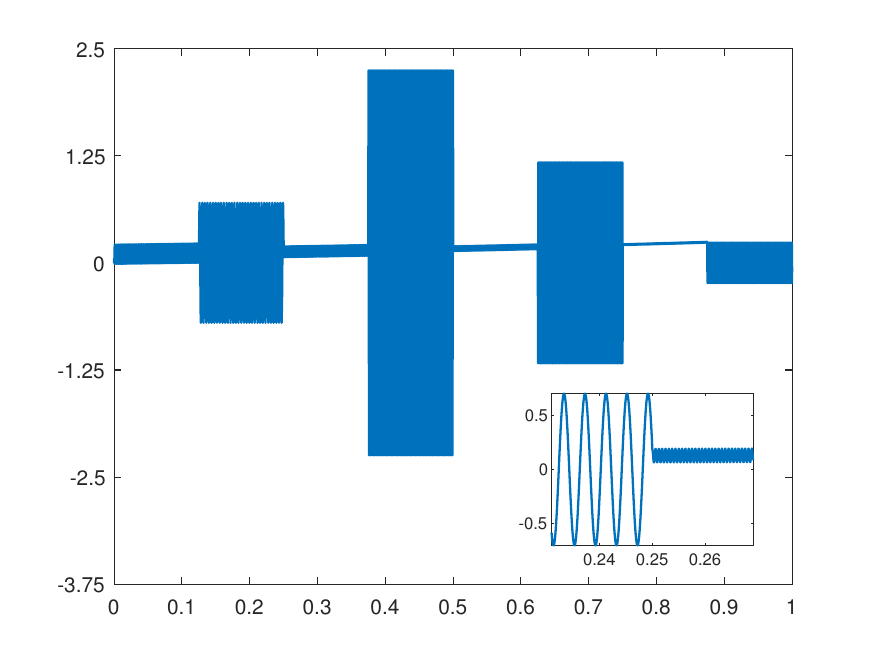}}
		\end{subfigure}
		 \begin{subfigure}[b]{0.3\textwidth}
			 \raisebox{0.1cm}{\includegraphics[width=\textwidth]{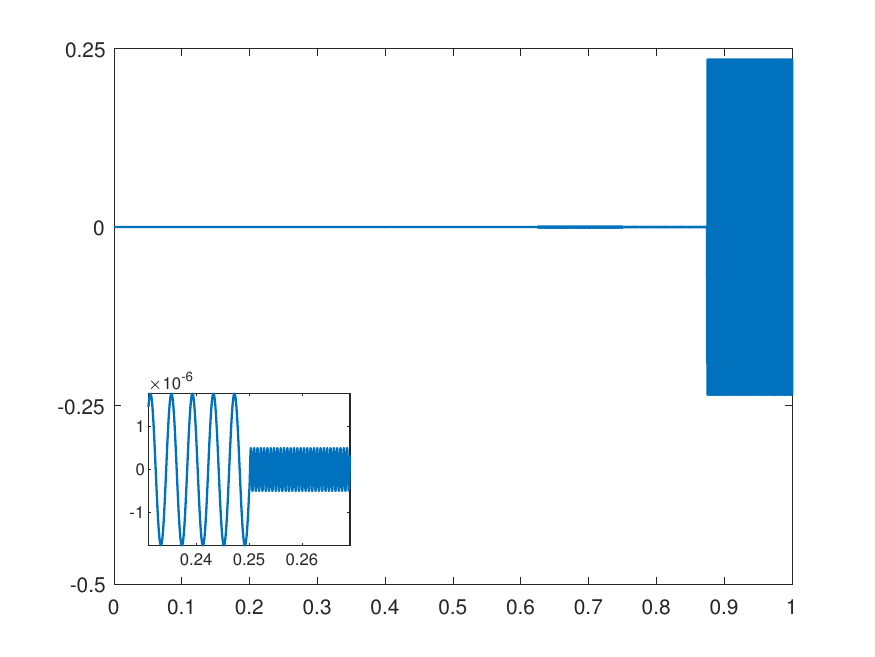}}
		\end{subfigure}
		 \caption{\cref{ex:piece8}: Convergence plot (left) of DAT using the compact FDM with order $M=6$ (solid) and $M=8$ (dashed) for errors $\frac{\|u_N-\eu\|_{\infty}}{\|\eu\|_{\infty}}$ (blue) and
		 $\frac{\|u_N'-\eu'\|_{\infty}}{\|\eu'\|_{\infty}}$ (red).
The displayed convergence rates are obtained by calculating $\log_2\left(\frac{\|u_N-\eu\|_{\infty}}{\|u_{2N}-\eu\|_{\infty}}\right)$ and $\log_2\left(\frac{\|u_N'-\eu'\|_{\infty}}
{\|u_{2N}'-\eu'\|_{\infty}}\right)$. The real (middle) and imaginary (right) parts of $u_N$ with $N=2^{18}$, $\ell=13$ and $M=8$.}
		\label{fig:piece8rel}
	\end{figure}

\begin{example} \label{ex:piece4het}
	\normalfont
	Consider
$[a(x)u'(x)]'+\kappa^2(x) u(x)=f(x), x\in (0,1)$
with the following coefficients
	\begin{align*}
	a & = e^{-x} \chi_{[0,\frac{31}{100}) \cup  [\frac{69}{100},\frac{81}{100})} + (e^{x}+1) \chi_{[\frac{31}{100},\frac{69}{100}) \cup  [\frac{81}{100},1]},\\
	\kappa & = 10^{4}e^{2x} \chi_{[0,\frac{31}{100})}
	+ 10^{5}x^4 \chi_{[\frac{31}{100},\frac{69}{100})}
	+ 10^{4}(1+x^4)\chi_{[\frac{69}{100},\frac{81}{100})}
	+ 10^{5}e^{-3x}\chi_{[\frac{81}{100},1]},\\
	f & = 10^{7} \left(\chi_{[0,\frac{31}{100})} + x^2 \chi_{[\frac{31}{100},\frac{69}{100})} + x^3 \chi_{[\frac{69}{100},\frac{81}{100})} + x^5\chi_{[\frac{81}{100},1]}\right),
	\end{align*}
and the boundary conditions $u(0)=1$ and $(e+1)^{1/2}u'(1)-i 10^5 e^{-3} u(1)=0$. The exact
solution's analytic expression is unknown.
See \cref{tab:piece4het} for the numerical performance measured by $\frac{\|u_{N}-u_{2N}\|_\infty}{\|u_{2N}\|_\infty}$ and
$\frac{\|u_N'-u_{2N}'\|_\infty}{\|u_{2N}'\|_\infty}$, and \cref{fig:piece4het} for the convergence plot and approximated solution $u_N$. As can be seen from \cref{tab:piece4het}, the convergence rates agree with the theoretical discussion in \cref{sec:fdm,sec:convergence}.
This example shows how DAT is stable with respect to splits and tree levels. For simplicity, we consider a tree level that is not high. Hence, it is to be expected that the maximum condition numbers of the local and linking problems are still relatively large, but are nonetheless smaller than the condition numbers of FDM. In fact, if we look at these condition numbers in granular detail, a large proportion of them are significantly smaller than those of FDM for any given $N$. We also note that the maximum condition numbers listed in the column ``Local CN" are the same for $(\ell,s)=(5,1)$ and $(\ell,s)=(3,2)$. The reason is because these two rows share the same local problems as defined in \eqref{localproblem:regular}. The only difference lies in the size of the linking problems: $3 \times 3$ for $(\ell,s)=(5,1)$ and $7 \times 7$ for $(\ell,s)=(3,2)$.
\end{example}

	{\tiny
		\begin{center}
			\begin{tabular}{c c | c c c c | c c c c}
				\hline
				\hline
				 \multicolumn{2}{c}{} \vline & \multicolumn{4}{c}{DAT using the compact FDM with order $M=6$} \vline & \multicolumn{4}{c}{DAT using the compact FD with order $M=8$}\\
                \hline
				$N$ & $(\ell,s)$ & $\frac{\|u_{N}-u_{2N}\|_{\infty}}{\|u_{2N}\|_{\infty}}$ & $\frac{\|u'_{N}-u'_{2N}\|_{\infty}}{\|u'_{2N}\|_{\infty}}$ & Local CN & Link CN & $\frac{\|u_{N}-u_{2N}\|_{\infty}}{\|u_{2N}\|_{\infty}}$ & $\frac{\|u'_{N}-u'_{2N}\|_{\infty}}{\|u'_{2N}\|_{\infty}}$ & Local CN & Link CN\\
				\hline
				$2^{14}$ & $(0,0)$ & $5.9033 \times 10^{-1}$ & $8.6408 \times 10^{-1}$ & $2.63 \times 10^{9}$ & $-$ &
				$9.4394 \times 10^{-2}$ & $1.2948 \times 10^{-1}$ & $4.59 \times 10^{9}$ & $-$\\
				
				& $(5,1)$ & $5.9033 \times 10^{-1}$ & $8.6408 \times 10^{-1}$ & $6.29 \times 10^{4}$ & $7.40 \times 10^{4}$ &
				$9.4394 \times 10^{-2}$ & $1.2948 \times 10^{-1}$ & $2.68 \times 10^{5}$ & $8.17 \times 10^{4} $\\
				
				& $(3,2)$ & $5.9033 \times 10^{-1}$ & $8.6408 \times 10^{-1}$ & $6.29 \times 10^{4}$ & $1.95 \times 10^{4}$ &
				$9.4394 \times 10^{-2}$ & $1.2948 \times 10^{-1}$ & $2.68 \times 10^{5}$ & $2.94 \times 10^{4}$\\
				\hline
				$2^{15}$ & $(0,0)$ & $4.7473 \times 10^{-2}$ & $6.7611 \times 10^{-2}$ & $5.19 \times 10^{6}$ & $-$ &
				$2.4465 \times 10^{-4}$ & $3.3087 \times 10^{-4}$ & $5.56 \times 10^{6}$ & $-$\\
				
				& $(5,1)$ & $4.7473\times 10^{-2}$ & $6.7611 \times 10^{-2}$ & $2.07 \times 10^{5}$ & $8.41 \times 10^{3}$ &
				$2.4465 \times 10^{-4}$ & $3.3087 \times 10^{-4}$ & $2.07 \times 10^{5}$ & $8.41 \times 10^{3}$\\
				
				& $(3,2)$ & $4.7473\times 10^{-2}$ & $6.7611 \times 10^{-2}$ & $2.07 \times 10^{5}$ & $2.90 \times 10^{4}$ &
				$2.4465 \times 10^{-4}$ & $3.3087 \times 10^{-4}$ & $2.07 \times 10^{5}$ & $2.92 \times 10^{4}$\\
				\hline
				$2^{16}$ & $(0,0)$ & $7.2618 \times 10^{-4}$ & $1.0353 \times 10^{-3}$ & $2.22 \times 10^{7}$ & $-$
				& $8.7748 \times 10^{-7}$ & $1.1834 \times 10^{-6}$ & $2.22 \times 10^{7}$ & $-$\\
				
				& $(5,1)$ & $7.2618 \times 10^{-4}$ & $1.0353 \times 10^{-3}$ & $8.29 \times 10^{5}$ & $8.41 \times 10^{3}$ &
				$8.8029 \times 10^{-7}$ & $1.1867 \times 10^{-6}$ & $8.29 \times 10^{5}$ & $8.41 \times 10^{3}$\\
				
				& $(3,2)$ & $7.2618 \times 10^{-4}$ & $1.0353 \times 10^{-3}$ & $8.29 \times 10^{5}$ & $2.92 \times 10^{4}$ &
				$8.7595 \times 10^{-7}$ & $1.1813 \times 10^{-6}$ & $8.29 \times 10^{5}$ & $2.92 \times 10^{4}$\\
				\hline
				$2^{17}$ & $(0,0)$ & $1.1182\times 10^{-5}$ & $1.5959 \times 10^{-5}$ & $8.89 \times 10^{7}$ & $-$ &
				$3.7755 \times 10^{-9}$ & $5.0250 \times 10^{-9}$ & $8.89 \times 10^{7}$ & $-$\\
				
				& $(5,1)$ & $1.1188\times 10^{-5}$ & $1.5967 \times 10^{-5}$ & $3.32 \times 10^{6}$ & $8.41 \times 10^{3}$ &
				$3.2508 \times 10^{-9}$ & $6.3200 \times 10^{-9}$ & $3.32 \times 10^{6}$ & $8.41 \times 10^{3}$ \\
				
				& $(3,2)$ & $1.1179\times 10^{-5}$ & $1.5958 \times 10^{-5}$ & $3.32 \times 10^{6}$ & $2.92 \times 10^{4}$ &
				$6.9405 \times 10^{-9}$ & $9.3269 \times 10^{-9}$ & $3.32 \times 10^{6}$ & $2.92 \times 10^{4}$\\
				\hline
			\end{tabular}	
			 \captionof{table}{Relative errors for \cref{ex:piece4het} using DAT with $N_0=16$ and $s = 1,2$ in \cref{alg:dat}. The grid increments used in $[0,\frac{31}{100}]$, $[\frac{31}{100},\frac{69}{100}]$, $[\frac{69}{100},\frac{81}{100}]$, and $[\frac{81}{100},1]$ are respectively $\frac{31}{25 N}$, $\frac{38}{25 N}$, $\frac{12}{25 N}$, and $\frac{19}{25 N}$.
			}
			 \label{tab:piece4het}
		\end{center}
	}

	\begin{figure}[hbtp]
		 \begin{subfigure}[b]{0.3\textwidth}
		\begin{tikzpicture}
		\begin{axis}
		[
		 xlabel={\footnotesize{$N\!\!=$}},
        xlabel style ={xshift=-2.3cm,yshift=0.48cm},
		 ylabel={\footnotesize{Relative errors}},
		height=0.9\textwidth,
		width=\textwidth,
		yminorticks=false,
		ymin=10e-11, ymax=10e0,
		xmode=log,
		log basis x={2},
		ymode=log,
		log basis y={10},
		 ytick={10e-10,10e-8,10e-6,10e-4,10e-2,10e-0,10e2},
		legend style={nodes={scale=0.6, transform shape}},
		legend pos=outer north east,
		]
		 \addplot[thick,mark=*,blue] coordinates
		{(16384,5.9033e-1)
			(32768,4.7473e-2)
			(65536,7.2618e-4)
			 (131072,1.1179e-5)};
		 \addplot[thick,mark=*,red] coordinates
		{(16384,8.6408e-1)
			(32768,6.7611e-2)
			(65536,1.0353e-3)
			 (131072,1.5958e-5)};
		\addplot[thick, dashed, mark=*,blue] coordinates
		{(16384,9.4394e-2)
			(32768,2.4465e-4)
			(65536,8.7595e-7)
			 (131072,6.9405e-9)};
		\addplot[thick, dashed, mark=*,red] coordinates
		{(16384,1.2948e-1)
			(32768,3.3087e-4)
			(65536,1.1813e-6)
			 (131072,9.3269e-9)};
		\end{axis}
		\node (B) at (1.1,0.75) {\tiny \scalebox{0.8}{$\mathcal{O}(N^{-8})$}};
		\node (A) at (1.9,1.6) {};
		\draw[thick,->] (B) edge (A);
		\node (D) at (3.05,3) {\tiny \scalebox{0.8}{$\mathcal{O}(N^{-6})$}};
		\node (C) at (2.35,2.2) {};
		\draw[thick,->] (D) edge (C);
		\end{tikzpicture}
		\end{subfigure}
		 \begin{subfigure}[b]{0.3\textwidth}
			 \raisebox{0.1cm}{\includegraphics[width=\textwidth]{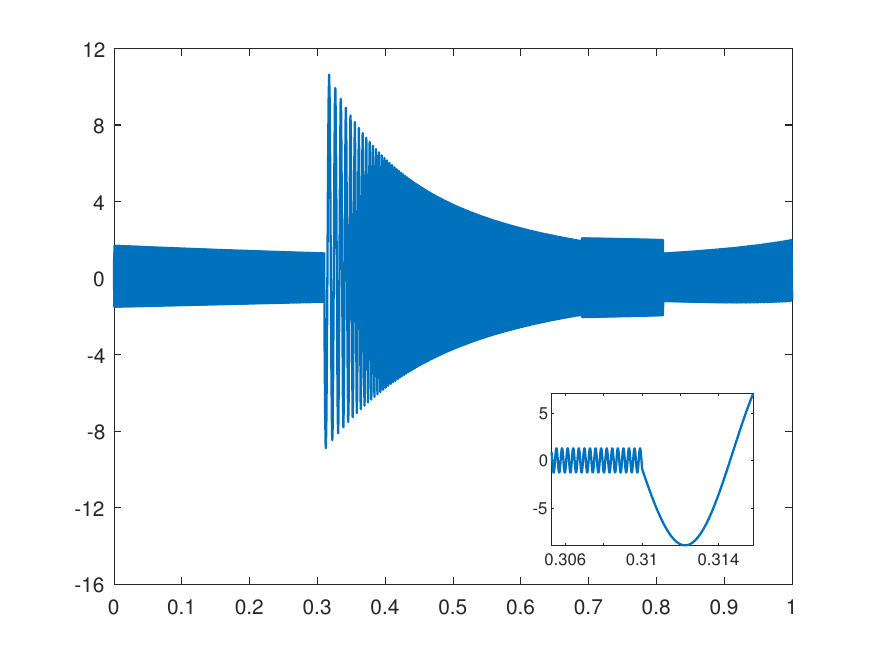}}
		\end{subfigure}
		 \begin{subfigure}[b]{0.3\textwidth}
			 \raisebox{0.1cm}{\includegraphics[width=\textwidth]{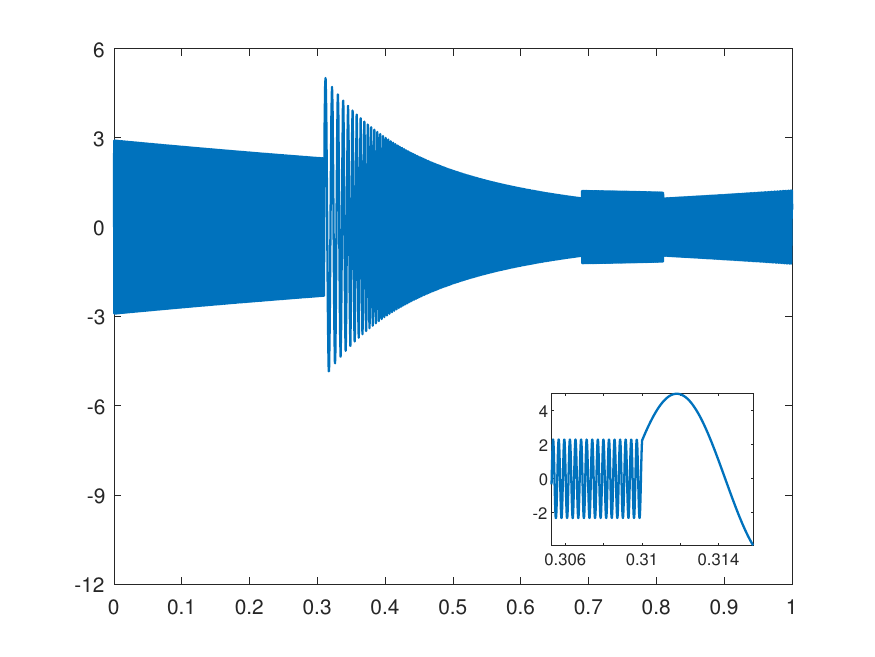}}
		\end{subfigure}
		\caption
		{\cref{ex:piece4het}:
	Convergence plot (left) of DAT using the compact FDM  with order $M=6$ (solid) and $M=8$ (dashed) for errors $\frac{\|u_N-u_{2N}\|_{\infty}}{\|u_{2N}\|_{\infty}}$ (blue) and
	 $\frac{\|u_N'-u_{2N}'\|_{\infty}}{\|u_{2N}'\|_{\infty}}$ (red). The displayed convergence rates are obtained by calculating $\log_2\left(\frac{\|u_N-u_{2N}\|_{\infty}/\|u_{2N}\|_{\infty}}
{\|u_{2N}-u_{4N}\|_{\infty}/\|u_{4N}\|_{\infty}}\right)$ and $\log_2\left(\frac{\|u_N'-u_{2N}'\|_{\infty}/\|u_{2N}'\|_{\infty} }{ \|u_{2N}'-u_{4N}'\|_{\infty}/\|u_{4N}'\|_{\infty}}\right)$. The real (middle) and imaginary (right) parts of $u_N$ with $N=2^{17}$, $(\ell,s)=(5,1)$ and $M=8$.}
		\label{fig:piece4het}
	\end{figure}

\begin{example}\label{ex:akvary}
	\normalfont
	Consider
$[a(x)u'(x)]'+\kappa^2(x) u(x)=f(x), x\in (0,1)$ with the following coefficients
	\[
	a=1.1+\sin(40\pi x), \quad \kappa=10^{5}\left(1-\left(x-0.5\right)^2\right), \quad f=10^{9}(x^{7}+1),
	\]
and the boundary conditions $\sqrt{1.1}u'(0)+i75000u(0)=-1$ and
$\sqrt{1.1}u'(1)-i 75000 u(1)=0$.
The exact
solution's analytic expression is unknown.
See \cref{tab:akvary} for the numerical performance measured by $\frac{\|u_N-u_{2N-1}\|_\infty}{\|u_{2N-1}\|_\infty}$ and
$\frac{\|u_N'-u_{2N-1}'\|_\infty}{\|u_{2N-1}'\|_\infty}$, and \cref{fig:akvary} for the convergence plot and approximated solution $u_N$. As can be seen from \cref{tab:akvary}, the convergence rates agree with the theoretical discussion in \cref{sec:fdm,sec:convergence}. As studied in \cite{GS20}, having $a$ and $\kappa$ that are oscillating and/or possess a large variation leads to an ill-conditioned coefficient matrix. This example explores DAT's potential in handling the Helmholtz problem with an oscillatory coefficient $a$ and a large wave number $\kappa$.
\end{example}

	{\tiny
		\begin{center}
			\begin{tabular}{c c | c c c c | c c c c }
				\hline
				\hline
				 \multicolumn{2}{c}{} \vline & \multicolumn{4}{c}{DAT using the compact FDM with order $M=6$} \vline & \multicolumn{4}{c}{DAT using the compact FDM with order $M=8$}\\
				\hline
				$N$ & $\ell$ & $\frac{\|u_{N}-u_{2N-1}\|_{\infty}}{\|u_{2N-1}\|_{\infty}}$ & $\frac{\|u'_{N}-u'_{2N-1}\|_{\infty}}{\|u'_{2N-1}\|_{\infty}}$ & Local CN & Link CN  & $\frac{\|u_{N}-u_{2N-1}\|_{\infty}}{\|u_{2N-1}\|_{\infty}}$ & $\frac{\|u'_{N}-u'_{2N-1}\|_{\infty}}{\|u'_{2N-1}\|_{\infty}}$ & Local CN & Link CN\\
				\hline
				$2^{18}+1$ & $0$ & $4.8707 \times 10^{-1}$ & $7.6812 \times 10^{-1}$ & $5.53 \times 10^{6}$ & $-$ &
				$6.7606 \times 10^{-3}$ & $1.0602 \times 10^{-2}$ & $5.66 \times 10^{6}$ & $-$\\
				& $16$ & $4.8707 \times 10^{-1}$ & $7.6812 \times 10^{-1}$ & $2.74 \times 10^{5}$ & $7.71 \times 10^{6}$ & $6.7606 \times 10^{-3}$ & $1.0602 \times 10^{-2}$ & $5.25 \times 10^{5}$ & $1.20 \times 10^{7}$\\ 
				
				\hline
				
				$2^{19}+1$ & $0$ & $7.1208 \times 10^{-3}$ & $1.1068 \times 10^{-2}$ & $1.47 \times 10^{7}$ & $-$ &
				$2.2709 \times 10^{-5}$ & $3.5577 \times 10^{-5}$ & $1.42 \times 10^{7}$ & $-$\\
				& $17$ & $7.1208 \times 10^{-3}$ & $1.1068 \times 10^{-2}$ & $4.39 \times 10^{1}$ & $1.16 \times 10^{7}$ & $2.2709 \times 10^{-5}$ & $3.5577 \times 10^{-5}$ & $4.39 \times 10^{1}$ & $1.19 \times 10^{7}$\\ 
				
				\hline
				
				$2^{20}+1$ & $0$ & $1.0754 \times 10^{-4}$ & $1.6712 \times 10^{-4}$ & $6.21 \times 10^{7}$ & $-$ &
				$8.6824 \times 10^{-8}$ & $1.3591 \times 10^{-7}$ & $5.98 \times 10^{7}$ & $-$\\
				& $18$ & $1.0754 \times 10^{-4}$ & $1.6712 \times 10^{-4}$ & $4.47 \times 10^{1}$ & $1.19 \times 10^{7}$ & $8.6814 \times 10^{-8}$ & $1.3589 \times 10^{-7}$ & $4.47 \times 10^{1}$ & $1.19 \times 10^{7}$\\ 
				
				\hline
				
				$2^{21}+1$ & $0$ & $1.6652 \times 10^{-6}$ & $2.5876 \times 10^{-6}$ & $2.23 \times 10^{8}$ & $-$ &
				$1.9291 \times 10^{-9}$ & $2.9177 \times 10^{-9}$ & $2.23 \times 10^{8}$ & $-$\\\
				& $19$ & $1.6652 \times 10^{-6}$ & $2.5877 \times 10^{-6}$ & $4.49 \times 10^{1}$ & $1.19 \times 10^{7}$ & $1.8549 \times 10^{-9}$ & $2.8085 \times 10^{-9}$ & $4.49 \times 10^{1}$ & $1.19 \times 10^{7}$\\ 
				
				\hline
			\end{tabular}	
			 \captionof{table}{Relative errors for \cref{ex:akvary} using DAT with $N_0=4$ and $s=1$ in \cref{alg:dat}. The grid increment used in $[0,1]$ is $(N-1)^{-1}$.
	 	}
			\label{tab:akvary}
		\end{center}
	}

	\begin{figure}[hbtp]
		 \begin{subfigure}[b]{0.3\textwidth}
		 \begin{tikzpicture}
		 \begin{axis}
		 [
		 xlabel={\footnotesize{$N\!\!=$}},
        xlabel style ={xshift=-2.3cm,yshift=0.48cm},
		ylabel={Relative errors},
		 height=0.9\textwidth,
		 width=\textwidth,
		ymin=10e-12, ymax=2,
		 ytick={10e-13,10e-11,10e-9,10e-7,10e-5,10e-3,1},
		 xmode=log,
		 log basis x={2},
		 ymode=log,
		 log basis y={10},
		 legend style={nodes={scale=0.6, transform shape}},
		 legend pos=outer north east
		 ]
		 \addplot[thick,mark=*,blue] coordinates
		 {(262145,4.8707e-1)
		 	(524289,7.1208e-3)
		 	(1048577,1.0754e-4)
		 	 (2097153,1.6652e-6)};
		 \addplot[thick,mark=*,red] coordinates
		 {(262145,7.6812e-1)
		 	(524289,1.1068e-2)
		 	(1048577,1.6712e-4)
		 	 (2097153,2.5877e-6)};
		\addplot[thick, dashed, mark=*,blue] coordinates
		{(262145,6.7606e-3)
			(524289,2.2709e-5)
			(1048577,8.6814e-8)
			 (2097153,1.8549e-9)};
		\addplot[thick, dashed, mark=*,red] coordinates
		{(262145,1.0602e-2)
			(524289,3.5577e-5)
			(1048577,1.3589e-7)
			 (2097153,2.8085e-9)};
		\end{axis}
		\node (B) at (1.1,0.75) {\tiny \scalebox{0.8}{$\mathcal{O}(N^{-8})$}};
		\node (A) at (1.9,1.6) {};
		\draw[thick,->] (B) edge (A);
		\node (D) at (3,2.95) {\tiny \scalebox{0.8}{$\mathcal{O}(N^{-6})$}};
		\node (C) at (2.25,2.1) {};
		\draw[thick,->] (D) edge (C);
		 \end{tikzpicture}
		 \end{subfigure}
	 	 \begin{subfigure}[b]{0.3\textwidth}
		 	 \raisebox{0.1cm}{\includegraphics[width=\textwidth]{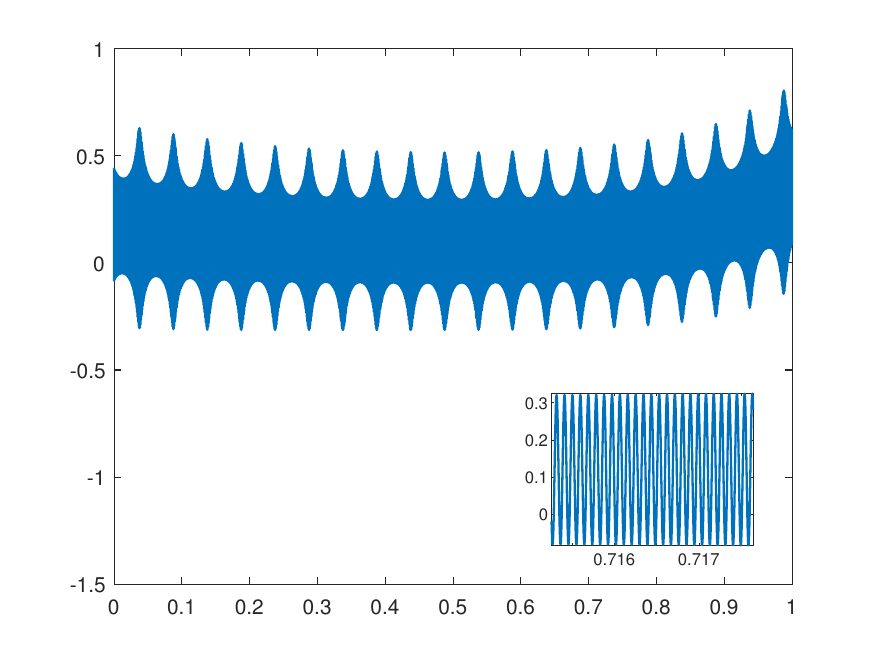}}
		 \end{subfigure}
		 \begin{subfigure}[b]{0.3\textwidth}
		 	 \raisebox{0.1cm}{\includegraphics[width=\textwidth]{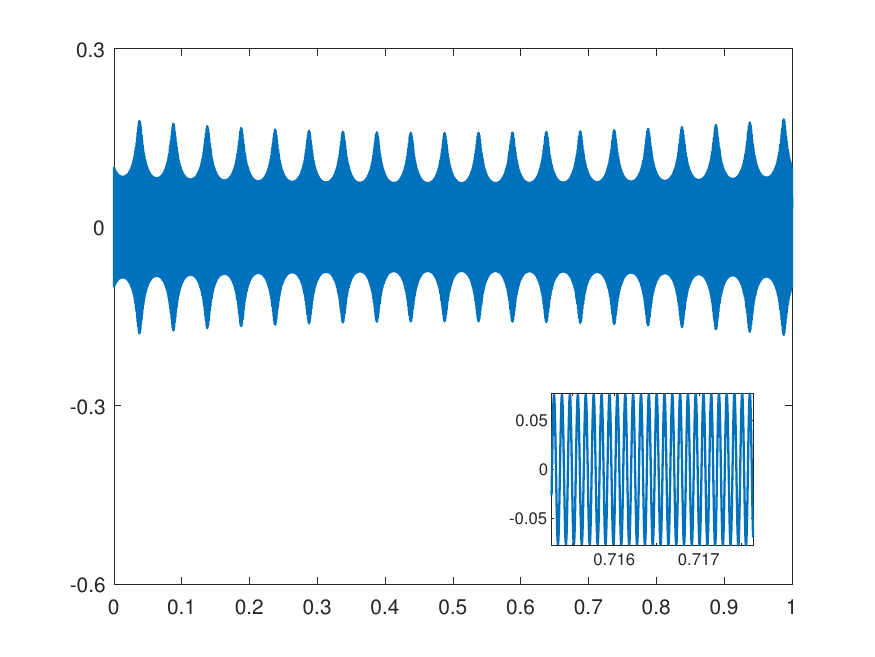}}
		 \end{subfigure}
		 \caption{\cref{ex:akvary}:
		Convergence plot (left) of DAT using the compact FDM  with order $M=6$ (solid) and $M=8$ (dashed) for errors $\frac{\|u_N-u_{2N-1}\|_{\infty}}{\|u_{2N-1}\|_{\infty}}$ (blue) and $\frac{\|u_N'-u_{2N-1}'\|_{\infty}}{\|u_{2N-1}'\|_{\infty}}$ (red). The displayed convergence rates are obtained by calculating $\log_2\left(
\frac{\|u_N-u_{2N-1}\|_{\infty}/\|u_{2N-1}\|_{\infty} }{ \|u_{2N-1}-u_{4N-3}\|_{\infty}/\|u_{4N-3}\|_{\infty} }\right)$ and $\log_2\left(\frac{\|u_N'-u_{2N-1}'\|_{\infty}/\|u_{2N-1}'\|_{\infty} }{\|u_{2N-1}'-u_{4N-3}'\|_{\infty}/\|u_{4N-3}'\|_{\infty}} \right)$. The real (middle) and imaginary (right) parts of $u_N$ with $N=2^{21}+1$, $\ell=19$ and $M=8$.}
		\label{fig:akvary}
	\end{figure}

\begin{example}\label{ex:newphet}
	\normalfont
	Consider
$[a(x)u'(x)]'+\kappa^2(x) u(x)=f(x), x\in (0,1)$
with the following coefficients
	\begin{align*}
	a & = (5+\sin(10 \pi x)) \chi_{[0,\frac{23}{100}) \cup [\frac{83}{100},1]} + (2+\sin(10 \pi x))\chi_{[\frac{23}{100},\frac{53}{100})} + (9+\sin(10 \pi x))\chi_{[\frac{53}{100},\frac{83}{100})},\\
	\kappa & = 2000 \left(e^{-x} \chi_{[0,\frac{23}{100})} + \chi_{[\frac{23}{100},\frac{53}{100}) \cup [\frac{83}{100},1]} + 0.5 e^{x}\chi_{[\frac{53}{100},\frac{83}{100})}\right),\\
	f & = 2^{21}\left(\cosh(x) \chi_{[0,\frac{23}{100})} + \sinh(x) \chi_{[\frac{23}{100},\frac{53}{100})} - \cosh(x) \chi_{[\frac{53}{100},\frac{83}{100})} - \sinh(x) \chi_{[\frac{83}{100},1]}\right),
	\end{align*}
and the boundary conditions $u'(0)=1$ and $\sqrt{5}u'(1)-2000iu(1)=0$.
The exact solution's analytic expression is unknown. See \cref{tab:elliptic} for the numerical performance measured by $\frac{\|u_N-u_{2N-1}\|_\infty}{\|u_{2N-1}\|_\infty}$ and
$\frac{\|u_N'-u_{2N-1}'\|_\infty}{\|u_{2N-1}'\|_\infty}$, and \cref{fig:elliptic} for the convergence plot and approximated solution $u_N$. As can be seen from \cref{tab:elliptic}, the convergence rates agree with the theoretical discussion in \cref{sec:fdm,sec:convergence}.
\end{example}

	{\tiny
		\begin{center}
			\begin{tabular}{c c | c c c c | c c c c}
				\hline
				\hline
				 \multicolumn{2}{c}{} \vline & \multicolumn{4}{c}{DAT using the compact FDM with order $M=6$} \vline & \multicolumn{4}{c}{DAT using the compact FDM with order $M=8$}\\
				\hline
				$N$ & $\ell$ & $\frac{\|u_N-u_{2N-1}\|_{\infty}}{\|u_{2N-1}\|_{\infty}}$ & $\frac{\|u_{N}'-u'_{2N-1}\|_{\infty}}{\|u_{2N-1}'\|_{\infty}}$ & Local CN & Link CN & $\frac{\|u_{N}-u_{2N-1}\|_{\infty}}{\|u_{2N-1}\|_{\infty}}$ & $\frac{\|u_{N}'-u'_{2N-1}\|_{\infty}}{\|u_{2N-1}'\|_{\infty}}$ & Local CN & Link CN\\
				\hline				
				$2^{10}+1$ & $0$ & $2.3932$ & $1.8969$ & $1.25 \times 10^{6}$ & $-$ &
				$4.8820 \times 10^{-2}$ & $1.1942 \times 10^{-1}$ & $4.51 \times 10^{5}$ & $-$\\	
				& $5$ & $2.3932$ & $1.8969$ & $2.66 \times 10^{4}$ & $3.44 \times 10^{3}$ &
				$4.8820 \times 10^{-2}$ & $1.1942 \times 10^{-1}$ & $6.47 \times 10^{3}$ & $2.01 \times 10^{3}$\\			 
				\hline
				$2^{11}+1$ & $0$ & $6.9794 \times 10^{-3}$ & $6.4621 \times 10^{-3}$ & $3.13 \times 10^{4}$ & $-$ &
				$1.6095\times 10^{-4}$ & $5.3061 \times 10^{-4}$ & $3.11 \times 10^{4}$ & $-$\\
				& $5$ & $6.9794 \times 10^{-3}$ & $6.4621 \times 10^{-3}$ & $9.00 \times 10^{3}$ & $6.92 \times 10^{3}$ &
				$1.6095 \times 10^{-4}$ & $5.3061 \times 10^{-4}$ & $9.00 \times 10^{3}$ & $2.48 \times 10^{3}$\\
				\hline
				$2^{12}+1$ & $0$ & $1.0216 \times 10^{-4}$ & $9.0880 \times 10^{-5}$ & $9.04 \times 10^{4}$ & $-$ &
				$6.6055 \times 10^{-7}$ & $2.1900 \times 10^{-6}$ & $9.04 \times 10^{4}$ & $-$\\
				& $5$ & $1.0216 \times 10^{-4}$ & $9.0880 \times 10^{-5}$ & $3.66 \times 10^{4}$ & $2.53 \times 10^{3}$ &
				$6.6055 \times 10^{-7}$ & $2.1900 \times 10^{-6}$ & $3.66\times 10^{4}$ & $2.51 \times 10^{3}$\\			 
				\hline
				$2^{13}+1$ & $0$ & $1.5468 \times 10^{-6}$ & $1.4036 \times 10^{-6}$ & $3.60 \times 10^{5}$ & $-$ &
				$2.6471 \times 10^{-9}$ & $8.6758 \times 10^{-9}$ & $3.60 \times 10^{5}$ & $-$\\
				& $5$ & $1.5468 \times 10^{-6}$ & $1.4036 \times 10^{-6}$ & $1.47 \times 10^{5}$ & $2.51 \times 10^{3}$ &
				$2.6448 \times 10^{-9}$ & $8.6783 \times 10^{-9}$ & $1.47 \times 10^{5}$ & $2.51 \times 10^{3}$\\
				\hline
			\end{tabular}	
			 \captionof{table}{Relative errors for \cref{ex:newphet} using DAT with $N_0=16$ and $s=1$ in \cref{alg:dat}. The grid increments used in $[0,\frac{23}{100}]$, $[\frac{23}{100},\frac{53}{100}]$, $[\frac{53}{100},\frac{83}{100}]$, and $[\frac{83}{100},1]$ are respectively $\frac{23}{25(N-1)}$, $\frac{6}{5(N-1)}$, $\frac{6}{5(N-1)}$, and $\frac{17}{25(N-1)}$.
}
			 \label{tab:elliptic}
		\end{center}
	}

	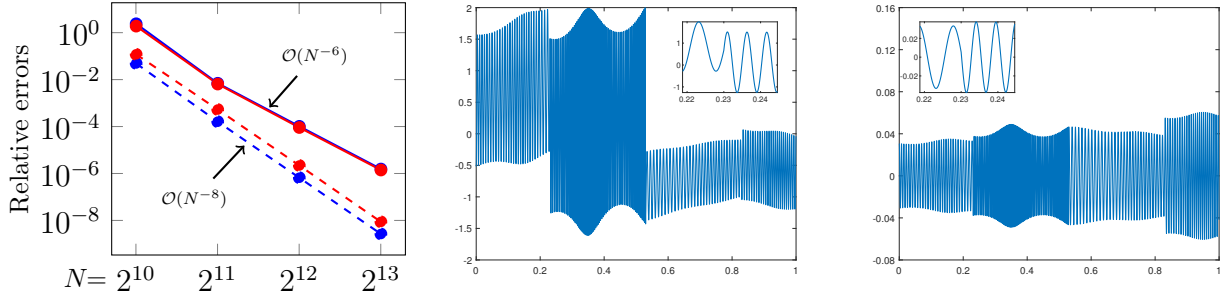
\begin{figure}[hbtp]
		 \begin{subfigure}[b]{0.3\textwidth}
		\begin{tikzpicture}
		\begin{axis}
		[
		 xlabel={\footnotesize{$N\!\!=$}},
        xlabel style ={xshift=-2.3cm,yshift=0.48cm},
		 ylabel={\footnotesize{Relative errors}},
		height=0.9\textwidth,
		width=\textwidth,
		xmode=log,
		log basis x={2},
		ymode=log,
		log basis y={10},
		 ytick={10e-13,10e-11,10e-9,10e-7,10e-5,10e-3,10e-1},
		 legend style={nodes={scale=0.6, transform shape}},
		 legend pos=outer north east,
		]	
		\addplot[thick, mark=*,blue] coordinates
		{(1025,2.3932)
			(2049,6.9794e-3)
			(4097,1.0216e-4)
			(8193,1.5468e-6)};
		 \addplot[thick,mark=*,red] coordinates
		{(1025,1.8969)
			(2049,6.4621e-3)
			(4097,9.0880e-5)
			(8193,1.4036e-6)};
		\addplot[thick, dashed, mark=*,blue] coordinates
		{(1025,4.8820e-2)
			(2049,1.6095e-4)
			(4097,6.6055e-7)
			(8193,2.6448e-9)};
		\addplot[thick, dashed, mark=*,red] coordinates
		{(1025,1.1942e-1)
			(2049,5.3061e-4)
			(4097,2.1900e-6)
			(8193,8.6783e-9)};
		\end{axis}	
		\node (B) at (1.1,0.75) {\tiny \scalebox{0.8}{$\mathcal{O}(N^{-8})$}};
		\node (A) at (1.9,1.6) {};
		\draw[thick,->] (B) edge (A);
		\node (D) at (2.7,2.65) {\tiny \scalebox{0.8}{$\mathcal{O}(N^{-6})$}};
		\node (C) at (1.97,1.82) {};
		\draw[thick,->] (D) edge (C);
		\end{tikzpicture}
		 \end{subfigure}
	 	 \begin{subfigure}[b]{0.3\textwidth}
	 	 	 \raisebox{0.1cm}{\includegraphics[width=\textwidth]{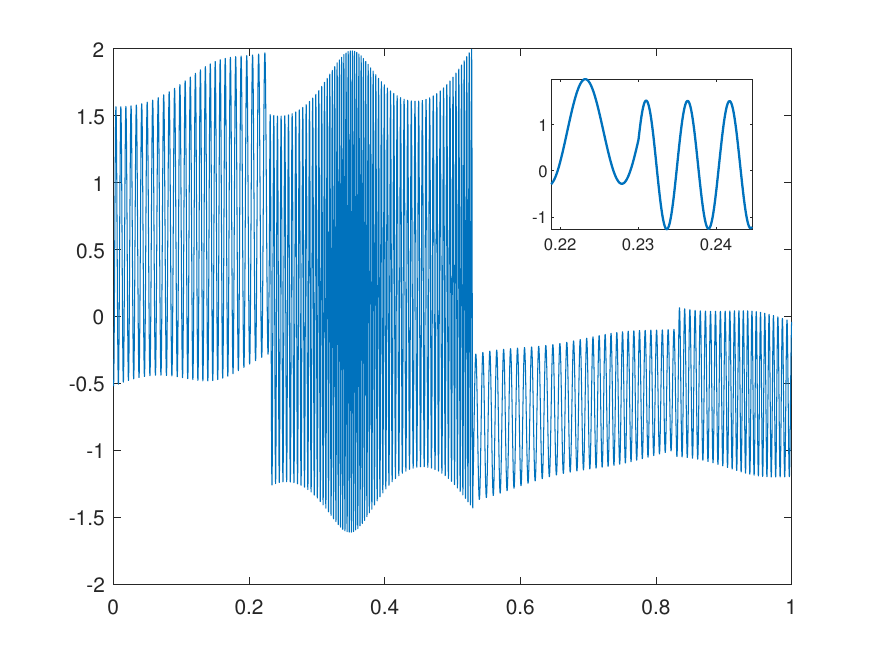}}
	 	 \end{subfigure}
	 	 \begin{subfigure}[b]{0.3\textwidth}
	 	 	 \raisebox{0.1cm}{\includegraphics[width=\textwidth]{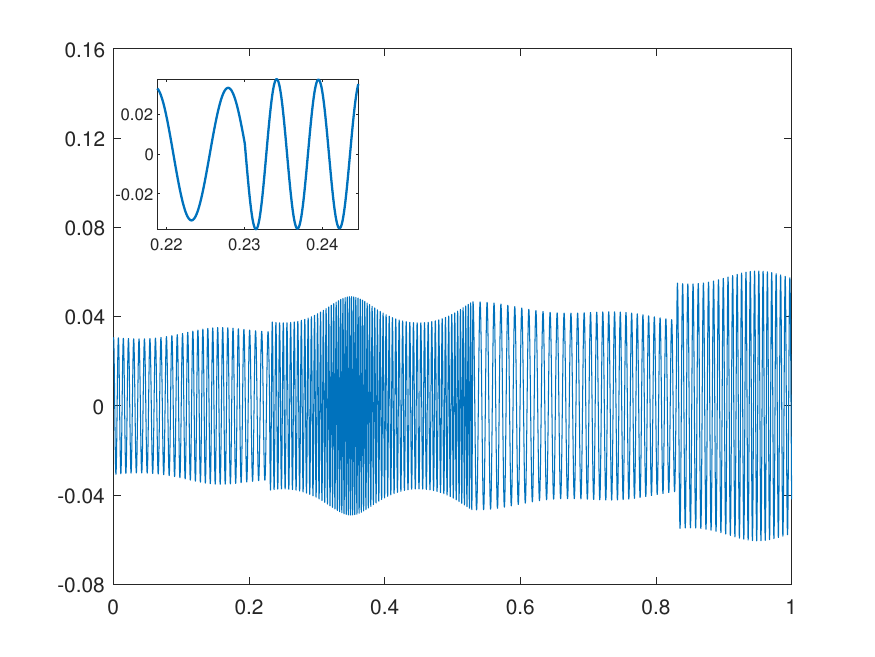}}
	 	 \end{subfigure}
		 \caption{\cref{ex:newphet}:
Convergence plot (left) of DAT using the compact FDM  with order $M=6$ (solid) and $M=8$ (dashed) for errors $\frac{\|u_N-u_{2N-1}\|_{\infty}}{\|u_{2N-1}\|_{\infty}}$ (blue) and
$\frac{\|u_N'-u_{2N-1}'\|_{\infty}}{\|u_{2N-1}'\|_{\infty}}$ (red). The displayed convergence rates are obtained by calculating $\log_2\left(\frac{\|u_N-u_{2N-1}\|_{\infty}/\|u_{2N-1}\|_{\infty}  }{ \|u_{2N-1}-u_{4N-3}\|_{\infty}/\|u_{4N-3}\|_{\infty}}\right)$ and $\log_2\left(\frac{\|u_N'-u_{2N-1}'\|_{\infty} /\|u_{2N-1}'\|_{\infty} }{ \|u_{2N-1}'-u_{4N-3}'\|_{\infty}/\|u_{4N-3}'
\|_{\infty}}\right)$. The real (middle) and imaginary (right) parts of $u_N$ with $N=2^{13}+1$, $\ell=5$ and $M=8$.}
		\label{fig:elliptic}
	\end{figure}

\subsection{Numerical experiments on 2D Helmholtz equations}
\label{subsec:2D}

Separable 2D Helmholtz equations can be converted into a sequence of 1D Helmholtz problems, to which we may apply DAT as demonstrated below.

\begin{example}\label{ex:2D}
	\normalfont
	Let $D_{1}:=\{(r,\theta):1 \le r <2, \theta \in [0,2\pi)\}$, $D_{2}:=\{(r,\theta):2 \le r \le 4, \theta \in [0,2\pi)\}$, and $D:=D_{1} \cup D_{2}$. Consider the following 2D Helmholtz equation $\nabla\cdot( \nabla u)+\kappa^2 u=0$ on the domain $D$, which can be rewritten
	in the polar coordinate system as follows:
	\begin{align*}
	& \tfrac{1}{r} \tfrac{\partial}{\partial r}\left(r \tfrac{\partial u}{\partial r}\right) + \tfrac{1}{r^{2}} \tfrac{\partial^{2}u}{\partial \theta^{2}} + \kappa^{2} u =0 \quad \text{on} \quad D,\\
	& \tfrac{\partial u}{\partial r}\rvert_{r=1}=0,\quad
	\left(\tfrac{\partial u}{\partial r} + \left(\tfrac{1}{2r} - 100i\right) u\right)\rvert_{r=4} = \left(\tfrac{\partial u_{I}}{\partial r} + \left(\tfrac{1}{2r} - 100i\right) u_{I}\right)\rvert_{r=4},
	\end{align*}
	where $\kappa=50\chi_{D_{1}} + 100\chi_{D_{2}}=\kappa_0(r)$ with $\kappa_0:=50\chi_{[1,2)}+100\chi_{[2,4]}$, $u_{I}:=\sum_{m=0}^{\infty} i^{m}(\dirac_{0,m}+2(1-\dirac_{0,m})) J_{m}(100r) \cos(m\theta)$, and $J_m(\cdot)$ is the Bessel function of the first kind of order $m$.
	Using the method outlined in \cite[Section 7.1]{LBPT05},
	the exact solution $\eu$ is given by the series $\eu=\sum_{m=0}^\infty r^{-1/2} v_m(r) \cos (m\theta)$, where $v_m, m\in \NN$ satisfy the following 1D Helmholtz equations:
	\begin{equation} \label{2Dto1D}
		\begin{aligned}
			& v_{m}''+\left(\kappa_0^2
			 -r^{-2}\left(m^{2}-\tfrac{1}{4}\right)\right)v_{m} = 0, \quad  r \in(1,4),\quad \text{with} \quad \left(v_{m}'- \tfrac{1}{2}v_{m}\right)\rvert_{r=1}=0,\\
			& \left(v_{m}'- 100i v_{m}\right)|_{r=4}= 2i^{m}(\dirac_{0,m}+2(1-\dirac_{0,m})) \left((J_{m}(100r))'|_{r=4}+\left(\tfrac{1}{8}-100i\right)J_{m}(400)\right).
		\end{aligned}
	\end{equation}
 	In particular, for each $m \in \mathbb{N}_0$, $v_m$ has the following  analytic expression
	\[
v_{m}=r^{1/2}\Big(A_m J_m(50r) + B_m Y_m(50r)\Big)\chi_{[1,2)} + r^{1/2}\Big(C_m J_m(100r) + D_m Y_m(100r)\Big)\chi_{[2,4]},
	\]
	where $Y_m(\cdot)$ is the Bessel function of the second kind of order $m$ and all the coefficients $A_m, B_m, C_m, D_m$ are uniquely determined by solving a system of linear equations that arises from imposing the boundary conditions and the transmission conditions: $v_m(2-)=v_m(2+)$ and $v'_m(2-)=v'_m(2+)$.
	
	Our approximated solution then takes the form $u_{N}=\sum_{m=0}^{640} r^{-1/2}v_{m,N} \cos(m\theta)$, where $v_{m,N}$ is the approximated solution to $v_{m}$ in \eqref{2Dto1D} using $N$ points. In all cases, we use $2049$ points to discretize the angle $\theta$ in our exact and approximated solutions. Also note that the following ``Local CN" and ``Link CN" record the maximum condition number of all local problems and all $m=0,\dots,640$.  See \cref{tab:2D} for the numerical performance measured by both $\frac{\|u_{N}-\eu\|_\infty}{\|\eu\|_\infty}$ and
	 $\frac{\|u_{N}-\eu\|_2}{\|\eu\|_2}$, where we use the first $641$ terms of $\eu$ (i.e., $\eu\approx\sum_{m=0}^{640} r^{-1/2} v_m(r) \cos(m\theta)$), and \cref{fig:2D} for the convergence plot and approximated solution $u_N$. Due to the separation of variables, the convergence rates observed in the plot are solely driven by the convergence rates that take place in each 1D problem. As can be seen, the convergence rates agree with the theoretical discussion in \cref{sec:fdm,sec:convergence}.
\end{example}

	{\tiny
	\begin{center}
		\begin{tabular}{c c | c c c c | c c c c}
			\hline
			\hline
			\multicolumn{2}{c}{} \vline & \multicolumn{4}{c}{DAT using the compact FDM with order $M=6$} \vline & \multicolumn{4}{c}{DAT using the compact FD with order $M=8$}\\
			\hline
			$N$ & $\ell$ & $\frac{\|u_N-\eu\|_{\infty}}{\|\eu\|_{\infty}}$  & $\frac{\|u_N-\eu\|_{2}}{\|\eu\|_{2}}$ & Local CN & Link CN  & $\frac{\|u_N-\eu\|_{\infty}}{\|\eu\|_{\infty}}$  & $\frac{\|u_N-\eu\|_{2}}{\|\eu\|_{2}}$ & Local CN & Link CN\\
			\hline			
			$2^{8}+1$ & $0$ & $1.0461\times 10^{-1}$ & $6.6021 \times 10^{-2}$ & $1.84 \times 10^{5}$ & $-$ &
			$6.8220\times 10^{-3}$ & $3.4958 \times 10^{-3}$ & $1.79 \times 10^{5}$ & $-$\\
			& $5$ & $1.0461\times 10^{-1}$ & $6.6021 \times 10^{-2}$ & $4.79\times 10^{4}$ & $3.05 \times 10^{5}$ &
			$6.8220\times 10^{-3}$ & $3.4958 \times 10^{-3}$ & $9.13 \times 10^{4}$ & $2.81 \times 10^{5}$\\
			\hline
			$2^{9}+1$ & $0$ & $1.2885 \times 10^{-3}$ & $7.6950 \times 10^{-4}$ & $4.04\times 10^{4}$ & $-$ &
			$2.7208 \times 10^{-5}$ & $1.4102 \times 10^{-5}$ & $4.04 \times 10^{4}$ & $-$\\
			& $6$ & $1.2885 \times 10^{-3}$ & $7.6950 \times 10^{-4}$ & $7.06\times 10^{2}$ & $2.78 \times 10^{5}$ &
			$2.7208 \times 10^{-5}$ & $1.4102 \times 10^{-5}$ & $7.05 \times 10^{2}$ & $2.78 \times 10^{5}$\\
			\hline
			$2^{10}+1$ & $0$ & $1.9204 \times 10^{-5}$ & $1.1290\times 10^{-5}$ & $1.52\times 10^{5}$ & $-$ &
			$1.0825 \times 10^{-7}$ & $5.5999\times 10^{-8}$ & $1.52\times 10^{5}$ & $-$\\
			& $7$ & $1.9204 \times 10^{-5}$ & $1.1290\times 10^{-5}$ & $4.78\times 10^{1}$ & $2.78 \times 10^{5}$ &
			$1.0825 \times 10^{-7}$ & $5.5999\times 10^{-8}$ & $4.78\times 10^{1}$ & $2.78 \times 10^{5}$\\
			\hline
			$2^{11}+1$ & $0$ & $2.9671 \times 10^{-7}$ & $1.7375 \times 10^{-7}$ & $6.01 \times 10^{5}$ & $-$ &
			$4.2441 \times 10^{-10}$ & $2.1932 \times 10^{-10}$ & $6.01 \times 10^{5}$ & $-$\\
			& $8$ & $2.9671 \times 10^{-7}$ & $1.7375 \times 10^{-7}$ & $4.56 \times 10^{1}$ & $2.78 \times 10^{5}$ &
			$4.2435 \times 10^{-10}$ & $2.1934 \times 10^{-10}$ & $4.56 \times 10^{1}$ & $2.78 \times 10^{5}$\\	 
			\hline
		\end{tabular}	
		 \captionof{table}{Relative errors for \cref{ex:2D} using DAT with $N_0=8$ and $s=1$ in \cref{alg:dat}. The grid increments used in each $[1,2]$ and $[2,4]$ are respectively $2(N-1)^{-1}$ and $4(N-1)^{-1}$.
}
		\label{tab:2D}
	\end{center}
	}

		\begin{figure}[htbp]
		\centering
		 \begin{subfigure}[b]{0.3\textwidth}
		\begin{tikzpicture}
		\begin{axis}
		[
		 xlabel={\footnotesize{$N\!\!=$}},
        xlabel style ={xshift=-2.3cm,yshift=0.48cm},
		 ylabel={\footnotesize{Relative errors}},
		height=0.9\textwidth,
		width=\textwidth,
		xmode=log,
		log basis x={2},
		ymode=log,
		log basis y={10},
		 ytick={10e-10,10e-8,10e-6,10e-4,10e-2,10e-0},
		legend style={nodes={scale=0.6, transform shape}},
		legend pos=outer north east
		]
		 \addplot[thick,mark=*,blue] coordinates
		{(257,1.0461e-1)
			(513,1.2885e-3)
			(1025,1.9204e-5)
			(2049,2.9671e-7)};
		 \addplot[thick,mark=*,red] coordinates
		{(257,6.6021e-2)
			(513,7.6950e-4)
			(1025,1.1290e-5)
			(2049,1.7375e-7)};
		 \addplot[thick,dashed,mark=*,blue] coordinates
		{(257,6.8220e-3)
			(513,2.7208e-5)
			(1025,1.0825e-7)
			(2049,4.2435e-10)};
		 \addplot[thick,dashed,mark=*,red] coordinates
		{(257,3.4958e-3)
			(513,1.4102e-5)
			(1025,5.5599e-8)
			(2049,2.1932e-10)};
		\end{axis}
		\node (B) at (1.1,0.75) {\tiny \scalebox{0.8}{$\mathcal{O}(N^{-8})$}};
		\node (A) at (1.9,1.6) {};
		\draw[thick,->] (B) edge (A);
		\node (D) at (2.85,2.8) {\tiny \scalebox{0.8}{$\mathcal{O}(N^{-6})$}};
		\node (C) at (2.1,1.95) {};
		\draw[thick,->] (D) edge (C);
		\end{tikzpicture}
		\end{subfigure}
		 \begin{subfigure}[b]{0.3\textwidth}
			 \raisebox{0.1cm}{\includegraphics[width=\textwidth]{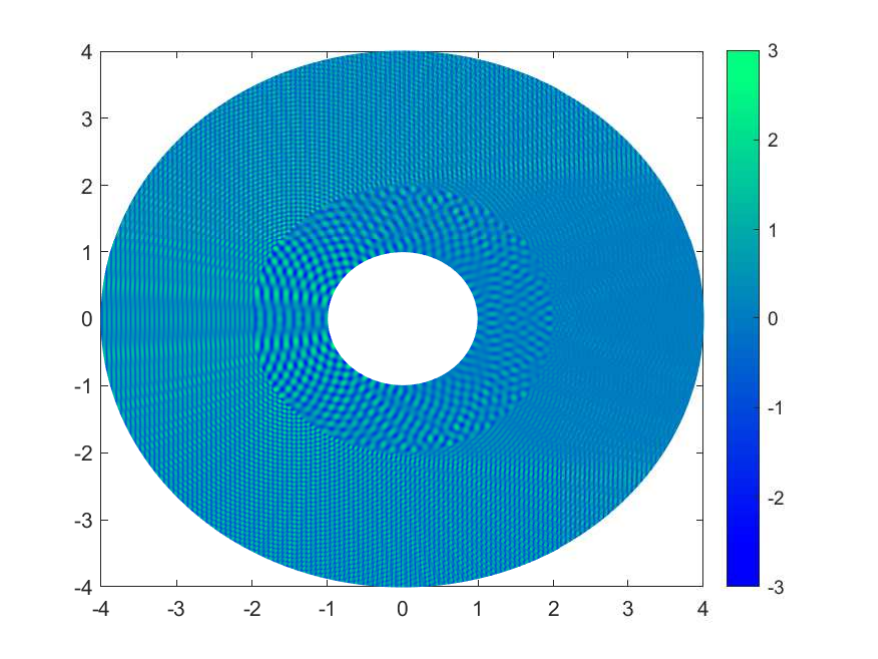}}
		\end{subfigure}
		 \begin{subfigure}[b]{0.3\textwidth}
			 \raisebox{0.1cm}{\includegraphics[width=\textwidth]{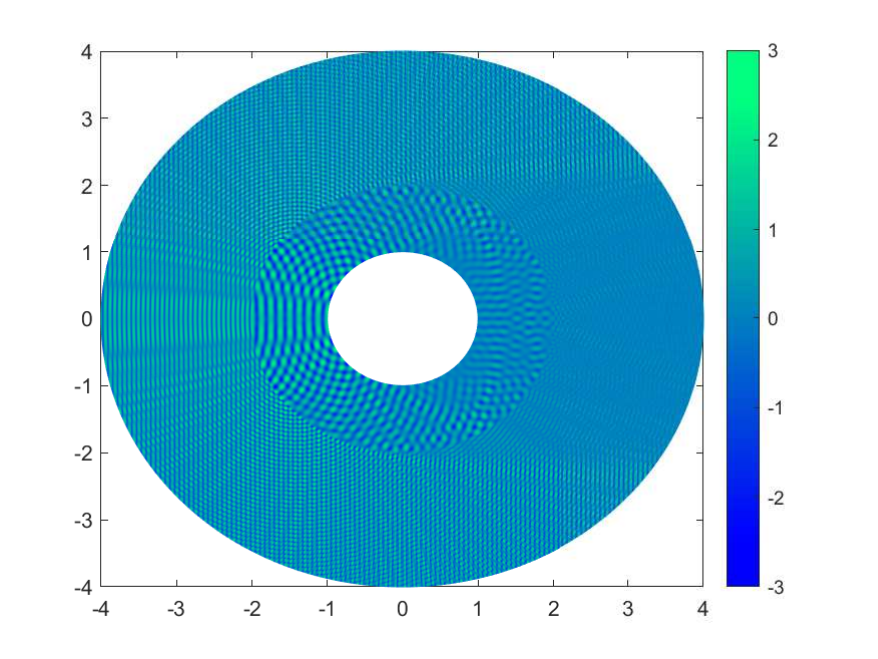}}
		\end{subfigure}
		\caption
		{\cref{ex:2D}:
	Convergence plot (left) of DAT using the compact FDM  with order $M=6$ (solid) and $M=8$ (dashed) for relative errors $\frac{\|u_N-\eu\|_{\infty}}{\|\eu\|_{\infty}}$ (red) and
	 $\frac{\|u_N-\eu\|_{2}}{\|\eu\|_{2}}$ (blue). The displayed convergence rates are obtained by calculating $\log_2\left(\frac{\|u_N-\eu\|_{\infty}}{\|u_{2N}-\eu\|_{\infty}}\right)$ and $\log_2\left(\frac{\|u_N-\eu\|_{2}}{\|u_{2N}-\eu\|_{2}}
\right)$. The real (middle) and imaginary (right) parts of $u_N$ with $N=2^{11}+1$, $\ell=8$ and $M=8$.}
		\label{fig:2D}
	\end{figure}
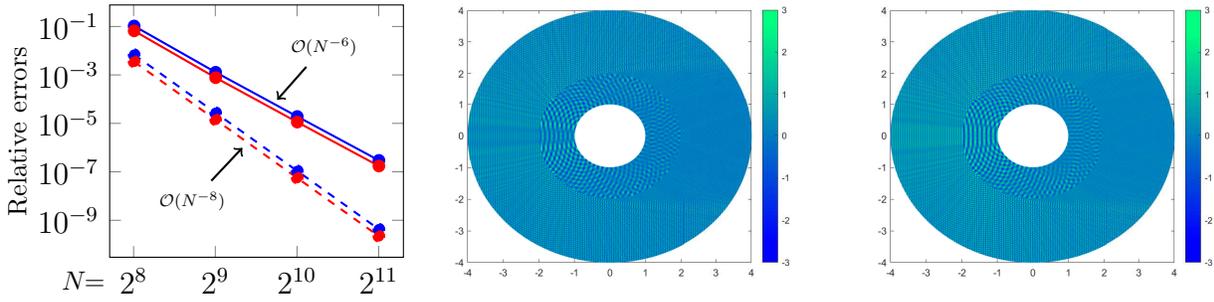

\begin{example}\label{ex:2D:f}
	\normalfont
	Let $D_{1}:=\{(r,\theta):1 \le r <3, \theta \in [0,2\pi)\}$, $D_{2}:=\{(r,\theta):3 \le r \le 4, \theta \in [0,2\pi)\}$, and $D:=D_{1} \cup D_{2}$. Consider the following 2D Helmholtz equation $\nabla\cdot( \nabla u)+\kappa^2 u=f$ on the domain $D$, which can be rewritten
	in the polar coordinate system as follows:
	\begin{align*}
		& \tfrac{1}{r} \tfrac{\partial}{\partial r}\left(r \tfrac{\partial u}{\partial r}\right) + \tfrac{1}{r^{2}} \tfrac{\partial^{2}u}{\partial \theta^{2}} + \kappa^{2} u =f \quad \text{on} \quad D \quad \text{with} \quad 	 u\rvert_{r=1}=\tfrac{\sin(4\theta)}{\sqrt{\pi}} \quad \text{and} \quad
		\left(\tfrac{\partial u}{\partial r} -50iu\right)\rvert_{r=4} = 0,
	\end{align*}
	where $\kappa=400\chi_{D_{1}} + 50\chi_{D_{2}}=\kappa_0(r)$ with $\kappa_0(r):=400\chi_{[1,3)} + 50\chi_{[3,4]}$ and $f=(10^5 J_1(r)\chi_{D_{1}} + 10^4J_0(r)\chi_{D_{2}})$.
	By applying the separation of variables twice, the exact solution $\eu$ in the polar coordinate system to the above 2D Helmholtz equation is given by the series $\eu=\sum_{m\in \mathbb{Z}} r^{-1/2} v_m(r) e^{im\theta}$, where for each $ m\in \mathbb{Z}$, $v_m$ satisfies
\begin{equation} \label{2Dto1D:f}
\begin{aligned}
	& v_{m}''+\left(\kappa_0^2
	 -r^{-2}\left(m^{2}-\tfrac{1}{4}\right)\right)v_{m} = r^{1/2} f_m(r), \quad  r \in(1,4), \\
	& v_{m}\rvert_{r=1}=\tfrac{i}{\sqrt{2}} (-\delta_{4,m}+\delta_{-4,m}), \quad \left(v_{m}'- \left(\tfrac{1}{8}+50i \right) v_{m}\right)\rvert_{r=4}=0,
\end{aligned}
\end{equation}
and
$f_m(r):=(2\pi)^{-1/2}\int_{0}^{2\pi} f(r,\theta)e^{im\theta} d\theta$ can be efficiently computed by FFT.
Note that $f_m$ are zero except for $m=0$.
Since $v_m$ are zero for $m\in \mathbb{Z}\backslash\{0,\pm 4\}$,
our approximated solution is of the form $u_{N}=(2\pi r)^{-1/2}v_{0,N} + (2\pi r)^{-1/2}\left(v_{4,N}e^{i4\theta}+v_{-4,N}e^{-i4\theta}\right)$, where $v_{m,N}$ is the approximated solution to $v_{m}$ in \eqref{2Dto1D:f} using $N$ points. We use $2049$ points to discretize the angle $\theta$ in our approximated solutions. Note that the following ``Local CN" and ``Link CN" record the maximum condition number of all local and linking problems, and all $m = 0,\pm4$.  See \cref{tab:2D:f} for the numerical performance measured by both $\frac{\|u_{N}-u_{2N}\|_\infty}{\|u_{2N}\|_\infty}$ and $\frac{\|u_{N}-u_{2N}\|_2}{\|u_{2N}\|_2}$, and \cref{fig:2D:f} for the convergence plot and approximated solution $u_N$. Due to the separation of variables, the convergence rates observed in the plot are solely driven by the convergence rates that take place in each 1D problem. As can be seen from \cref{tab:2D:f}, the convergence rates agree with the theoretical discussion in \cref{sec:fdm,sec:convergence}.
\end{example}

	{\tiny
	\begin{center}
		\begin{tabular}{c c | c c c c | c c c c}
			\hline
			\hline
			\multicolumn{2}{c}{} \vline & \multicolumn{4}{c}{DAT using the compact FDM with order $M=6$} \vline & \multicolumn{4}{c}{DAT using the compact FD with order $M=8$}\\
			\hline
			$N$ & $\ell$ & $\frac{\|u_{N}-u_{2N}\|_{\infty}}{\|u_{2N}\|_{\infty}}$  & $\frac{\|u_{N}-u_{2N}\|_{2}}{\|u_{2N}\|_{2}}$ & Local CN & Link CN  & $\frac{\|u_{N}-u_{2N}\|_{\infty}}{\|u_{2N}\|_{\infty}}$  & $\frac{\|u_{N}-u_{2N}\|_{2}}{\|u_{2N}\|_{2}}$ & Local CN & Link CN\\
			\hline			
			$2^{10}$ & $0$ & $2.8213\times 10^{-1}$ & $2.0128 \times 10^{-1}$ & $1.54 \times 10^{6}$ & $-$ &
			$1.4812\times 10^{-2}$ & $1.1534 \times 10^{-2}$ & $1.47 \times 10^{6}$ & $-$\\
			& $7$ & $2.8213\times 10^{-1}$ & $2.0128 \times 10^{-1}$ & $1.61\times 10^{2}$ & $5.26 \times 10^{3}$ &
			$1.4812\times 10^{-2}$ & $1.1534 \times 10^{-2}$ & $1.43 \times 10^{2}$ & $4.95 \times 10^{3}$\\
			\hline
			$2^{11}$ & $0$ & $5.6999 \times 10^{-3}$ & $4.3636 \times 10^{-3}$ & $1.21\times 10^{5}$ & $-$ &
			$4.1235 \times 10^{-5}$ & $3.2294 \times 10^{-5}$ & $1.22 \times 10^{5}$ & $-$\\
			& $8$ & $5.6999 \times 10^{-3}$ & $4.3636 \times 10^{-3}$ & $7.06\times 10^{2}$ & $4.98 \times 10^{3}$ &
			$4.1235 \times 10^{-5}$ & $3.2294 \times 10^{-5}$ & $7.05 \times 10^{2}$ & $4.98 \times 10^{3}$\\
			\hline
			$2^{12}$ & $0$ & $8.3447 \times 10^{-5}$ & $6.3909\times 10^{-5}$ & $4.85\times 10^{5}$ & $-$ &
			$1.5005 \times 10^{-7}$ & $1.1778\times 10^{-7}$ & $4.86\times 10^{5}$ & $-$\\
			& $9$ & $8.3447 \times 10^{-5}$ & $6.3909\times 10^{-5}$ & $4.50\times 10^{1}$ & $4.98 \times 10^{3}$ &
			$1.5005 \times 10^{-7}$ & $1.1778\times 10^{-7}$ & $4.50\times 10^{1}$ & $4.98 \times 10^{3}$\\
			\hline
			$2^{13}$ & $0$ & $1.2805 \times 10^{-6}$ & $9.8055 \times 10^{-7}$ & $1.95 \times 10^{6}$ & $-$ &
			$5.8547 \times 10^{-10}$ & $4.5911 \times 10^{-10}$ & $1.95 \times 10^{6}$ & $-$\\
			& $10$ & $1.2805 \times 10^{-6}$ & $9.8055 \times 10^{-7}$ & $4.50 \times 10^{1}$ & $4.98 \times 10^{3}$ &
			$5.8767 \times 10^{-10}$ & $4.6687 \times 10^{-10}$ & $4.50 \times 10^{1}$ & $4.98 \times 10^{3}$\\	 
			\hline
		\end{tabular}	
		 \captionof{table}{Relative errors for \cref{ex:2D:f} using DAT with $N_0=8$ and $s=1$ in \cref{alg:dat}. The grid increments used in each $[1,3]$ and $[3,4]$ are respectively $4N^{-1}$ and $2N^{-1}$.}
	\label{tab:2D:f}
	\end{center}
	}

	\begin{figure}[htbp]
		\centering
		 \begin{subfigure}[b]{0.3\textwidth}
			\begin{tikzpicture}
				\begin{axis}
					[
					 xlabel={\footnotesize{$N\!\!=$}},
					xlabel style ={xshift=-2.3cm,yshift=0.48cm},
					 ylabel={\footnotesize{Relative errors}},
					 height=0.9\textwidth,
					 width=\textwidth,
					xmode=log,
					log basis x={2},
					ymode=log,
					log basis y={10},
					 ytick={10e-10,10e-8,10e-6,10e-4,10e-2,10e-0},
					legend style={nodes={scale=0.6, transform shape}},
					legend pos=outer north east
					]
					 \addplot[thick,mark=*,blue] coordinates
					 {(1024,2.8213e-1)
						 (2048,5.6999e-3)
						 (4096,8.3447e-5)
						 (8192,1.2805e-6)};
					 \addplot[thick,mark=*,red] coordinates
					 {(1024,2.0128e-1)
						 (2048,4.3636e-3)
						 (4096,6.3909e-5)
						 (8192,9.8055e-7)};
					 \addplot[thick,dashed,mark=*,blue] coordinates
					 {(1024,1.4812e-2)
						 (2048,4.1235e-5)
						 (4096,1.5005e-7)
						 (8192,5.8767e-10)};
					 \addplot[thick,dashed,mark=*,red] coordinates
					 {(1024,1.1534e-2)
						 (2048,3.2294e-5)
						 (4096,1.1778e-7)
						 (8192,4.6687e-10)};
				\end{axis}
				\node (B) at (1.1,0.75) {\tiny \scalebox{0.8}{$\mathcal{O}(N^{-8})$}};
				\node (A) at (1.9,1.6) {};
				\draw[thick,->] (B) edge (A);
				\node (D) at (2.9,2.85) {\tiny \scalebox{0.8}{$\mathcal{O}(N^{-6})$}};
				\node (C) at (2.15,2.05) {};
				\draw[thick,->] (D) edge (C);
			\end{tikzpicture}
		\end{subfigure}
		 \begin{subfigure}[b]{0.3\textwidth}
			 \raisebox{0.1cm}{\includegraphics[width=\textwidth]{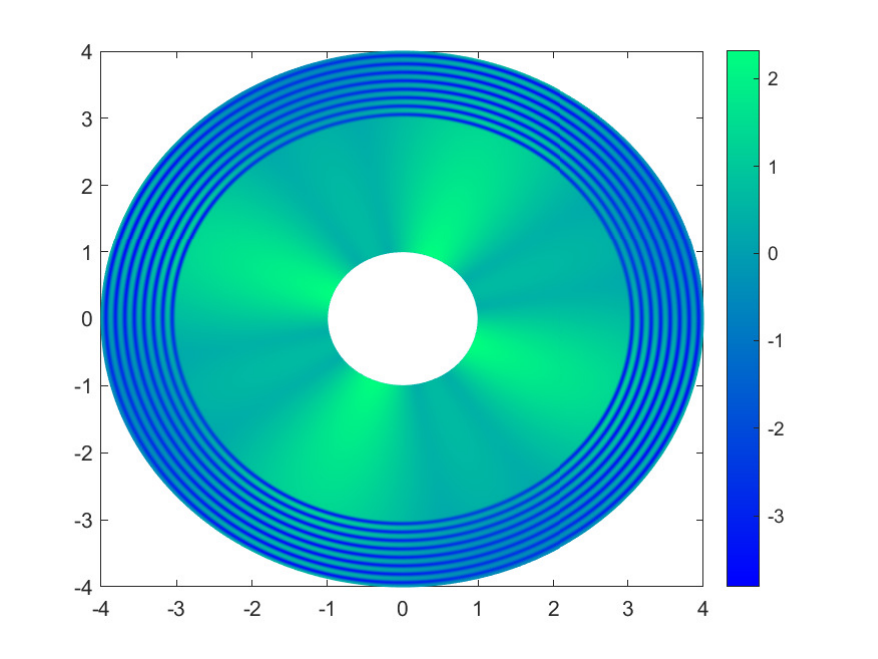}}
		\end{subfigure}
		 \begin{subfigure}[b]{0.3\textwidth}
			 \raisebox{0.1cm}{\includegraphics[width=\textwidth]{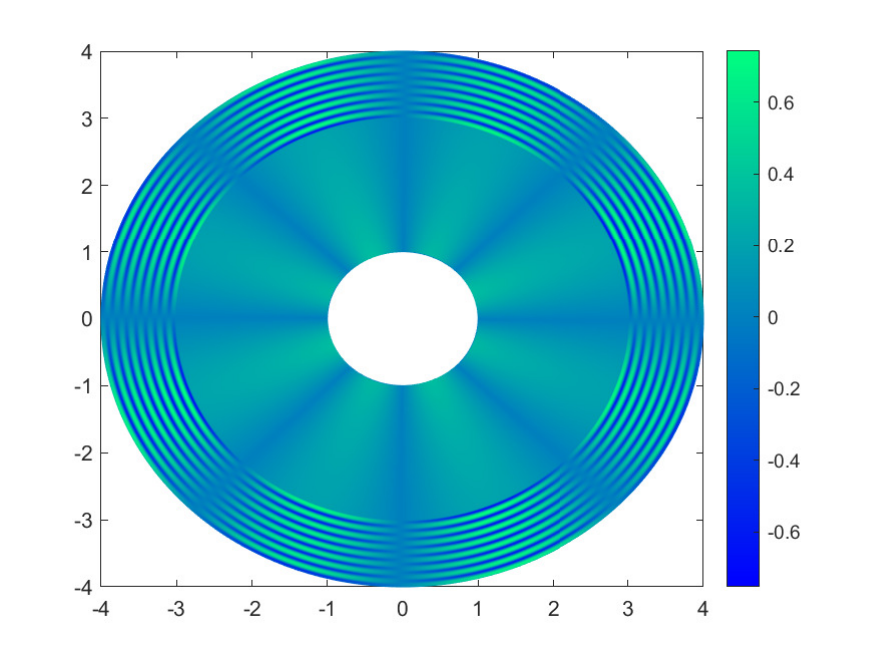}}
		\end{subfigure}
		\caption
		{\cref{ex:2D:f}:
			Convergence plot (left) of DAT using the compact FDM  with order $M=6$ (solid) and $M=8$ (dashed) for errors $\frac{\|u_{N}-u_{2N}\|_{\infty}}{\|u_{2N}\|_{\infty}}$ (red) and
			 $\frac{\|u_{N}-u_{2N}\|_{2}}{\|u_{2N}\|_{2}}$ (blue). The displayed convergence rates are obtained by calculating $\log_2\left(\frac{\|u_{N}-u_{2N}\|_{\infty} /\|u_{2N}\|_{\infty} }{
 \|u_{2N}-u_{4N}\|_{\infty}/\|u_{4N}\|_{\infty}}\right)$ and $\log_2\left(\frac{\|u_{N}-u_{2N}\|_{2}/\|u_{2N}\|_{2}
 }{ \|u_{2N}-u_{4N}\|_{2}/\|u_{4N}\|_{2}}\right)$. The real (middle) and imaginary (right) parts of $u_N$ with $N=2^{13}$, $\ell=10$ and $M=8$.}
		\label{fig:2D:f}
	\end{figure}

\begin{example} \label{ex:2D:waveguide}
	\normalfont
	Consider the following 2D Helmholtz equation
$\nabla \cdot (\nabla u) + \kappa^2 u = f$ on $\Omega=(0,1)^{2}$,
	where
{\footnotesize{
	\begin{align*}
		\kappa_0(x) & := 2^{\frac{5}{2}}(900000 x^5 - 2250000 x^4 + 2130000x^3 - 945000x^2 + 198450 x - 15445)^{\frac{1}{2}} \chi_{[\frac{3}{10},\frac{7}{10})}
+ 2^7 (\chi_{[0,\frac{3}{10})}  +2 \chi_{[\frac{7}{10},1]}),\\
		f(x,y) & := 5000 \sqrt{2} \left((6x+1)\cos(6.5\pi y)\chi_{[0,\frac{3}{10})\times[0,1]} + 2\cos(25.5\pi y)\chi_{[0,\frac{3}{10})\times[0,1]}
+ (-6x+7)\cos(6.5\pi y) \chi_{[\frac{7}{10},1]\times[0,1]}\right),
	\end{align*}
}}
and $\kappa(x,y):=\kappa_0(x)$ with the following boundary conditions
	\[
	\begin{split}
		& \tfrac{\partial u}{\partial \nu} = 0 \quad \mbox{on} \quad (0,1) \times \{0\},
		\qquad \tfrac{\partial u}{\partial \nu} - i2^8u=0 \quad \mbox{on} \quad \{1\} \times (0,1), \qquad u=0 \quad \mbox{on} \quad (0,1) \times \{1\},\\
		&
		\mbox{and} \qquad u = \sqrt{2}(\cos(14.5\pi y) + \cos(30.5\pi y)) \quad \mbox{on}  \quad \{0\} \times (0,1).
	\end{split}
	\]
	By the separation of variables, the exact solution $\eu$ to the above 2D Helmholtz equation is given by the series $\eu(x,y)=
\sum_{m=0}^\infty \sqrt{2} v_m(x) \cos((m+1/2)\pi y)$, where for each $m\in \NN$, $v_m$ satisfies
\[
v_{m}''+(\kappa_0^2 - (m+\tfrac{1}{2})^2 \pi^2)v_{m} = f_m(x), \quad  x \in(0,1) \quad \mbox{with} \quad v_{m}(0)=\dirac_{14,m}+\dirac_{30,m}, \quad v_{m}'(1)- 2^8 i v_{m}(1) =0,
\]
and
$f_m(x):=\sqrt{2}\int_{0}^{1}f(x,y)\cos((m+\tfrac{1}{2})\pi y) dy$ can be efficiently computed through FFT.
Note that $f_m$ are zero except $m=6,25$.
Since $v_m$ are zero for all $m\in \NN\backslash\{6,14,25,30\}$,
our approximated solution is of the form $u_{N}=\sqrt{2}(v_{6,N}\cos(6.5\pi y) + v_{14,N}\cos(14.5\pi y) + v_{25,N}\cos(25.5\pi y) +v_{30,N}\cos(30.5\pi y))$, where $v_{m,N}$ with $m = 6,14,25,30$ are the approximated solutions to $v_{m}$ in \eqref{2Dto1D:f} using $N$ points. We use $2049$ points to discretize $\cos((m+1/2)\pi y)$ for $m=6,14,25,30$ in our approximated solutions. Also note that the following ``Local CN" and ``Link CN" record the maximum condition number of all local problems and all $m = 6,14,25,30$. See \cref{tab:2D:waveguide} for the numerical performance measured by both $\frac{\|u_{N}-u_{2N}\|_\infty}{\|u_{2N}\|_\infty}$ and
 $\frac{\|u_{N}-u_{2N}\|_2}{\|u_{2N}\|_2}$, and \cref{fig:2D:waveguide} for the convergence plot and approximated solution $u_N$. Due to the separation of variables, the convergence rates observed in the plot are solely driven by the convergence rates that take place in each 1D problem. As can be seen from \cref{tab:2D:waveguide}, the convergence rates agree with the theoretical discussion in \cref{sec:fdm,sec:convergence}.
\end{example}

{\tiny
	\begin{center}
		\begin{tabular}{c c | c c c c | c c c c}
			\hline
			\hline
			\multicolumn{2}{c}{} \vline & \multicolumn{4}{c}{DAT using the compact FDM with order $M=6$} \vline & \multicolumn{4}{c}{DAT using the compact FD with order $M=8$}\\
			\hline
			$N$ & $\ell$ & $\frac{\|u_{N}-u_{2N}\|_{\infty}}{\|u_{2N}\|_{\infty}}$  & $\frac{\|u_{N}-u_{2N}\|_{2}}{\|u_{2N}\|_{2}}$ & Local CN & Link CN  & $\frac{\|u_{N}-u_{2N}\|_{\infty}}{\|u_{2N}\|_{\infty}}$  & $\frac{\|u_{N}-u_{2N}\|_{2}}{\|u_{2N}\|_{2}}$ & Local CN & Link CN\\
			\hline			
			$3(2^{6})$ & $0$ & $1.2184\times 10^{-2}$ & $1.4965 \times 10^{-2}$ & $1.52 \times 10^{5}$ & $-$ &
			$2.5802\times 10^{-3}$ & $3.4368 \times 10^{-3}$ & $4.52 \times 10^{5}$ & $-$\\
			& $2$ & $1.2184\times 10^{-2}$ & $1.4965 \times 10^{-2}$ & $1.54 \times 10^{3}$ & $5.18 \times 10^{2}$ &
			$2.5802 \times 10^{-3}$ & $3.4368 \times 10^{-3}$ & $3.22 \times 10^{3}$ & $5.19 \times 10^{2}$\\
			\hline
			$3(2^{7})$ & $0$ & $1.9809 \times 10^{-4}$ & $2.3348 \times 10^{-4}$ & $7.48\times 10^{3}$ & $-$ &
			$1.1417 \times 10^{-5}$ & $1.5130 \times 10^{-5}$ & $7.48 \times 10^{3}$ & $-$\\
			& $3$ & $1.9809 \times 10^{-4}$ & $2.3348 \times 10^{-4}$ & $2.60\times 10^{3}$ & $5.19 \times 10^{2}$ &
			$1.1417 \times 10^{-5}$ & $1.5130 \times 10^{-5}$ & $2.60 \times 10^{3}$ & $5.19 \times 10^{2}$\\
			\hline
			$3(2^{8})$ & $0$ & $3.1521 \times 10^{-6}$ & $3.7059\times 10^{-6}$ & $2.96\times 10^{4}$ & $-$ &
			$4.6092 \times 10^{-8}$ & $6.0943 \times 10^{-8}$ & $2.96\times 10^{4}$ & $-$\\
			& $4$ & $3.1521 \times 10^{-6}$ & $3.7059\times 10^{-6}$ & $3.85\times 10^{4}$ & $8.49 \times 10^{2}$ &
			$4.6094 \times 10^{-8}$ & $6.0946\times 10^{-8}$ & $3.85\times 10^{4}$ & $8.49 \times 10^{2}$\\
			\hline
			$3(2^{9})$ & $0$ & $4.9446 \times 10^{-8}$ & $5.8172 \times 10^{-8}$ & $1.18 \times 10^{5}$ & $-$ &
			$1.8328 \times 10^{-10}$ & $2.3989 \times 10^{-10}$ & $1.18 \times 10^{5}$ & $-$\\
			& $5$ & $4.9447 \times 10^{-8}$ & $5.8172 \times 10^{-8}$ & $1.54 \times 10^{2}$ & $8.87 \times 10^{3}$ &
			$1.8344 \times 10^{-10}$ & $2.4014 \times 10^{-10}$ & $1.54 \times 10^{2}$ & $8.87 \times 10^{3}$\\	 
			\hline
		\end{tabular}	
		 \captionof{table}{Relative errors for \cref{ex:2D:waveguide} using DAT with $N_0=12$ and $s=1$ in \cref{alg:dat}. The grid increments used in each $[0,\frac{3}{10}]$, $[\frac{3}{10},\frac{7}{10}]$, and $[\frac{7}{10},1]$ are respectively $\frac{9}{10N}$, $\frac{6}{5N}$, and $\frac{9}{10N}$.}
		 \label{tab:2D:waveguide}
	\end{center}
}

\begin{figure}[htbp]
	\centering
	 \begin{subfigure}[b]{0.3\textwidth}
		\begin{tikzpicture}
			\begin{axis}
				[
				 xlabel={\footnotesize{$N\!\!=$}},
				xlabel style ={xshift=-2.3cm,yshift=0.48cm},
				 ylabel={\footnotesize{Relative errors}},
				 height=0.9\textwidth,
				 width=\textwidth,
				xmode=log,
				log basis x={2},
				ymode=log,
				log basis y={10},
				 ytick={10e-10,10e-8,10e-6,10e-4,10e-2,10e-0},
				legend style={nodes={scale=0.6, transform shape}},
				legend pos=outer north east
				]
				 \addplot[thick,mark=*,blue] coordinates
				 {(192,1.2184e-2)
					 (384,1.9809e-4)
					 (768,3.1521e-6)
					 (1536,4.9447e-8)};
				 \addplot[thick,mark=*,red] coordinates
				 {(192,1.4965e-2)
					 (384,2.3348e-4)
					 (768,3.7059e-6)
					 (1536,5.8172e-8)};
				 \addplot[thick,dashed,mark=*,blue] coordinates
				 {(192,2.5802e-3)
					 (384,1.1417e-5)
					 (768,4.6094e-8)
					 (1536,1.8344e-10)};
				 \addplot[thick,dashed,mark=*,red] coordinates
				 {(192,3.4368e-3)
					 (384,1.5130e-5)
					 (768,6.0946e-8)
					 (1536,2.4014e-10)};
			\end{axis}
			\node (B) at (1.1,0.75) {\tiny \scalebox{0.8}{$\mathcal{O}(N^{-8})$}};
			\node (A) at (1.9,1.6) {};
			\draw[thick,->] (B) edge (A);
			\node (D) at (2.80,2.75) {\tiny \scalebox{0.8}{$\mathcal{O}(N^{-6})$}};
			\node (C) at (2.05,1.90) {};
			\draw[thick,->] (D) edge (C);
		\end{tikzpicture}
	\end{subfigure}
	 \begin{subfigure}[b]{0.3\textwidth}
		 \raisebox{0.1cm}{\includegraphics[width=\textwidth]{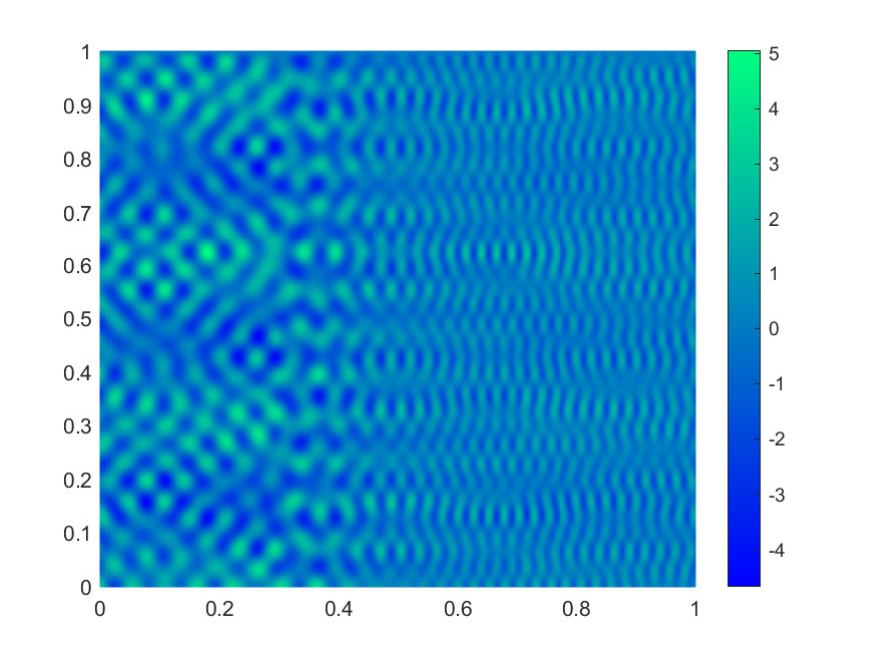}}
	\end{subfigure}
	 \begin{subfigure}[b]{0.3\textwidth}
		 \raisebox{0.1cm}{\includegraphics[width=\textwidth]{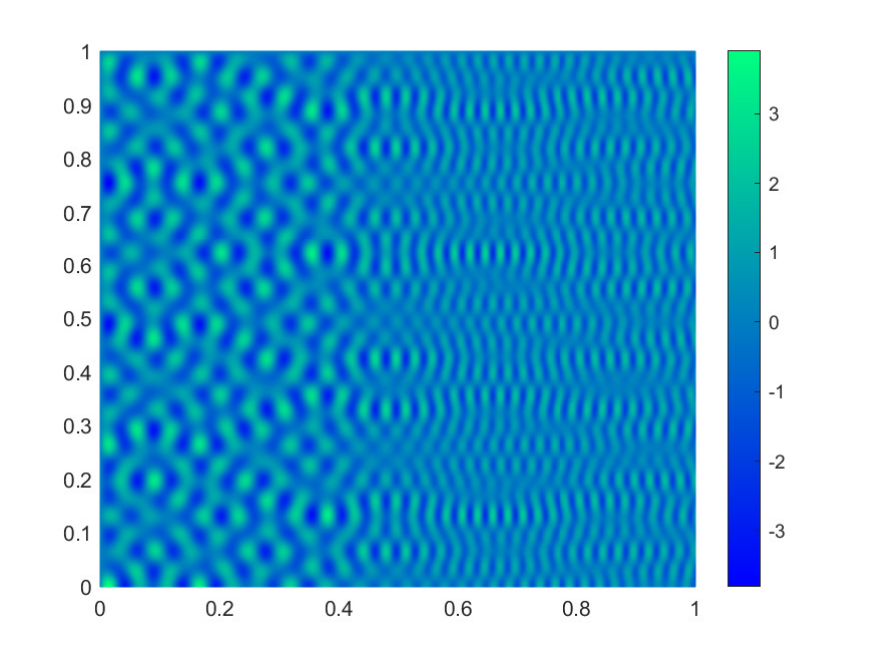}}
	\end{subfigure}
	\caption
	{\cref{ex:2D:waveguide}:
		Convergence plot (left) of DAT using the compact FDM  with order $M=6$ (solid) and $M=8$ (dashed) for errors $\frac{\|u_{N}-u_{2N}\|_{\infty}}{\|u_{2N}\|_{\infty}}$ (red) and
		 $\frac{\|u_{N}-u_{2N}\|_{2}}{\|u_{2N}\|_{2}}$ (blue). The displayed convergence rates are obtained by calculating $\log_2\left(
\frac{\|u_{N}-u_{2N}\|_{\infty} /\|u_{2N}\|_{\infty} }{ \|u_{2N}-u_{4N}\|_{\infty}/\|u_{4N}\|_{\infty}}\right)$ and $\log_2\left(\frac{\|u_{N}-u_{2N}\|_{2}/\|u_{2N}\|_{2}  }{ \|u_{2N}-u_{4N}\|_{2}/\|u_{4N}\|_{2}}\right)$. The real (middle) and imaginary (right) parts of $u_N$ with $N=3(2^9)$, $\ell=5$ and $M=8$.}
	\label{fig:2D:waveguide}
\end{figure}
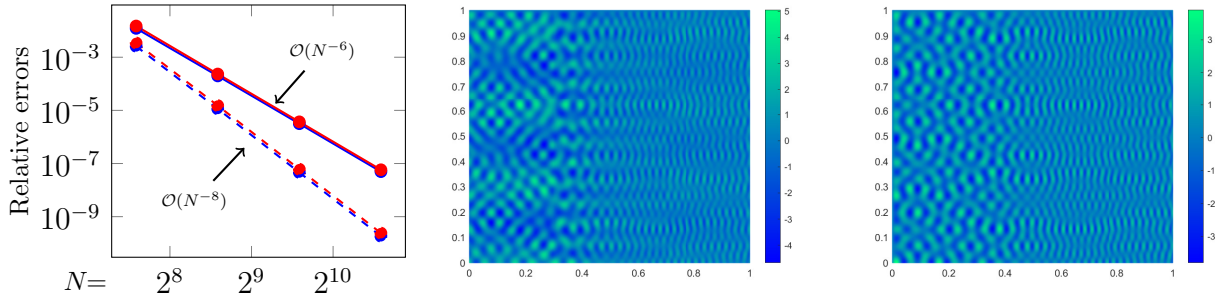

\subsection{DAT and compact FDMs using only function values}

The direct usage of derivatives of $a,\kappa^2,f$ in \cref{thm:interior,thm:bdry} may not be computationally efficient. These derivatives can in fact be estimated by using only function values of $a,\kappa^2,f$ through a local polynomial approximation. For example, if we consider an interior stencil of the form \eqref{cbo:special} and \eqref{dcoeff:special},
we know from \eqref{prop:span} that all stencil coefficients depend on $a(\xb),a'(\xb),\ldots,a^{(M-1)}(\xb)$, $\kappa^2(\xb),[\kappa^2]'(\xb),\ldots,[\kappa^2]^{(M-2)}(\xb)$, and $f(\xb),f'(\xb),\ldots,f^{(M-2)}(\xb)$. Consider $a(\xb),a'(\xb),\ldots,a^{(M-1)}(\xb)$. Let $J\ge M$ and
take $J$ points $\{x_j\}_{j=1}^J$ near the base point $\xb$ such that all the points fall into one piece of the piecewise smooth functions $a,\kappa^2$ and $f$.
Find the unique polynomial $p$ of degree $J-1$ satisfying $p(x_j)=a(x_j)$ for all $j=1,\ldots,J$.
Then $a^{(n)}(\xb)\approx p^{(n)}(\xb)$ for $n=0,\ldots,M-1$. We often use $J=M$ and $\xb\in \{x_j\}_{j=1}^J$ such that $\{x_j\}_{j=1}^J$ is evenly spaced with mesh size $h/2$.

Using only function values, we re-calculate numerical experiments in Examples~\ref{ex:constant}--\ref{ex:2D:waveguide}, which yield virtually same results as those using derivatives explicitly. It demonstrates the convenience of using a local polynomial approximation in lieu of true derivatives, which may have complicated expressions. For the sake of conciseness, we only provide re-calculated \cref{ex:piece4het,ex:newphet}.

{\tiny
	\begin{center}
		 \resizebox{\textwidth}{!}{\begin{tabular}{c c | c c c c | c c | c c c c}
			\hline
			\hline
				 \multicolumn{2}{c}{} \vline & \multicolumn{4}{c}{DAT using the compact FDM with order $M=6$} \vline & \multicolumn{2}{c}{} \vline & \multicolumn{4}{c}{DAT using the compact FDM with order $M=6$}\\
				 \multicolumn{2}{c}{} \vline & \multicolumn{4}{c}{for \cref{ex:piece4het}} \vline & \multicolumn{2}{c}{} \vline & \multicolumn{4}{c}{for \cref{ex:newphet}}\\
			\hline
			$N$ & $(\ell,s)$ & $\frac{\|u_{N}-u_{2N}\|_{\infty}}{\|u_{2N}\|_{\infty}}$ & $\frac{\|u'_{N}-u'_{2N}\|_{\infty}}{\|u'_{2N}\|_{\infty}}$ & Local CN & Link CN & $N$ & $\ell$ & $\frac{\|u_N-u_{2N-1}\|_{\infty}}{\|u_{2N-1}\|_{\infty}}$ & $\frac{\|u'_{N}-u'_{2N-1}\|_{\infty}}{\|u'_{2N-1}\|_{\infty}}$ & Local CN & Link CN \\
			\hline
			$2^{15}$ & $(0,0)$ & $4.7473 \times 10^{-2}$ & $6.7611 \times 10^{-2}$ & $5.19 \times 10^{6}$ & $-$ &
			$2^{10}+1$ & $0$ & $2.3932$ & $1.8969$ & $1.25 \times 10^{6}$ & $-$ \\ 
			& $(5,1)$ & $4.7473\times 10^{-2}$ & $6.7611 \times 10^{-2}$ & $2.07 \times 10^{5}$ & $8.41 \times 10^{3}$ &
			& $5$ & $2.3932$ & $3.3990$ & $2.66 \times 10^{4}$ & $3.44 \times 10^{3}$ \\ 
			& $(3,2)$ & $4.7473\times 10^{-2}$ & $6.7611 \times 10^{-2}$ & $2.07 \times 10^{5}$ & $2.90 \times 10^{4}$ & & & & & & \\
			\hline
			$2^{16}$ & $(0,0)$ & $7.2618 \times 10^{-4}$ & $1.0353 \times 10^{-3}$ & $2.22 \times 10^{7}$ & $-$ &
			$2^{11}+1$ & $0$ & $6.9794 \times 10^{-3}$ & $6.4621 \times 10^{-3}$ & $3.13 \times 10^{4}$ & $-$ \\ 
			& $(5,1)$ & $7.2618 \times 10^{-4}$ & $1.0353 \times 10^{-3}$ &  $8.29 \times 10^{5}$ & $8.41 \times 10^{3}$ &
			& $5$ & $6.9794 \times 10^{-3}$ & $6.4621 \times 10^{-3}$ & $9.00 \times 10^{3}$ & $6.92 \times 10^{3}$ \\ 
			& $(3,2)$ & $7.2618 \times 10^{-4}$ & $1.0353 \times 10^{-3}$ &  $8.29 \times 10^{5}$ & $2.92 \times 10^{4}$  & & & & & & \\
			\hline
			$2^{17}$ & $(0,0)$ & $1.1182\times 10^{-5}$ & $1.5959 \times 10^{-5}$ & $8.89 \times 10^{7}$ & $-$ &
			$2^{12}+1$ & $0$ & $1.0216 \times 10^{-4}$ & $9.0880 \times 10^{-5}$ & $9.03 \times 10^{4}$ & $-$ \\ 
			& $(5,1)$ & $1.1188\times 10^{-5}$ & $1.5967 \times 10^{-5}$ &  $3.32 \times 10^{6}$ & $8.41 \times 10^{3}$ &
			& $5$ & $1.0216 \times 10^{-4}$ & $9.0880 \times 10^{-5}$ & $3.66 \times 10^{4}$ & $2.53 \times 10^{3}$\\ 
			& $(3,2)$ & $1.1179\times 10^{-5}$ & $1.5955 \times 10^{-5}$ & $3.32 \times 10^{6}$ & $2.92 \times 10^{4}$ & & & & & &\\
			\hline
			$2^{18}$ & $(0,0)$ & $1.7274\times 10^{-7}$ & $2.4663 \times 10^{-7}$ & $3.56\times 10^{8}$ & $-$ &
			$2^{13}+1$ & $0$ & $1.5468 \times 10^{-6}$ & $1.4036 \times 10^{-6}$ & $3.60 \times 10^{5}$ & $-$\\ 
			& $(5,1)$ & $1.6132\times 10^{-7}$ & $2.3056 \times 10^{-7}$ & $1.33 \times 10^{7}$ & $8.41 \times 10^{3}$ &
			& $5$ & $1.5468 \times 10^{-6}$ & $1.4036 \times 10^{-6}$ & $1.47 \times 10^{5}$ & $2.51 \times 10^{3}$ \\ 
			& $(3,2)$ & $1.7914\times 10^{-7}$ & $2.5548 \times 10^{-7}$ & $1.33 \times 10^{7}$ & $2.92 \times 10^{4}$  & & & & & & \\
			\hline
		\end{tabular}}
		 \captionof{table}{Relative errors for \cref{ex:piece4het,ex:newphet} using only point values (without explicitly computing derivatives) in DAT with the compact FDM with order $M=6$.}
		 \label{tab:piece4hetptval}
	\end{center}
}

\section{Conclusions}
\label{sec:conclusions}

In this paper, we have presented a new method called DAT (Dirac Assisted Tree), which is capable of handling 1D heterogeneous Helmholtz equation
and special multidimensional Helmholtz equations that can be decomposed into a series of 1D problems.
One of DAT's strengths is its ability to break a global problem into many parallel local problems with a tree structure, which are then assembled by solving small linking problems. In the extreme case, these local and linking problems are at most $4 \times 4$ in size that can be solved in a parallel fashion.
Another strength of DAT lies in that DAT can handle 1D heterogeneous Helmholtz equations with arbitrarily large variable wave numbers and oscillatory jumping coefficients having large variations.
To solve local problems arising from DAT, we propose an arbitrarily accurate compact FDM.
Finally, we have also applied our DAT algorithm coupled with the proposed FDM to solve several 1D heterogeneous Helmholtz equations and several special 2D Helmholtz equations to showcase its efficacy. The convergence behaves in accordance with the theory discussed in \cref{sec:fdm,sec:convergence} and the coefficient matrices arising from DAT are, as expected, better conditioned than FDM. In particular, for any given positive integer $N_0$ independent of the total number of freedoms $N$ and the mesh size $h$ with $h=1/N$, if all local problems in DAT employing the developed FDMs are at most $N_0\times N_0$ in size, then all the condition numbers of all local problems in DAT must be uniformly bounded.

\end{document}